\documentclass[reqno,a4paper,11pt]{amsart}
 \usepackage{amssymb, mathtools, enumitem, xcolor, todonotes}
\usepackage[unicode,hidelinks]{hyperref}
 
\allowdisplaybreaks 
\renewcommand{\d}{{\rm d}}

 \newcommand{\eps}{\varepsilon}

 \newcommand{\Var}{{\rm Var}}
\newcommand{\Cov}{{\rm Cov}}

 \newcommand{\lip}{{\rm Lip}}

  \newcommand{\diag}{{\rm diag}}

 \newcommand{\R}{\mathbb{R}}
\newcommand{\N}{\mathbb{N}}
 \newcommand{\Z}{\mathbb{Z}}

 \newcommand{\E}{\mathbb{E}}
  
 \renewcommand{\P}{\mathbb{P}}
 \newcommand{\Cstern}{C^*}

\newcommand{\abs}[1]{\left\lvert #1 \right\rvert}

 \def\1{{\mathchoice {1\mskip-4mu\mathrm l} 
{1\mskip-4mu\mathrm l}
{1\mskip-4.5mu\mathrm l} {1\mskip-5mu\mathrm l}}}

\DeclareMathOperator{\Hess}{Hess}
\DeclareMathOperator{\sech}{sech}
\DeclareMathOperator{\D}{D}
\DeclareMathOperator{\artanh}{artanh}

\newtheorem{theorem}{Theorem}[section]
\newtheorem{lemma}[theorem]{Lemma}
\newtheorem{cor}[theorem]{Corollary}
\newtheorem{prop}[theorem]{Proposition}

\theoremstyle{remark}
\newtheorem{remark}[theorem]{Remark} 

\numberwithin{equation}{section}
 
\title{The infinite Block Spin Ising Model}

\author{Jonas Jalowy}
   \address{\footnotesize Jonas Jalowy: {\itshape Paderborn University, Warburger Str.~100, 33098 Paderborn, Germany}}
\email{jjalowy@math.upb.de}

\author{Isabel Lammers}
\address{\footnotesize Isabel Lammers: {\itshape Universit\"at M\"unster, Einsteinstrasse 62, M\"unster 48149, Germany}}
\email{isabel.lammers@uni-muenster.de}

\author{Matthias L\"owe}
\address{\footnotesize Matthias L\"owe: {\itshape Universit\"at M\"unster, Einsteinstrasse 62, M\"unster 48149, Germany}}
\email{matthias.loewe@uni-muenster.de}


\begin{document}
 \begin{abstract}
We study a block mean-field Ising model with $N$ spins split into $s_N$ blocks, with Curie–Weiss interaction within blocks and nearest-neighbor coupling between blocks. While previous models deal with the block magnetization for a fixed number of blocks, we study the the simultaneous limit $N\to\infty$ and $s_N\to\infty$. The model interpolates between Curie–Weiss model for $s_N=1$, multi-species mean field for fixed $s_N=s$, and the 1D Ising model for each spin in its own block at $s_N=N$.

Under mild growth conditions on $s_N$, we prove a law of large numbers and a multivariate CLT with covariance given by the lattice Green's function. For instance, the high temperature CLT essentially covers the optimal range up to $s_N=o(N/(\log N)^c)$ and the low temperature regime is new even for fixed number of blocks $s>2$. In addition to the standard competition between entropy and energy, a new obstacle in the proofs is a curse of dimensionality as $s_N \to \infty$.
\end{abstract}
\maketitle
\section{Introduction}
On the one hand, block spin models are fascinating from a statistical mechanics point of view, see e.g., 
\cite{GC08, FC11, FU12, Col14,LSblock1,KLSS20,JLS} as models of interacting mean-field systems. On the other hand, they were
rediscovered as models for social interactions between groups, see e.g.\ \cite{GBC09, KT20a, LSV20, Fleermann+Kirsch+Toth:2022}, thereby following a re-interpretation of mean-field models in a social choice context, see \cite{BrockDurlauf, KnLo07, CoLo10}, and recently have posed interesting statistical problems, see \cite{BRS19, LS20}.

The standard setup of block models is that $N$ spins are divided into $s$ blocks, and eventually we shall be interested in the thermodynamic limit $N \to \infty$. In general, the spins may take values in some set, but here we shall restrict ourselves to Ising spins $\sigma_i\in \{-1,+1\}$. All spins interact in a mean-field fashion with other spins in the same block with a certain strength, but the mean field interaction of spins in different blocks has a different strength. This can be encoded in an 
$s \times s$ matrix $A=(A_{i,j})_{i,j=1}^s$, where $A_{i,j}$
then describes the interaction strength between any two
particles in blocks $i$ and $j$. Many of the above models have two blocks, while the most general setup is chosen in \cite{JLS},  where the number of blocks is arbitrary (but fixed), the spins
may take a finite number of possible values (Potts model), and $A$ is an
arbitrary positive definite interaction matrix. 

In this note, we are interested in a new extension of the model, which allows the number of blocks $s=s_N$ also to increase as $N \to \infty$ and where the block interaction $A_{i,j}$ resembles an 1D nearest neighbor Ising model with periodic boundary, i.e.~$A_{i,j}=0$ if $|i-j|\not\in\{ 0,1,s_N-1\}$. In particular, this spin system serves as an interpolation between the Curie Weiss (mean field) model for $s_N=1$, existing block models from the literature for $s_N=2$ or other fixed $s_N=s\in\N$ and the Ising model for the extreme case $s_N=N$. One goal of this article is to close the gap between the study of finite $s$-dimensional magnetization vectors of earlier block models and the Ising model, where our magnetization vector would consist of the single spins. 

The model is close in spirit to so-called Kac-Ising models, see e.g.~\cite[\S 4.2]{Presutti}, \cite[Appendix C]{Thompson}, where each spin $\sigma_i$ may interact with a set of neighboring spins that crucially depends on the position $i$ of the spin $\sigma_i$. In contrast, this set is a fixed predetermined block in block spin systems. Moreover, in the classical Lebowitz–Penrose limit of Kac models, one first takes the thermodynamic limit $N\to\infty$, and then sends the interaction range to infinity.
In our model, these limits are taken simultaneously instead. This will result in a different behavior of our block models compared to Kac-Ising models. In particular, with the right choice of parameters, we will see phase transitions even in dimension one, while this is not the case for 1d-Kac-Ising models.

We will analyze this phase transition by
studying the vectors of block magnetization.
This block averaging procedure goes by the name of \emph{coarse graining} and is a standard analytical tool for approximations, e.g.~for the Kac-Ising model. However, in the situation studied in our note, the block magnetization vector $m$ forms an order parameter/sufficient statistic of the Hamiltonian, since the underlying microscopic block spin system is a priori structured into block variables. This allows us to derive explicit laws of large numbers (LLN) and central limit theorems (CLT) for the vector of block magnetizations, which are free from approximations and bridge the finite–block mean–field regime and the fully microscopic Ising chain.
For classical mean-field spin models (i.e.\ for models without an additional block structure), e.g.\ mean-field Ising models, such LLNs and CLTs for the 
magnetization have a long tradition. First results go back to Ellis and Newman
\cite{Ellis_Newman_78a, Ellis_Newman_78b, EllisNewman_80}. By now, extensions cover various directions such as explicit error bounds by Steins' method in \cite{Chatterjee_Shao} and \cite{EL10} and $p$-spin versions in \cite{MSB21,MSB25}.

\section{Overview}

\subsection{The Setting}
We consider a spin system of $N\in\N$ many Ising spins $\sigma_i \in \{-1, +1\}$. The index set $\{1, \ldots, N\}$ is partitioned into $s_N$ blocks that we denote by $S_1, \ldots, S_{s_N}$ of equal size $\frac{N}{s_N}$, where $s_N$ is a non-decreasing sequence of integers. We define the Hamiltonian of our model to be
$$
H_N(\sigma) = \frac{\beta}{2} \frac{s_N}{N} \sum_{k=1}^{s_N} \sum_{i,j \in S_k}\sigma_i\sigma_j + \frac{\alpha}{2} \frac{s_N}{N} \sum_{k=1}^{s_N} \sum_{\substack{i \in S_k \\ j \in S_{k-1}\cup S_{k+1}}}\sigma_i\sigma_j, \quad \beta>2\alpha>0,
$$
where here and throughout the entire paper we will identify $1 = s_N+1$ and $s_N = 1-1$ respectively. The associated Gibbs measure is given by
$$
\mu_{N, \beta , \alpha}(\d\sigma) = \frac{1}{Z_{N,\beta, \alpha}}\exp\left\{ H_N(\sigma) \right\}\rho(\d\sigma),
$$
where the reference measure $\rho=(\frac 1 2 \delta_{-1}+\frac 1 2 \delta_{+1})^{\otimes N}$ is the product measure of $N$ symmetric $\{+1,-1\}$-valued random variables and $Z_{N,\beta, \alpha} = \E[\exp\left\{H_N(\sigma)\right\}]$ denotes the partition function.
In other words, every spin interacts with every spin from its own block at strength $\beta >0$ and with all the spins from its two neighboring blocks at another strength $\alpha>0$. As usual, $\alpha,\beta$ can be seen as individual inverse temperatures, whereas the total inverse temperature of the whole system is given by $\beta+2\alpha$.

We are interested in the magnetization $m = (m_1, \ldots, m_{s_N})$ of the model, where we define the magnetization of the $k$'th block is
$$
m_k = \frac{s_N}{N} \sum_{ i \in S_k} \sigma_i,\qquad k=1,\dots,s_N.
$$
This is an order parameter (or, sufficient statistic), i.e.~the Hamiltonian can be rewritten in terms of $m$ via
$$
H_N(m)= \frac{1}{2} \frac{N}{s_N} m^TAm,
$$
where $m^T$ is the transpose of $m \in \R^{s_N}$ and the interaction matrix $A$ is the symmetric circulant matrix
$$
A =  \left[ {\begin{array}{cccccc}
    \beta & \alpha & 0& \cdots &0 & \alpha\\
    \alpha & \beta & \alpha & 0 &\cdots & 0\\
    0 & \alpha & \beta & \alpha & \ddots & \vdots \\
    \vdots & 0 & \ddots & \ddots & \ddots & 0\\
    0 & \vdots & \ddots & \alpha & \beta & \alpha \\
    \alpha & 0 & \cdots & 0 & \alpha & \beta\\
  \end{array} } \right] \in \R^{s_N\times s_N}.
$$
Consistently, $A =  \left[ {\begin{smallmatrix}
    \beta & 2\alpha\\
    2\alpha & \beta\\
  \end{smallmatrix} } \right]$ in the two block case $s_N=2$. The eigenvalues of $A$ are given by
$$
\lambda_j = \beta + 2 \alpha \cos \left( 2\pi \frac{j}{s_N}\right), \quad j = 1, \ldots ,s_N,
$$
implying that $A$ is positive definite if (and only if) $\beta >2\alpha$. 

\subsection{The main results}
We investigate the model for the existence of a phase transition for different choices of numbers of blocks $s_N$. 
When the number of spins per block $\frac{N}{s_N}$ goes to infinity, we prove uniform laws of large numbers in Theorem \ref{WLLN} and finite dimensional Central Limit Theorems in Theorem \ref{CLT} in different temperature regimes, showing the existence of a phase transition in these cases. Define $ m^*=m^*(\beta + 2\alpha)$ to be the largest solution to $$x=\tanh((\beta+2\alpha)x)$$
and $\overrightarrow m^*=(m^*,\dots)$ is the vector that carries the value $m^*$ in every coordinate.

\begin{theorem}[Uniform Law of Large Numbers]\label{WLLN}
Assume that $s_N = o\left(\frac{N}{\log N} \right)$.
\begin{enumerate} 
    \item If $ \beta + 2\alpha \le 1$ (high temperature phase), then for every $\eps >0$,
    $$
    \mu_{N, \beta, \alpha} \left( \sup_{k=1,\dots,s_N} |m_k| >\eps \right) \overset{N\to\infty}{\longrightarrow}0
    $$ 
    \item If $\beta + 2\alpha > 1$ (low temperature phase), then for every $\eps > 0$,
    $$
    \mu_{N, \beta, \alpha}\left(  \min\left\{ \sup_{k=1,\dots,s_N} |m_k - m^*|, \sup_{k=1,\dots,s_N} |m_k + m^*| \right\}>\eps \right) \overset{N\to \infty}{\longrightarrow}0.
    $$
In particular, this means that $\mu_{N,\beta, \alpha} \circ m^{-1} \Rightarrow \frac{1}{2}\delta_{\overrightarrow m^*} + \frac{1}{2}\delta_{-\overrightarrow m^*}$ since the measures $\mu_{N,\beta,\alpha}$ are invariant under a global spin flip.
\end{enumerate}
\end{theorem}

On the other hand, if $s_N\asymp N$, then each block contains only finitely many spins. 
In this case, the model is a finite range spin system, behaving rather like an 1D Ising model with no phase transition. 

\begin{prop}\label{Ising}
    If $N/s_N\in\N$ is of fixed equal size, then, for any choice of parameters $\alpha>0$ and $\beta>0$ there is no phase transition. In particular, there is a unique infinite volume measure $\mu_m$ on $\R^\Z$ such that $$\mu_{N,\beta,\alpha}(m\in \cdot)\Rightarrow\mu_m(\cdot)$$
    (in the product topology).
            As a consequence, the total magnetization $\frac{1}{N}\sum_{i=1}^{N} \sigma_i$ converges to $0$.
\end{prop}
This result will be proved in the appendix.

\medskip
Our second main result identifies the scale of Gaussian fluctuations of the block magnetization at non-criticality with covariance being the lattice Green's function. 
\begin{theorem}[CLT in finite dimensional distribution]\label{CLT}

    \begin{enumerate}
        \item If $s_N = o\left( {N}{(\log N)^{-{5}/{2}}} \right)$ and $ \beta + 2\alpha < 1$, then the rescaled magnetization satisfies the Central Limit Theorem 
        $$  \mu_{N, \beta, \alpha} \left( \sqrt{\frac{N}{s_N}}m\in  \cdot\right)\Rightarrow \mathcal N(0,\Sigma)$$ 
      	as $N\to\infty$ in the sense of finite dimensional distributions. If $s_N\equiv s$ is constant for sufficiently large $N$, then the limiting Gaussian distribution on $\R^s$ has covariance $\Sigma=(I-A)^{-1}$, otherwise it lives on $\R^\N$ with $\Sigma$ being the limit in the sense of finite projections, that is
        $$
        \Sigma_{i,j} =\lim_{N\to\infty}(I-A)^{-1}_{i,j}
        = \frac{\big(\frac{(1-\beta)-\sqrt{(1-\beta)^2-4\alpha^2}}{2\alpha}\big)^{|i-j|}}{\sqrt{(1-\beta)^2-4\alpha^2}}, \quad i,j\in\N.
        $$
        \item If $s_N = o\left({\sqrt{N}}({\log N})^{-1} \right)$ and $ \beta + 2\alpha > 1$, then for $0 <\delta < 2m^* $, the rescaled magnetization satisfies the Central Limit Theorems
        $$
         \mu_{N, \beta, \alpha} \left( \sqrt{\frac{N}{s_N}}(m- \overrightarrow m^*)\in  \cdot \, \middle\vert \, m \in  B_{\delta \sqrt{s_N}}(\overrightarrow m^*)\right)\Rightarrow \mathcal N(0,\Sigma^*)
        $$
        as $N \to \infty$ and
        $$
        \mu_{N, \beta, \alpha} \left( \sqrt{\frac{N}{s_N}}(m+ \overrightarrow m^*)\in  \cdot \, \middle\vert \, m \in  B_{\delta \sqrt{s_N}}(-\overrightarrow m^*)\right)\Rightarrow \mathcal N(0,\Sigma^*),
        $$
        again in the sense of finite dimensional distributions and where (see \eqref{eq:Sigma*} below)
        $$
        \Sigma^*_{i,j}=\lim_{N\to\infty}(1-(m^*)^2) (I-A(1-(m^*)^2))^{-1}_{i,j}.
        $$ 
    \end{enumerate}
\end{theorem}

\begin{remark}
Naturally, the CLT implies a weak LLN, but only for finite dimensional marginals. The uniform LLN in Theorem \ref{WLLN} however holds for the supremum norm of $m$.
\end{remark} 

\subsection{Discussion}
	Let us comment on how one may view the model and results like Theorem \ref{CLT} as an interpolation between related results in the literature.
		\begin{itemize}
			\item $s_N=1$ corresponds to the Curie Weiss Model with a well known CLT of the magnetization $\sqrt N m\Rightarrow\mathcal N(0,(1-\beta)^{-1}) $ on $\R$, see \cite{Ellis_Newman_78a}.
			\item $s_N=2$ corresponds to the bipartite model introduced in \cite{GC08} or the two group Curie-Weiss model of \cite{KT20a}, see \cite{LSblock1} for the CLT for the rescaled magnetization vector $\sqrt N(m_1,m_2)$ on $\R^2$. We refer to \cite{FU12} for the limiting pressure per particle and see \cite{Col14} for the study of the dynamics. In the spirit of \cite{BrockDurlauf, KnLo07} and \cite{CoLo10}, \cite{GBC09} introduces and interprets the two block model as a social choice model. In \cite{ LSV20}, the authors combine a block spin model with a stochastic block model and study it in the context of social choice. 
			\item fixed $s_N\in\N$ corresponds to a finite dimensional Gaussian Limit, see \cite{KLSS20} and see \cite{FC11} for a CLT for high and low temperature (without giving the exact limit points of the magnetization vector). This case also corresponds to the multigroup Curie-Weiss model from \cite{KT22} where the authors prove LLN and CLT and is generalized in \cite{Fleermann+Kirsch+Toth:2022} to local limit theorems. However, note that in the low temperature case, the authors in \cite{KT22} need that $s_N=2$ or that the interaction matrix has identical entries. Extensions to the Block Potts model have been studied in \cite{kim2potts,munpotts,JLS}.
			\item $s_N\to\infty$ yields a system (or, stochastic process) on $\R^\N$. If $N/s_N\to\infty$, then the blocks renormalize the discrete model into a continuous Gaussian spin model, cf. Theorem \ref{CLT}.
			\item For fixed $N/s_N\in\N$ spins per block, the discrete finite range spin model $\mu_{N,\beta,\alpha}$ gives rise to a nearest neighbor model $\mu_{N,\beta,\alpha}(m\in\cdot)$ with a unique limit, see Proposition \ref{Ising}. 
			\item Finally, $s_N=N$  corresponds to the nearest neighbor 1D Ising model at inverse temperature $\alpha$ with spin configuration $m=(m_k)_{k\in\N}=(\sigma_k)_{k\in\N}$, see \cite{Ising25} or \cite{Velenik_book}.
		\end{itemize}
	Note that for $N/s_N \to \infty$, i.e.\ when there is an increasing number of spins per block, our model falls into the category of models analyzed in \cite{basak_muk} and \cite{deb_mukherjee}. There, the authors show universality results for a broad class of spin models on graphs (in \cite{basak_muk} these spin models can be of Ising or Potts type with ferromagnetic and anti-ferromagnetic interaction). They show that if the average degree in such models grows quickly enough, then these models exhibit universal behavior. This concerns, among others, the \textit{global} magnetization. Since our model is much more specific, we may analyze (with the exception of Proposition \ref{Ising}) the vector of block magnetizations, a quantity that would not make sense in the general setup of \cite{basak_muk} and \cite{deb_mukherjee}. Hence, though our model falls into their general framework, our specific choice opens the field for a much more granular analysis, such as the description of correlations between blocks.

    In the extreme case $\alpha=0$ we have independent Curie Weiss Models and the limit distribution on the configuration space becomes a trivial product measure, e.g.~Theorem \ref{CLT}\,(i) remains true for all $s_N=o(N)$, since $\sqrt{N/s_N}m\Rightarrow \mathcal N(0,(1-\beta)^{-1})^{\otimes \N}$. The CLT for the magnetization hold for any choice of $s_N$, which is the setting of classical CLTs for sums $ \sqrt{N}\sum_{k=1}^{s_N}m_k$ of independent random variables $m_k$ whose variance depends on the choice of $\beta,s_N$.

   For Theorem 2.3\,(ii), the assumption $\alpha>0$ is necessary as independent blocks may have any combination of signs of $\pm m^*$ without interaction. On the other hand, the assumption $s_N = o({\sqrt{N}}({\log N})^{-1} )$ is technical only and could be weakened. Yet, our stronger assumption still conveys the same message of the result and in order not to unnecessarily overcomplicate the proof, we confine ourselves with a simpler step in Lemma \ref{lem:minimal} and after.

 \begin{figure}[h]
		\centering
        \includegraphics[width=\linewidth,trim={0 0 0 110},clip]{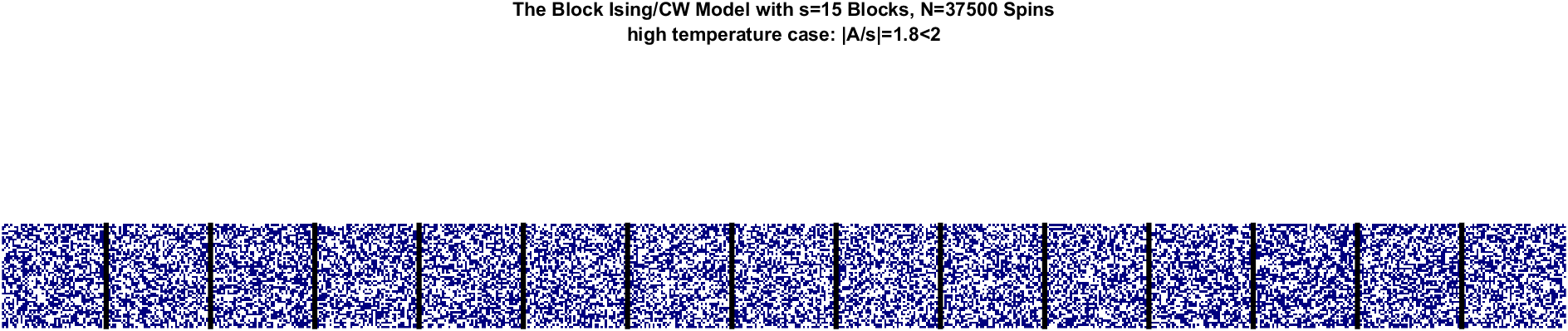}\\
		\includegraphics[width=\linewidth,trim={0 0 0 110},clip]{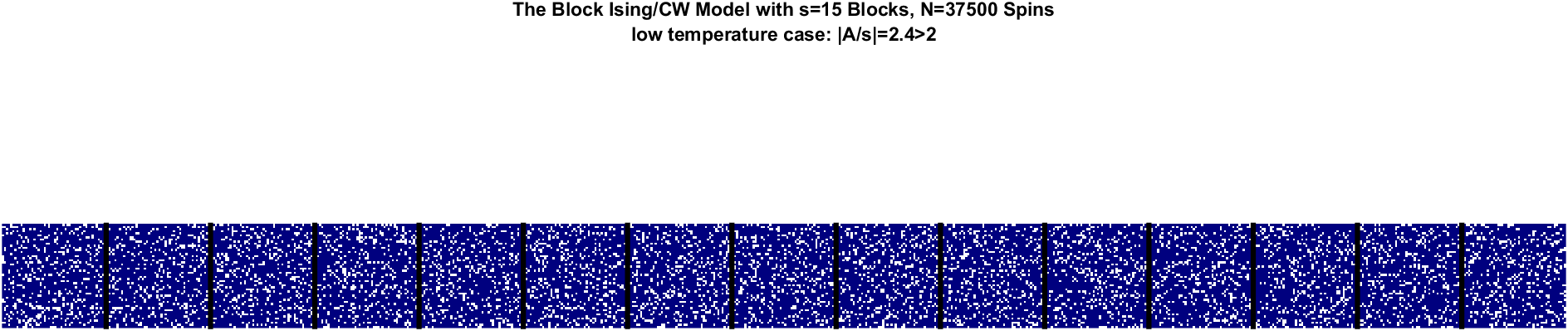}\\
        \includegraphics[width=\linewidth,trim={0 0 0 110},clip]{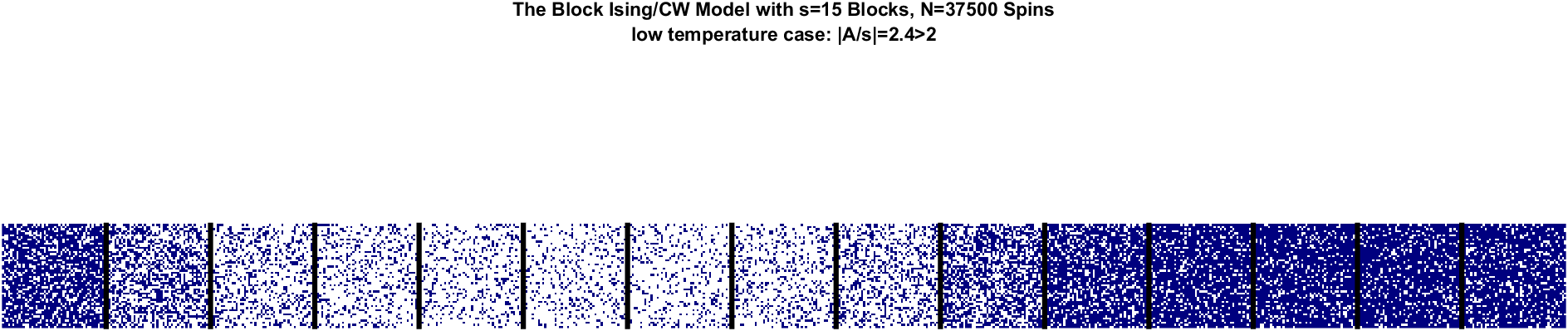}
		\caption{Three different simulations of the block spin Curie Weiss model for $s_N=15$ blocks with $N/15=2500$ spins in each block (each of which itself is a mean field model). The high temperature regime in the top row ($\beta=0.5, \alpha=0.2$) shows zero magnetization, whereas the low temperature regime in the center does magnetize ($\beta=0.8,\alpha=0.2$). The bottom row shows an unstable (improbable) realization in the low temperature regime, which visualizes the correlation of (cyclically) neighboring block interaction.}
		\label{fig:gibt_doch_nur_eine}
	\end{figure}

\subsection{Structure of the paper} In the remainder of the paper, we prove our main results. We begin with the \emph{high temperature regime} ($\beta+2\alpha<1$) both for clarity and because the arguments are cleaner there. In all proofs we will encounter a curse of dimensionality from $s_N\to\infty$. The general difficulty of the proofs is to localize the contribution of the magnetization vector to lower (logarithmic) dimensions, for which we make use of the fast correlation decay between far away blocks of our model.

In Section 3 we prove the law of large numbers Theorem \ref{WLLN}\,(i), via a decoupling of the block interactions. Then, the competition between entropy and energy can be controlled, so we can prove the result with a counting argument of the relevant configurations and the Laplace principle.

In order to apply Laplace method for the multivariate Central Limit Theorem, we introduce the foundations of the Hubbard--Stratonovich transform in Section 4. The proof of the high temperature CLT Theorem \ref{CLT}\,(i) is presented in section 5, where the growing dimension will cause significant problems for Laplace method. Note in particular that a Gaussian approximation in high dimension $s_N\to\infty$ will concentrate on shells at $\approx \sqrt {s_N}$, so volume effects compete with typical bounds. We control such effects with sub-Gaussian concentration inequalities and the decay of the off diagonal entries of $(I-A)^{-1}$ (more on the idea of the proof can be found in Section 5).

The \emph{low temperature regime} ($\beta+2\alpha>1$) of Theorems \ref{WLLN}\,(ii) and \ref{CLT}\,(ii) is treated in Sections 6 and 7. We expand around the two uniform symmetry-broken states and follow similar steps of previous proofs, only being technically more involved. The Appendix contains the standard proof of the finite range model of Proposition \ref{Ising}.

\section{Proof of Theorem \ref{WLLN}\,(i)}
First, we will prove the law of large numbers in the high temperature regime. Assume for the whole section that $\beta + 2\alpha \le 1$. The first step is to show that under the Gibbs measure the magnetization of the first block converges to zero in probability as $N \to \infty$. Then the statement of the theorem follows by a union bound.

\begin{prop}\label{m1}
    Assume that $s_N = o\left( \frac{N}{\log N} \right)$ and let $\eps>0$. If $\beta + 2\alpha <1$, then there exists a constant $c_{\beta,\alpha}>0$ such that
    \begin{equation}
        \mu_{N,\beta,\alpha}\left(|m_1|>\eps\right) =  \mathcal{O}\left( \exp\left\{- \frac{N}{s_N} c_{\beta,\alpha} \eps^2\right\} \right).
    \end{equation}
    If $\beta + 2\alpha = 1$, then
    \begin{equation}
        \mu_{N,\beta,\alpha}\left(|m_1|>\eps\right) =  \mathcal{O}\left( \exp\left\{- \frac{N}{s_N} \frac{1}{48} \eps^4\right\} \right).
    \end{equation}
\end{prop}

 The idea of the proof is to approximate the spin system by independent Curie-Weiss models on the individual blocks. Because of increasing dimension, the error of decoupling all of the blocks by using a Cauchy-Schwarz bound 
  \begin{equation}\label{CauchySchwarz}
 m_km_{k+1}\le \frac{m_k^2}{2}+ \frac{m_{k+1}^2}{2}
 \end{equation}
  on the interaction terms is too large compared to the speed of convergence. Therefore, we first need to reduce the number of relevant blocks. This can be achieved by only applying the bound until we find some index $k$ corresponding to a block of sufficiently small magnetization. Blocks with a small magnetization also have small interactions allowing us to simply forget about it and in this way decoupling those blocks from everything that comes beyond while making a significantly reduced error. To apply this strategy, we need to ensure that there will actually appear such blocks that carry a small magnetization. Indeed, it is not hard to see that the probability that all block magnetizations are uniformly bounded from below by a constant goes to zero exponentially fast (in the number of spins $N$). More precisely, we will first prove the following lemma.
 
\begin{lemma}
Let $ 0<  \delta <1$. For $\beta + 2\alpha < 1$ and $s_N = o(N)$
\begin{equation}\label{boundHigh}
    \mu_{N,\beta,\alpha}\left( |m_1|, \ldots, |m_{s_N}| >  \delta \right)  \le \exp\left\{ - \left(\frac{1-(\beta + 2\alpha)}{2} - o(1) \right)N \delta^2 \right\} =o(1).
\end{equation}
Similarly, at the critical temperature $\beta + 2\alpha =1$, we obtain the upper bound
\begin{equation}\label{boundCritical}
\mu_{N,\beta,\alpha}\left( |m_1|, \ldots, |m_{s_N}| >  \delta \right) \le \exp\left\{ -\left(\frac{1}{12} - o(1) \right)N \delta^4 \right\}.
\end{equation}
   
\end{lemma}

\begin{proof} First note that rewriting the partition function in terms of the magnetization yields
$$
    Z_{N,\beta, \alpha} = \sum_{m_1, \ldots, m_{s_N}} \exp \left\{\frac{1}{2}\frac{N}{s_N}\sum_{k=1}^{s_N} \beta m_k^2 + 2 \alpha m_k m_{k+1} \right\} \frac{1}{2^N} \prod_{k=1}^{s_N} { \frac{N}{s_N} \choose \frac{N}{s_N}\frac{1+m_k}{2}},
$$
where the $m_k$ take values in the set $\mathcal{A}_N = \left\{ -1+ 2k\frac{s_N}{N} \, : \, k=0,\ldots, \frac{N}{s_N}\right\}$. Then by Jensens' inequality, we obtain that
$$
Z_{N,\beta,\alpha} = \E[ \, \exp\{ H_N(m) \}\, ] \ge \exp\{ \, \E[H_N(m)] \, \} \ge 1
$$
where the last step is true since $H_N(m) = \frac{N}{s_N}\frac{1}{2}m^TAm \ge 0$ for all values of $m(\sigma)$ by positive definiteness of the matrix $A$. Applying this observation in the first step and the Cauchy-Schwarz bound (\ref{CauchySchwarz}) on all interactions in the second step, we obtain that
\begin{flalign*}
    &\mu_{N,\beta,\alpha} \left( |m_1|, \ldots, |m_{s_N}| > \delta\right) \\
    \le &  \sum_{m_1, \ldots, m_{s_N}} \1_{ \left\{|m_1|, \ldots, |m_{s_N}| > \delta\right\}} \exp\left\{ \frac{1}{2}\frac{N}{s_N} \sum_{k=1}^{s_N} \beta m_k^2 + 2\alpha m_k m_{k+1} \right\} \frac{1}{2^N}\prod_{k=1}^{s_N} {\frac{N}{s_N} \choose \frac{N}{s_N}\frac{1+m_k}{2}} \\
    \le & \sum_{m_1, \ldots, m_{s_N}} \1_{ \left\{|m_1|, \ldots, |m_{s_N}| > \delta\right\}}  \exp\left\{ \frac{1}{2}\frac{N}{s_N} \sum_{k=1}^{s_N} (\beta+ 2 \alpha) m_k^2 \right\} \frac{1}{2^N}\prod_{k=1}^{s_N} {\frac{N}{s_N} \choose \frac{N}{s_N}\frac{1+m_k}{2}} \\
    \le& \left( \sum_m \1_{\left\{m> \delta \right\}} \exp\left\{ \frac{N}{s_N}\left(\frac{\beta+2\alpha}{2}m^2 -\log2 +s(m) \right) \right\} \right)^{s_N},
\end{flalign*}
where $s(m) = -\frac{1+m}{2}\log \frac{1+m}{2} - \frac{1-m}{2}\log \frac{1-m}{2}$ and the above is simply an $s_N$-fold product of Curie-Weiss models. This is well-studied and we know (see e.g.\ \cite{Velenik_book}) that in the high temperature regime the function 
\begin{equation}\label{F}
F_{\beta+2 \alpha}(m) \coloneqq \frac{\beta+2\alpha}{2}m^2 -\log2 +s(m)  
\end{equation}
has its maximum at 0 and is negative everywhere else and and monotonically decreasing in $|m|$. On the event where $|m|> \delta$, we can upper bound the exponential function in the Gibbs measure by plugging $m = \delta$ into $F_{\beta+2 \alpha}$. For $m$ close to zero we can use the Taylor expansion
$$
-\log 2 +s(m) = -\frac{1}{2}m^2 - \frac{1}{12}m^4  +\mathcal{O}(m^6).
$$
Therefore, for $\beta + 2\alpha < 1$, i.e.\ $\frac{1 - (\beta + 2\alpha)}{2}>0$ we obtain that
\begin{flalign*}
    \mu_{N,\beta,\alpha}\left( |m_1|, \ldots, |m_{s_N}| > \delta \right) 
    \le& \left( \exp\left\{-\frac{N}{s_N}\frac{1 - (\beta + 2\alpha)}{2} \delta^2 \right\} \sum_{m}\1_{\left\{m> \delta \right\}} \right)^{s_N} \nonumber\\
    \le & \exp\left\{ -\frac{1 - (\beta + 2\alpha)}{2} N \delta ^2 \right\} \left( \frac{N}{s_N} \delta \right)^{s_N},
\end{flalign*}
which converges to zero since $ \log \frac{N}{s_N} = o\left(\frac{N}{s_N}\right)$. This means that for any fixed $\delta$, with high probability there will be blocks whose magnetization will be at most $\delta$.

\medskip

For the critical line $\beta + 2\alpha = 1$, we do the same steps as above. In this case, the second order terms cancel and instead we obtain
\begin{flalign*}
    \mu_{N,\beta,\alpha}\left( |m_1|, \ldots, |m_{s_N}| > \delta \right) 
    \le&  \left( \exp\left\{-\frac{N}{s_N}\frac{1}{12} \delta^4 \right\} \sum_{m}\1_{\left\{m> \delta \right\}} \right)^{s_N} \nonumber \\
    \le & \exp\left\{ -\frac{1}{12} N\delta^4 \right\} \left( \frac{N}{s_N} \delta \right)^{s_N},
\end{flalign*}
which also converges to zero.
\end{proof}

In the above lemma, we have shown that with high probability not all of the blocks can have a "large" magnetization. The idea of the proof  of Proposition~\ref{m1} is to start in the first block, then move to the left and the right and in both directions stop at the first block that has a sufficiently small magnetization. Apply the Cauchy-Schwarz bound on all the visited blocks and decouple everything that comes beyond.

\begin{proof}[Proof of Proposition \ref{m1}]

For $N\in \N$ and some fixed $\delta >0$ that we will choose later, define the random variables $R_N$ and $L_N$ as

\begin{flalign*}
    &R_N = \inf\left\{i=1,\ldots, s_N \, :\, |m_i|\le  \delta\right\} \\
    &L_N = \inf\left\{ i=1,\ldots, s_N-1 \,:\, |m_{s_N-i}| \le \delta\right\}.
\end{flalign*}

By \eqref{boundHigh}, we know that the event $2 \le R_N < s_N-L_N \le s_N$ occurs with probability $1-\exp\left\{ -\frac{1 - (\beta + 2\alpha)}{2}N  \delta^2 \right\}$. Therefore, instead of $\mu_{N,\beta,\alpha}(|m_1|>\eps)$, we will consider
$$
    \mu_{N,\beta,\alpha}( |m_1|>\eps , \, 2 \le R_N < s_N-L_N \le s_N)
    \le  \sum_{ 2\le r<s_N-l\le s_N} \mu_{N, \beta,\alpha}(|m_1|>\eps ,R_n=r, L_n=l).
$$
Let us compute the probabilities on the right hand side: To that end, for fixed $2\le r < s_N - l \le s_N$, we define the set
$$
\mathcal{A}_{r,l} \coloneqq \left\{ R_N = r, \, L_N = l \right\}.
$$
Let us define the set $\mathcal{I}_{l,r} \coloneqq \{1, \ldots, r, s_N-l, \ldots, s_N\}$ and note that, by definition, on the event that $\mathcal{A}_{r,l}$ occurs we know that 
\begin{equation}\label{setrl}
     |m_r|,|m_{s_N-l}| \le \delta \quad \text{and} \quad |m_k|>\delta \forall k\in \mathcal{I}_{l-1,r-1} \setminus\{1\} 
\end{equation}
and we want to compute
\begin{flalign}\label{sumrl}
   & \mu_{N,\beta,\alpha} \left(\, |m_1|>\eps, \, \mathcal{A}_{r,l}\right)\nonumber \\ 
    =&\frac{1}{Z_{N,\beta,\alpha}}\sum_{m_1, \ldots, m_{s_N}} \1_{\{|m_1|>\eps\}} \1_{\mathcal{A}_{r,l}}\exp\left\{ \frac{1}{2}\frac{N}{s_N} \sum_{k=1}^{s_N}\beta m_k^2 + 2\alpha m_km_{k+1} \right\}\frac{1}{2^N}\prod_{k=1}^{s_N}{\frac{N}{s_N} \choose \frac{N}{s_N}\frac{1+m_k}{2}}. 
\end{flalign}
In the above lemma, we bounded $Z_{N, \beta, \alpha}$ from below by 1. Since here we want to forget about everything that comes beyond $r$ and $s_N-l$, we have to be more delicate in bounding the partition function to keep a normalizing constant for these blocks. Recall that $Z_{N,\beta, \alpha}$ is an $s_N$-fold sum. By keeping only the "$m_k=0$"-summands for $k \in \mathcal{I}_{l-1,r-1}$ and using the fact that by Stirling's formula, for large $n$, it holds that ${n \choose \frac{n}{2}} = \mathcal{O}\left(  \frac{1}{\sqrt{n}} 2^n \right)$, we obtain the lower bound
\begin{equation}
    Z_{N,\beta,\alpha} \ge \left(\frac{1}{2^{\frac{N}{s_N}}} { \frac{N}{s_N}\choose \frac{N}{2s_N}}\right)^{r+l-1}  \widetilde{Z}_{N}^{r,l}  
     \sim \left(C_1\sqrt{\frac{s_N}{N}}\right)^{r+l-1} \widetilde{Z}_{N}^{r,l},
\end{equation}
where $C_1>0$ is a constant and
$$
\widetilde{Z}_{N}^{r,l} = \sum_{m_r, \ldots, m_{s_N-l}}\exp\left\{\frac{1}{2}\frac{N}{s_N}\left(\sum_{k=r}^{s_N-l}\beta m_k^2+ \sum_{k=r}^{s_N-l-1} 2\alpha m_km_{k+1}\right)\right\} \left(\frac{1}{2^{\frac{N}{s_N}}}\right)^{s_N-r-l+1} \prod_{k=r}^{s_N-l}{\frac{N}{s_N}\choose\frac{N}{s_N}\frac{1+m_k}{2}}
$$
is exactly the partition function for all the remaining blocks for a model with zero boundary condition. (Note that for abbreviation reasons, we drop the indication for $\beta$ and $\alpha$ here). In order to also obtain a bound for the sum in (\ref{sumrl}), we will apply the Cauchy-Schwarz upper bound (\ref{CauchySchwarz}) on $m_km_{k+1}$ for all $k \in \mathcal{I}_{l-1,r-1}\cup \{s_N-l\}$. Hence, for the Hamiltonian, we get the following upper bound:
\begin{flalign*}
    \sum_{k=1}^{s_N}\beta m_k^2 + 2\alpha m_k m_{k+1} 
    \le& \  \bigg( \, \sum_{k\in \mathcal{I}_{l-1,r-1}} (\beta + 2\alpha)m_k^2 \, \bigg) \\
    &+ \alpha m_r^2 + \alpha m_{s_N-l}^2 + \sum_{k=r, \ldots s_N-l}\beta m_k^2 + \sum_{k=r, \ldots, s_N-l-1}2\alpha m_km_{k+1}
\end{flalign*}
Recall that on $\mathcal{A}_{r,l}$, we have that $m_r^2, m_{s_N-l}^2 \le \delta^2$. Plugging this into the above observation, for the sum in (\ref{sumrl}) we get the following upper bound
\begin{flalign*}
    &\sum_{m_1, \ldots, m_{s_N}} \1_{\{|m_1|>\eps\}} \1_{\mathcal{A}_{r,l}}\exp\left\{ \frac{1}{2}\frac{N}{s_N} \sum_{k=1}^{s_N}\beta m_k^2 + 2\alpha m_km_{k+1} \right\}\frac{1}{2^N}\prod_{k=1}^{s_N}{\frac{N}{s_N} \choose \frac{N}{s_N}\frac{1+m_k}{2}} \\
    \le &  \left(\sum_{m_1} \1_{|m_1|>\eps} \exp\left\{ \frac{N}{s_N} \left(\frac{\beta + 2\alpha}{2}m_1^2 - \log 2 +s(m_1)\right)\right\}\right) \exp\left\{ \frac{1}{2}\frac{N}{s_N} 2 \alpha \delta^2 \right\} \\
    & \quad \times \left( \sum_{m_k} \1_{\{|m_k|\ge \delta \}} \exp\left\{\frac{N}{s_N} \left(\frac{\beta + 2 \alpha}{2} m_k^2 - \log 2 + s(m_k) \right)\right\}\right)^{r+l-2} \widetilde{Z}_{N}^{r,l}.
\end{flalign*}
Note that the partition function $\widetilde{Z}_{N}^{r,l}$ cancels since it also appears in the bound for $Z_{N,\beta, \alpha}$. Applying the usual Curie-Weiss bounds on the independent Curie-Weiss parts for $\beta + 2\alpha <1$, yields
\begin{flalign*}
   &\mu_{N,\beta,\alpha} (|m_1|>\eps , \, R_N=r, L_N=l) \\
   \le & \left( \sqrt{\frac{N}{s_N}} \right)^{r+l-1}\exp\left\{ - \frac{N}{s_N} \frac{1-(\beta + 2\alpha)}{2}\eps^2\right\}\frac{N}{s_N} \exp\left\{\alpha \frac{N}{s_N}\delta^2\right\} \\
   &\quad \times \exp\left\{ -\frac{N}{s_N}(r+l-2)\frac{1-(\beta + 2\alpha)}{2}\delta^2\right\} \left(\frac{N}{s_N} \right)^{r+l-2} \\
   \le& \exp\left\{ -\frac{N}{s_N}\left(\frac{1-(\beta + 2\alpha)}{2}\eps^2 - \alpha \delta^2 \right) \right\}\exp\left\{-(r+l-1) \left( \frac{N}{s_N}(1-(\beta + 2\alpha))\delta^2 - \frac{3}{2}\log\frac{N}{s_N}\right) \right\}.
\end{flalign*}
Partitioning $\{ |m_1|> \eps\}$ into $\{  2 \le R_N \le s_N-L_N \le s_N\}$ and its complement and then applying the above observation together with (\ref{boundHigh}) yields 
\begin{flalign*}
    &\mu_{N, \beta, \alpha} (|m_1|>\eps ) \\
    \le&\mu_{N,\beta,\alpha}( |m_1|>\eps , \, 2 \le R_N \le s_N-L_N \le s_N) + \exp\left\{ - N\left(\frac{1-(\beta + 2\alpha)}{2} -o(1)\right) \delta^2\right\} \\
    \le& \exp\left\{ -\frac{N}{s_N}\left(\frac{1-(\beta + 2\alpha)}{2}\eps^2 - \alpha \delta^2 \right)\right\} \\
    & \qquad \qquad \times \sum_{ 2\le r<s_N-l\le s_N} \exp\left\{-(r+l-1) \left( \frac{N}{s_N}(1-(\beta + 2\alpha))\delta^2 - \frac{3}{2}\log\frac{N}{s_N}\right) \right\} \\
     & \qquad\qquad + \exp\left\{ - N\left(\frac{1-(\beta + 2\alpha)}{2} -o(1)\right) \delta^2\right\} \\
    \le &  s_N^2 \exp\left\{ -\frac{N}{s_N}\left(\frac{1-(\beta + 2\alpha)}{2}\eps^2 - \alpha \delta^2 \right)\right\} \exp\left\{-\left( \frac{N}{s_N}(1-(\beta + 2\alpha))\delta^2 - \frac{3}{2}\log\frac{N}{s_N}\right) \right\} \\
    &\qquad \qquad + \exp\left\{ - N\left(\frac{1-(\beta + 2\alpha)}{2} -o(1)\right) \delta^2\right\}  \\
    & \le  \exp\left\{ -\frac{N}{s_N}\left(\frac{1-(\beta + 2\alpha)}{2}\eps^2 - \alpha \delta^2 - 2\frac{s_N}{N}\log s_N \right)\right\}  + \exp\left\{ - N\left(\frac{1-(\beta + 2\alpha)}{2} -o(1)\right) \delta^2\right\}. 
\end{flalign*}
Now since $s_N =o( \frac{N}{\log N})$, we have that for $N$ large enough, $2 \frac{s_N}{N}\log s_N \le \frac{1-(\beta + 2\alpha)}{4}$ and if we choose $\delta < \sqrt{\frac{1-(\beta + 2\alpha)}{8\alpha}}\eps$ and define $c_{\beta, \alpha}\coloneqq \frac{1-(\beta + 2\alpha)}{8}$, then we have
$$
\mu_{N, \beta, \alpha} (|m_1|>\eps ) \le 2\exp\left\{- c_{\beta, \alpha} \frac{N}{s_N} \eps^2\right\}.
$$

Similarly, for $\beta + 2\alpha =1$, we obtain
\begin{flalign*}
   &\mu_{N,\beta,\alpha} (|m_1|>\eps , \, R_N=r, L_N=l) \\
   \le & \left( \sqrt{\frac{N}{s_N}} \right)^{r+l-1}\exp\left\{ - \frac{N}{s_N} \frac{1}{12}\eps^4\right\}\frac{N}{s_N} \exp\left\{\alpha \frac{N}{s_N} \delta^4\right\}  \exp\left\{ -\frac{N}{s_N}\frac{(r+l-2)}{12}\delta^4\right\} \left(\frac{N}{s_N} \right)^{r+l-2} \\
   \le& \exp\left\{ -\frac{N}{s_N}\left(\frac{1}{12}\eps^4 - \alpha \delta^4 \right) \right\}\exp\left\{-(r+l-1)\left(\frac{N}{s_N}\frac{\delta^4}{24} - \frac{3}{2}\log \frac{N}{s_N}\right) \right\}.
\end{flalign*}
and hence it follows together with (\ref{boundCritical})
\begin{flalign*}
    &\mu_{N, \beta, \alpha} (|m_1|>\eps ) \\ 
    \le&  s_N^2 \exp\left\{ -\frac{N}{s_N}\left( \frac{\eps^4}{12} - \alpha \delta^4\right) \right\} \exp\left\{-\left( \frac{N}{s_N} \frac{\delta^4}{24} - \frac{3}{2}\log\left(\frac{N}{s_N}\right) \right) \right\}  +  \exp\left\{ -\left(\frac{1}{12} - o(1) \right)N \delta^4 \right\} \\
    \le & 2\exp\left\{ - \frac{1}{48}\eps^4\frac{N}{s_N} \right\},
\end{flalign*}
where the last step follows since for $N$ large enough, $2\frac{s_N}{N}\log s_N < \frac{\eps^4}{24}$ and then we choose $\delta< \left( 48\alpha\right)^{-\frac{1}{4}} \eps$
\end{proof}
It follows immediately that also $\sup_{i=1}^{s_N}|m_i|$ converges to zero in probability since
$$
\mu_{N,\beta,\alpha}(\exists k \, : \, |m_k|>\eps) \le s_N \cdot \mu_{N, \beta, \alpha} (|m_1|>\eps ),
$$
giving us the bounds
$$
\mu_{N,\beta,\alpha}(\exists k \, : \, |m_k|>\eps) \le 2s_N\exp\left\{- c_{\beta, \alpha} \frac{N}{s_N} \eps^2\right\} 
$$
in the strict high temperature case and 
$$
\mu_{N,\beta,\alpha}(\exists k \, : \, |m_k|>\eps) \le 2s_N \exp\left\{ - \frac{1}{48}\eps^4\frac{N}{s_N} \right\}
$$
for the critical line. This finishes the proof of Theorem \ref{WLLN}\,(i).

\section{Hubbard-Stratonovich}
For the CLT we need the rescaled version $\sqrt{\frac{N}{s_N}}m$ of the magnetization vector. We prove the convergence via a Laplace method using a Hubbard-Stratonovich transform. More precisely, we will convolute the distribution of the rescaled magnetization vector under the Gibbs measure with an $s_N$-dimensional normal distribution. If the covariance matrix is chosen properly, the resulting measure has a "nice" density, which will allow us to compute its Laplace transform and show convergence to the moment generating function of a Gaussian.
\begin{lemma}\label{phi}
    The distribution of the rescaled magnetization under the Gibbs measure satisfies
    $$
    \mu_{N,\alpha, \beta}\left( \sqrt{\frac{N}{s_N}}m \in \cdot \right) \star \mathcal{N}\left(0, A^{-1}\right) = z_N \exp \left\{ - \frac{N}{s_N}\varphi_N\left(\sqrt{\frac{s_N}{N}}x\right)  \right\}\d^{s_N}x,
    $$
    where 
     \begin{flalign}\label{def_phi}
     \varphi_N(x) &= \frac{1}{2}x^TAx - \sum_{k=1}^{s_N} \log\cosh \left( x^T A e_k \right)   \\
     &= \frac{\beta}{2}\sum_{k=1}^{s_N}x_k^2 + \frac{\alpha}{2}\sum_{k=1}^{s_N} x_k(x_{k-1} + x_{k+1}) - \sum_{k=1}^{s_N}\log\cosh(\beta x_k + \alpha(x_{k-1} + x_{k+1})) \nonumber
    \end{flalign}
    and the normalizing constant $z_N$ is given by
    $$
    z_N = \int_{\R^{s_N}} \exp \left\{ - \frac{N}{s_N}\varphi_N\left(\sqrt{\frac{s_N}{N}}x\right)  \right\}\d^{s_N}x.
    $$
        \end{lemma}

    \begin{proof}
 For any measurable set $U \in \mathcal{B}(\R^{s_N})$,
    \begin{align*}
        &\mu_{N,\alpha, \beta}\left( \sqrt{\frac{N}{s_N}}m \in \cdot \right) \star \mathcal{N}\left(0, A^{-1}\right)(U) \\
        =& z_N \sum_{\sigma\in \{-1,+1\}^N} \frac{1}{2^N} \int_U \exp\left\{-\frac{1}{2}\left(x-\sqrt{\frac{N}{s_N}}m\right)^TA\left(x-\sqrt{\frac{N}{s_N}}m\right) \right\}\exp\left\{ \frac{1}{2}\frac{N}{s_N}m^TAm\right\} \d^{s_N}x \\
        =&z_N  \int_U \exp\left\{-\frac{1}{2}x^TAx \right\}\sum_{\sigma\in \{-1,+1\}^N} \frac{1}{2^N}\exp\left\{ \sqrt{\frac{N}{s_N}}x^TAm\right\} \d^{s_N}x \\
        =& z_N  \int_U \exp\left\{-\frac{1}{2}x^TAx \right\}\sum_{\sigma\in \{-1,+1\}^N} \frac{1}{2^N}\exp\left\{\sqrt{\frac{s_N}{N}}x^TA \sum_{k=1}^{s_N} \sum_{i \in S_k} \sigma_i e_k \right\} \d^{s_N}x \\
        =& z_N  \int_U \exp\left\{-\frac{1}{2}x^TAx \right\} \frac{1}{2^N}\prod_{k=1}^{s_N}\prod_{i \in S_k} \left(\exp\left\{\sqrt{\frac{s_N}{N}}x^TAe_k \right\} + \exp\left\{-\sqrt{\frac{s_N}{N}}x^TAe_k \right\}\right) \d^{s_N}x \\
        =& z_N  \int_U \exp\left\{-\frac{1}{2}x^TAx \right\} \frac{1}{2^N}\prod_{k=1}^{s_N} \left(2 \cosh\left(\sqrt{\frac{s_N}{N}}x^TAe_k \right)\right)^{\frac{N}{s_N}} \d^{s_N}x \\
        =& z_N  \int_U \exp\left\{-\frac{1}{2}\frac{N}{s_N}\left(\sqrt{\frac{s_N}{N}}x\right)^TA\left(\sqrt{\frac{s_N}{N}}x\right) + \sum_{k=1}^{s_N}\frac{N}{s_N}\log\cosh\left(\sqrt{\frac{s_N}{N}}x^TAe_k \right)\right\} \d^{s_N}x \\
        =& z_N  \int_U \exp\left\{-\frac{N}{s_N}\varphi_N\left(\sqrt{\frac{s_N}{N}}x\right)\right\} \d^{s_N}x.\qedhere
    \end{align*}
    \end{proof}

    The central aim is to find the minimizers of the function $\varphi_N$. Any minimizer has to satisfy the critical equation $\nabla \varphi_N (x) = 0$. The gradient of $\varphi_N$ is given by
    \begin{flalign*}\nabla \varphi_N(x) &= \begin{pmatrix}
        \beta x_1 + \alpha(x_2 + x_{s_N}) - \beta \tanh\left(x^TAe_1\right) - \alpha (\tanh\left(x^TAe_2\right) + \tanh\left(x^TAe_{s_N}\right)) \\
          \vdots \\
        \beta x_{s_N} + \alpha(x_1 + x_{s_N-1}) - \beta \tanh\left(x^TAe_{s_N}\right) - \alpha (\tanh\left(x^TAe_1\right) + \tanh\left(x^TAe_{s_N-1}\right))
    \end{pmatrix}\\
    &= Ax - A \begin{pmatrix}
        \tanh\left( x^TAe_1\right) \\ \vdots \\  \tanh\left( x^TAe_{s_N}\right)
    \end{pmatrix}
\end{flalign*}
    and since the matrix $A$ is invertible, we arrive at the critical equation
    \begin{equation} \label{fixedpoint}
        x = \begin{pmatrix}
        \tanh\left( x^TAe_1\right) \\ \vdots \\  \tanh\left( x^TAe_{s_N}\right)
    \end{pmatrix}.
    \end{equation}
    It is immediate that the point $x=(0,\ldots, 0)$ always (i.e.\ independently of the choice of $\alpha, \beta$) satisfies (\ref{fixedpoint}) and therefore is a critical point. The next lemma will characterize the regimes where the point zero is a minimizer of $\varphi_N$.
    \begin{lemma}\label{definiteness}
        The Hessian matrix of $\varphi_N$ at $x=(0,\ldots, 0)$ is 
        \begin{itemize}
            \item positive definite if $\beta + 2\alpha < 1$,\\
            \item positive semi-definite if $\beta + 2\alpha = 1$ and \\
            \item not positive (semi-)definite if $\beta + 2\alpha > 1$.
        \end{itemize}
    \end{lemma}

    \begin{proof}
    The Hessian matrix of $\varphi_N$ at a point $x$ is given by
    \begin{equation}\label{Hess}
    \Hess \varphi_N (x) = A - A\sum_{k=1}^{s_N} \sech^2\left(x^TAe_k \right)e_ke_k^TA.
    \end{equation}
    Since $\sech(0)=1$, the above becomes simply $\Hess \varphi_N(0) = A- A^2$ and the eigenvalues are given by
\begin{align}\label{eq:eigenvalues}
    \lambda_j = \beta - \beta^2 + (2\alpha - 4\alpha \beta) \cos\left(2\pi \frac{j}{s_N} \right) - 2\alpha^2\cos \left(4 \pi \frac{j}{s_N}\right), \quad j = 1, \ldots, s_N.
    \end{align}
    Indeed, this is not hard to see since $A^2$ is a cyclic matrix as well, hence we can readily compute its eigenvalues and together with the eigenvalues of $A$ one obtains the above result. The Hessian matrix is positive definite if and only if its smallest eigenvalue is positive. This leads to the following three cases:

    \begin{itemize}
        \item If $2\alpha > 4\alpha\beta$ (this is equivalent to $\beta > \frac{1}{2}$ and $ \alpha >0$), then 
                $$
        \lambda_{\frac{s_N}{2}}= \beta - 2\alpha - (\beta - 2\alpha)^2
        $$
        is the smallest eigenvalue and it is always positive since $0< \beta - 2\alpha < 1$ is satisfied in this case. Note that if $s_N$ is odd, then the above gives a lower bound.
        \item If $2\alpha < 4\alpha\beta$ (this is equivalent to $\beta < \frac{1}{2}$ and $ \alpha >0$), then 
        $$
        \lambda_0 = \beta + 2\alpha - (\beta + 2\alpha)^2 
        $$
        is the smallest eigenvalue and it is positive if and only if $0< \beta + 2\alpha < 1$.    
    \end{itemize}
    Putting together both cases proves that $\Hess \varphi_N(0)$ is positive definite iff $\beta + 2 \alpha <1$ and positive semi-definite on the critical line $\beta + 2 \alpha=1$.
    \end{proof}
    Lemma \ref{definiteness} shows that the in the high temperature regime (and possibly in the critical case) the function $\varphi_N$ has a minimum in the origin. The next lemma will prove that in both cases this is actually the global minimum. Afterwards we will turn to the low temperature case.
    \begin{lemma}\label{HighTemp}
        If $\beta + 2 \alpha \le 1$, the function $\varphi_N$ takes its global minimum in $x = (0, \ldots, 0)$.
    \end{lemma}
    \begin{proof}
    Let us start off with the subcritical regime $\beta + 2\alpha < 1$. The critical case will follow in a very similar manner. Recall that any critical point has to satisfy (\ref{fixedpoint}), and therefore has to be a fixed point of the function $h$ defined by
    $$
    x \overset{(\ref{fixedpoint})}{=} \sum_{k=1}^{s_N}\tanh\left( x^TAe_k \right)e_k \eqqcolon h(x).
    $$
    The Jacobian of $h$ is given by $\D h(x) = \sum_{k=1}^{s_N}\sech^2\left( x^TAe_k \right)e_k e_k^T A$ and is the product of the diagonal matrix $\sum_{k=1}^{s_N}\sech^2\left( x^TAe_k \right)e_k e_k^T$, whose entries are bounded by 1, and $A$. Hence for the spectral norm of $\D h (x) $ we obtain for any $x \in \R^{s_N}$,
    $$
    \|\D h (x)\|_2 \le \left\| \, \sum_{k=1}^{s_N}\sech^2\left( x^TAe_k \right)e_k e_k^T \, \right\|_2 \|A\|_2 \le \|A\|_2 = \beta + 2\alpha.
    $$
    Hence $h$ is a contraction in the case where $\beta + 2\alpha <1$ and by Banach's fixed point Theorem has a unique fixed point at $(0,\ldots,0)$ and consequently this is the unique minimizer of $\varphi_N$.

    Next, consider the critical line $\beta + 2\alpha = 1$. The arguments in this case are similar. The only difference is that now $\|\D h(x)\|_2 < 1$ for all $x$ that have no zero component and $\|\D h (x)\|_2 = 1$ for all $x$ such that there exists $k \in \{1, \ldots, s_N\}$ with $x_k= 0$. In this case $h$ is only a weak contraction. However, the block magnetizations $m_k$ only live on $[-1,1]$, hence it suffices if we restrain ourselves to the compact set $[-1,1]^{s_N}$ such that again by Banach's fixed point Theorem $h$ has the unique fixed point $x=(0, \ldots,0)$.
    \end{proof}
    
It remains to find the minimizers in the low temperature regime, which is less straightforward. In this case, $x=0$ is also a critical point but we have seen in Lemma \ref{definiteness} that there is no minimum. In order to identify the minimizers, the first step will be to show that the solutions must satisfy that all coordinates have the same sign. From there it will not be hard to show that they must already be equal.

\begin{lemma}\label{SameSign}
    Let $\beta + 2 \alpha > 1$. Let $x$ be a minimizer of $\varphi_N$. Then $x$ satisfies
        $$
     x_k \ge 0 \quad \text{for some $k\in \{1,\ldots,s_N\}$} \Leftrightarrow x_k \ge 0 \quad \text{for all $k\in \{1,\ldots,s_N\}$}.    
    $$
\end{lemma}

\begin{proof}
Let $x \in [0,1]^{s_N}$. Further let $J \subset \{1, \ldots, s_N\}$ such that $0< |J| < s_N$ and define $y$ as 
$$
y_k = x_k \quad \text{for all $ k \in J^c$} \quad \text{and} \quad y_k = -x_k \quad \text{for all $ k \in J$}.
$$
We will show that $\varphi_N(x) - \varphi_N(y) \le 0$ with equality iff $x_k = 0$ for all $k \in J$. To that end, first note that 
$$
\left|\frac{\d}{\d x}\log \cosh(x)\right| = |\tanh(x)| <1,
$$
and therefore for all $0 \le \widetilde{x}\le\widetilde{y}$ it is true that
\begin{equation}\label{sublinear}
    -1 < \frac{\log\cosh(\widetilde{x}) - \log\cosh(\widetilde{y})}{\widetilde{x}-\widetilde{y}} <1 .
\end{equation}
Note that the quadratic terms of $\varphi_N$ cancel when considering the difference, hence applying the observation in (\ref{sublinear}) in the second step and then plugging in the definition of $y$ as we obtain
\begin{flalign*}
    \varphi_N(x) - \varphi_N(y) 
    =& \alpha\sum_{k=1}^{s_N} x_k x_{k+1} - y_k y_{k+1} \\
    &+ \sum_{k=1}^{s_N} \log \cosh(\beta y_k + \alpha(y_{k-1} + y_{k+1})) -\log \cosh(\beta x_k + \alpha(x_{k-1} + x_{k+1})) \\
    \le& \alpha\sum_{k=1}^{s_N} 2x_k x_{k+1}(\1_{k\in J, k+1\notin J}+ \1_{k \notin J, k+1\in J}) \\  
    &+ \sum_{k=1}^{s_N} (\beta y_k + \alpha(y_{k-1} + y_{k+1}) -(\beta x_k + \alpha(x_{k-1} + x_{k+1}))\1_{\{k-1,k,k+1\}\cap J \neq \emptyset}\\
    =& \alpha\sum_{j\in J}2 x_jx_{j-1}\1_{\{j-1\notin J\}} + 2x_jx_{j+1}\1_{\{j+1\notin J\}}  + \sum_{j \in J} -2\beta x_j - 2\alpha x_{j} - 2\alpha x_{j} \\
    \le & \sum_{j \in J} ( -2\beta)x_j
\end{flalign*}
where the last inequality follows since $x_{j-1}, x_{j+1}$ and the indicator function can be bounded from above by 1. It follows that $\varphi_N(x) - \varphi_N(y) \le 0$ with equality iff $x_j = 0$ for all $j \in J$ as desired.
\end{proof}

\begin{lemma}\label{LowTemp}
    Let $\beta + 2\alpha >1$. The function $\varphi_N$ attains its minimum at $x= \overrightarrow m^*$ and $x= - \overrightarrow m^*$ where $\overrightarrow m^* = (m^*(\beta + 2 \alpha), \ldots ,m^*(\beta + 2 \alpha)) $.
\end{lemma}

\begin{proof}
From Lemma \ref{SameSign} together with (\ref{fixedpoint}) it follows that a minimizer $x$ satisfies
     \begin{equation} \label{AllZero}
     x_k = 0 \quad \text{for some $k\in \{1,\ldots,s_N\}$} \Leftrightarrow x_k = 0 \quad \text{for all $k\in \{1,\ldots,s_N\}$}.    
    \end{equation}
    Indeed, from (\ref{fixedpoint}) we obtain that for all $k = 1, \ldots, s_N$,
    \begin{equation}\label{artanh}
        \frac{\artanh(x_k) - \beta x_k}{\alpha} = x_{k+1} + x_{k-1}
    \end{equation}
    Assume that $x_1, \ldots, x_{s_N} \ge 0$ and $x_k=0$ for some $k \in \{1, \ldots, s_N\}$. (The proof works analogously for $x_1, \ldots, x_{s_N} \le 0$.) By the above it follows that 
    $$
     0 = \frac{\artanh(0) - \beta \cdot0}{\alpha} = x_{k+1} + x_{k-1},
    $$
    hence $x_{k+1} = -x_{k-1}$, but they have the same sign so they both have to be zero as well. Inductively, the right hand side in (\ref{AllZero}) follows.
    
    We know from Lemma \ref{definiteness} that $x=0$ is not a minimizer. The above observation together with Lemma \ref{SameSign} implies that all coordinates of a minimizer have to be either strictly positive or strictly negative. Assume that they are strictly positive (again the negative counterpart of the proof works analogously). Further assume that $x_k = \min\{ x_1, \ldots, x_{s_N}\}$. By (\ref{artanh}) it follows that
    $$
    \artanh(x_k) - \beta x_k =\alpha( x_{k+1} + x_{k-1}) \ge 2\alpha x_k
    $$
    which implies that $\artanh(x_j) - \beta x_j \ge 2 \alpha x_j$ for all $j \in \{1, \ldots, s_N\}$ by taking a look at the function $f(x) = \artanh (x) - (\beta + 2\alpha)x$, for $\beta + 2\alpha >1$ and $x \in \R_{>0}$. Finally, summing over the equations (\ref{artanh}) for all $k$ yields 
    \begin{flalign*}
        &\sum_{k=1}^{s_N} \frac{\artanh(x_k) - \beta x_k}{\alpha} = \sum_{k=1}^{s_N} x_{k+1} + x_{k-1} = 2 \sum_{k=1}^{s_N}x_k  \\
\Leftrightarrow \quad & 0 = \sum_{k=1}^{s_N} \artanh(x_k) - (\beta + 2\alpha)x_k
    \end{flalign*}
    But we have seen above that all summands in this last sum are non-negative, hence they all have to be zero, i.e. for all $k$, $x_k$ has to satisfy the Curie-Weiss equation which leads to the desired result.
\end{proof}

\section{Proof of Theorem \ref{CLT}\,(i)}
Here, we will prove the CLT in the high temperature regime.
In the previous section we have already computed the global minimizers of the density function $\varphi_N$, so that we are now in the position proof Theorem \ref{CLT}. Throughout the whole section we will assume that $\beta + 2\alpha <1$. Let us first provide some first steps and the idea of the proof.

\subsection{Method of proof}\label{subsec:method}
Our goal is to show that all finite marginals of the distribution of $\widetilde{m} \coloneqq \sqrt{\frac{N}{s_N}}m$ under $\mu_{N, \alpha, \beta}$ converge weakly to the Normal distribution $\mathcal{N}(0,\Sigma)$ on $\R^\N$, or on $\R^{\lim_{N}s_N}$ respectively, where $\Sigma$ is the covariance given by (the infinite version) of $(I-A)^{-1}$. 
Let $Z$ be an $s_N$ dimensional Gaussian vector (independent of $\tilde m$) with distribution $\mathcal{N}\left(0,  A^{-1} \right)$. For now, we will be interested in the distribution of $Z + \widetilde{m}$ and we shall remove the Gaussian part of $Z$ in the last step of the proof.
By the Cram\'er-Wold device and using the Laplace transform (the moment generating function), we are interested in the limit $\E[\exp \{ t^T (Z + \widetilde{m})\}]$ for any fixed $t\in \R^{s_N}\subseteq \R^\N$ with only finitely many non-zero entries. By Lemma \ref{phi} and a change of variables we have
\begin{align*}
    \E\left[\exp\left\{ t^T (Z + \widetilde{m}) \right\}\right] 
    =& z_N \int_{\R^{s_N}} \exp\left\{ -\frac{N}{s_N}\varphi_N\left(\sqrt{\frac{s_N}{N}}x\right)  + t^Tx\right\}\d^{s_N}x  \\
    =&\frac{ \int_{\R^{s_N}}
\exp\Bigg\{-\frac{1}{2} x^T A x + \frac{N}{s_N}\sum_{k=1}^{s_N}
\log\cosh\Big(\sqrt{\frac{s_N}{N}}(A x)_k\Big)
+ t^Tx
\Bigg\}\d^{s_N}x }{ \int_{\R^{s_N}}
\exp\Bigg\{
-\frac{1}{2} x^T A x
+ \frac{N}{s_N}\sum_{k=1}^{s_N}
\log\cosh\Big(\sqrt{\frac{s_N}{N}}(A x)_k\Big)
\Bigg\}\d^{s_N}x },
\end{align*}
where we write $(Ax)_k\coloneqq x^TAe_k$. We will show that the above converges to the moment generating function of the appropriate normal distribution, which is given by $\exp\left\{ -\frac{1}{2}t^T(A-A^2)^{-1}t\right\}$. To that end, first note that close to the origin the $\log\cosh$ behaves like a quadratic function with deviation
\begin{align}\label{eq:G}
G(y)\coloneqq\log\cosh(y)\;-\;\frac{y^2}{2}, \quad y\in\R.
\end{align}
Note that 
	\begin{align}\label{eq:logcosh}
		 -\frac{y^4}{12}\le	G(y)\le 0 \quad\text{ for all }y\in\R
	\end{align}
 follows for instance from checking $y=0$ and the first two derivatives. 
Plugging this into the above integrals and using that $\sum_{k=1}^{s_N} (A x)_k^2 = x^TA^2 x$, the above becomes
\begin{flalign}\label{eq:MGFrho}
 \E\left[\exp\left\{ t^T (Z + \widetilde{m}) \right\}\right] 
= \frac{ \int_{\R^{s_N}}
\exp\Big\{-\frac{1}{2} \|x\|^2 + \frac{N}{s_N}\sum_{k=1}^{s_N}
G\Big(\sqrt{\frac{s_N}{N}}(C x)_k\Big)
+ t^T(A-A^2)^{-\frac{1}{2}}x
\Big\}\d^{s_N}x }{ \int_{\R^{s_N}}
\exp\Big\{
-\frac{1}{2} \|x\|^2
+ \frac{N}{s_N}\sum_{k=1}^{s_N}
G\Big(\sqrt{\frac{s_N}{N}}(C x)_k\Big)
\Big\}\d^{s_N}x },
\end{flalign}
where $C\coloneqq A(A-A^2)^{-\frac{1}{2}}=A^{1/2}(I-A)^{-1/2}$ is some symmetric positive definite matrix and we did the change of variables $x \mapsto (A-A^2)^{\frac{1}{2}}x$ in the numerator and denominator. 

Now, we would like to apply Laplace's method to show that the dominating part of the integral is determined by the minimizer of $\varphi_N$ and is precisely given by our expected limit. A direct application however is impossible due to the growing dimension $s_N\to\infty$ of integration. Naively (as in Laplace' method for finite dimensions) one would like to restrict the domain of integration to some ball of large radius $R\gg 0$ and use a Taylor approximation of $\log\cosh$ around 0. The difficulty comes from the curse of dimensionality: an $s_N$-dimensional Normal distribution has most of its mass outside of this ball, and is concentrated close to a radius of order $\sqrt {s_N}$. 
In order to circumvent this problem, we need to 
\begin{enumerate}
    \item carefully restrict the area of integration to controllable regions (see $D_N$ below),
    \item crucially make use of the fact that $t$ has only finitely many coordinates in order to view \eqref{eq:MGFrho} as a perturbation of a Gaussian distribution,
    \item combine the above and remove the contribution of $Z\sim \mathcal N (0,A^{-1})$ to finish the proof of Theorem \ref{CLT}.
\end{enumerate}
We will deal with these steps in the following subsections.

\subsection{Restricting the domain of integration}\label{subsec:integration}

To this end, we need the following representation.

\begin{lemma}\label{lem:log-concave}
We can write the above as 
$$
\E\left[\exp\left\{ t^T (Z + \widetilde{m}) \right\}\right] =\E_{\rho_N}\left[\exp\{{t^T(A-A^2)^{-\frac{1}{2}}X}\}  \right], 
$$
where $X$ denotes a $s_N$-dimensional random variable whose distribution is strongly log-concave with density
\begin{equation}\label{def:rho_N}
\rho_N(x) \propto \exp\Bigg\{
-\frac{1}{2} \|x\|^2
+ \frac{N}{s_N}\sum_{k=1}^{s_N}
G\Big(\sqrt{\frac{s_N}{N}}(C x)_k\Big)
\Bigg\}.
\end{equation}
\end{lemma}

\begin{proof} A (strongly) log-concave measure has a density of the form
$
e^{-U(x)}
$
for where $U(x)$ is a (strongly) convex function.  
Similar to \eqref{Hess}, we have
\begin{flalign*}
\Hess U(x) =& \Hess \left(\frac{1}{2} \|x\|^2
- \frac{N}{s_N}\sum_{k=1}^{s_N}
G\Big(\sqrt{\frac{s_N}{N}}(C x)_k\Big) \right) \\
=& \Hess \left(\frac{1}{2} x^T(I+C^2)x
- \frac{N}{s_N}\sum_{k=1}^{s_N}
\log\cosh \Big(\sqrt{\frac{s_N}{N}}(C x)_k\Big) \right) \\
=&I+C^2 - \sum_{k=1}^{s_N}\sech^2 \Big(\sqrt{\frac{s_N}{N}}(C x)_k\Big) C e_k e_k^T C \\
=& I + C\left(I - \sum_{k=1}^{s_N}\sech^2 \Big(\sqrt{\frac{s_N}{N}}(C x)_k\Big)  e_k e_k^T  \right)C.
\end{flalign*}
Note that $0\le \sech^2\le 1$, hence the diagonal matrix $I - \sum_{k=1}^{s_N}\sech^2 \Big(\sqrt{\frac{s_N}{N}}(C x)_k\Big)  e_k e_k^T$ is component-wise non-negative and since $C^2$ is positive definite, it follows that $\Hess U(x)$ is strictly positive definite with a uniform lower bound $I$ in terms of Loewner order (i.e.~the least singular value is at least $1$). 
\end{proof}

Let us denote by $B_R=\{x\in\R^{s_N}:\|x\|\le R\}$ the centered ball of radius $R$ in $s_N$ dimensions. Define the set $\mathcal{I}_d \coloneqq \{1, \ldots, d, s_N-d+1, \ldots, s_N \}$ and let projection $P_{\mathcal{I}_d}\colon \R^{s_N} \to \R^{2d}$ denote the the projection onto the set $\mathcal{I}_d$, i.e.\ onto the first and last $d$ coordinates, where $1\le 2d \le s_N$. Further define $$B_R^{(2d)}=\{x\in\R^{s_N}: \|P_{\mathcal{I}_d}(x)\|=\Big(\sum_{j\in \mathcal{I}_d} x_j^2\Big)^{1/2}\le R\}$$
the cylinder set of a ball in the dimensions  $ 1, \ldots  ,d$ and $s_N - d+1, \ldots, s_N$ for $d<s_N$. The following crucial lemma will lower the effective dimension to a logarithmic number, thus \emph{purging the curse of dimensionality}.

\begin{lemma}\label{lemma:D_N-restriction}
 Let $d_N := \lceil c_6\log N\rceil$ for some constant $c_6>0$ (to be defined later) and define
 \begin{flalign*}
     D_N := \begin{cases}
          B_{2\sqrt{s_N}}\cap B^{(2d_N)}_{2\sqrt {2d_N}} & \quad \text{if } 2 d_N\le s_N \\
           B_{2\sqrt{2d_N}} & \quad \text{if } 2d_N > s_N.
     \end{cases}
 \end{flalign*}
Then, there exists a universal constant $c_1>0$ such that 
$$
\E_{\rho_N}\Big[ \1_{ D_N^c}(X) \exp\{t^T(A-A^2)^{-\frac{1}{2}}X\}\Big] =\mathcal O (\exp\{-c_1 d_N\}) \longrightarrow 0,
$$
as $N\to\infty$.
\end{lemma}

\begin{proof}
By Theorem 5.2.11 in \cite{Vershynin}, for every Lipschitz function $f \colon \R^{s_N} \to \R$, it holds that 
\begin{equation}\label{eq:subgauss}
    \|f(X) - \E{\rho_N}[f(X)]\|_{\psi_2}\le c_2\|f\|_{\lip},
\end{equation}
where $\|\cdot\|_{\psi_2}$ denotes the subgaussian norm and $c_2>0$ is some universal constant. The constant depends on the convexity of $U$ from the proof of Lemma \ref{lem:log-concave}, whose lower bound we found to be $1$. In particular, $f(X)- \E_{\rho_N}[f(X)]$ is a subgaussian random variable. We want to apply this to the norm of the random vector and the norm of the vector projected onto some $q_N\in\N$ of its coordinates, respectively. The norm of the projection $P_{\mathcal{I}_{q_N}}\colon \R^{s_N} \to \R^{2q_N}$ defines a $1$-Lipschitz function  $f(x)= \|P_{\mathcal{I}_{q_N}}(x) \|$.
By \eqref{eq:subgauss}, $
\|P_{\mathcal{I}_{q_N}}(x) \| - \E_{\rho_N}\|P_{\mathcal{I}_{q_N}}(x) \|
$
is subgaussian with
$$
\| \|P_{\mathcal{I}_{q_N}}(x) \| - \E_{\rho_N}\left[\|P_{\mathcal{I}_{q_N}}(x) \|\right]\|_{\psi_2} \le c_2,
$$
where again $X$ denotes a random variable with distribution given by the probability measure with density $\rho_N$.
Let us compute $\E_{\rho_N}[f(X)]$. The Poincar\'e inequality \cite[Equation (1.6)]{bobkov:ledoux:2000} implies for any $v\in \R^{s_N}$, 
$$
v^T\Cov(X)v = \Var(v^TX) \le c_N \int_{\R^{s_N}} \rho_N(x) \|\nabla f_v(x)\|^2 \d^{s_N}x = \|v\|^2,
$$
where by $f_v$ we denote the function $x\mapsto v^Tx$.
It follows by Jensen's inequality that
$$
    \E_{\rho_N}\left[ \|P_{\mathcal{I}_{q_N}}(X) \| \right]^2 \le  \E_{\rho_N}\left[\|P_{\mathcal{I}_{q_N}}(X) \|^2 \right] = \sum_{j \in \mathcal{I}_{q_N}}\E_{\rho_N}[X_j^2] = \sum_{j \in \mathcal{I}_{q_N}}e_k^T\Cov(X)e_k \le 2q_N.
$$
Therefore, we have $ \E_{\rho_N}\left[ \|P_{\mathcal{I}_{q_N}}(X) \| \right]\le \sqrt{2q_N}$. The standard subgaussian tail probability as in \cite[Proposition 2.6.6]{Vershynin} implies the existence of some other universal constant $c_1>0$ such that 
\begin{align}\label{eq:Lip1}
    \P\Big( \|P_{\mathcal{I}_{q_N}}(X)\| >2\sqrt{2q_N}\Big) &= \P\Big( \|P_{\mathcal{I}_{q_N}}(X)\| - \E_{\rho_N}[\|P_{\mathcal{I}_{q_N}}(X)\|] > 2\sqrt{2q_N} - \E_{\rho_N}[\|P_{\mathcal{I}_{q_N}}(X)\|] \Big) \nonumber \\
    &\le \exp\left\{ -{c_1}{(2\sqrt{2q_N} - \E_{\rho_N}[\|P_{\mathcal{I}_{q_N}}(X)\|] )^2} \right\} \le  \exp\left\{ -c_1{2q_N} \right\} .
\end{align}
In order to restrict the integration domain to $D_N$ as claimed, we need a second observation. To that end, we will apply \cite[Theorem 5.2.11]{Vershynin} to another Lipschitz function, namely $f_{\tilde{t}} \colon x \mapsto \tilde{t}^Tx$, for the vector $\tilde t \coloneqq (A-A^2)^{-\frac{1}{2}}t\in \R^{s_N}$, whose Lipschitz norm $\|\tilde{t}\|$ is a fixed constant. Note that $\E_{\rho_N}[\tilde{t}^TX]=0 $ by symmetry, thus $\|\tilde{t}^TX\|_{\psi_2} \le {c_2\|\tilde{t}\|}$ by \eqref{eq:subgauss}. Since $\tilde{t}^TX$ is a subgaussian random variable, we obtain a uniform upper bound
\begin{align}\label{eq:Lip2}
	\E_{\rho_N}\Big[ e^{\tilde{t}^TX} \Big] \le \exp\left\{ c_3\|\tilde{t}\|^2\right\}<\infty.
\end{align} 
Finally, let us put together the observations \eqref{eq:Lip1} and \eqref{eq:Lip2}.

Let us consider the case $2d_N\le s_N$, the other case being analogous and simpler. First we will apply \eqref{eq:Lip1} to $q_N = \frac{s_N}{2}$, i.e. we are considering the norm of the whole vector. It follows, by Cauchy-Schwarz 
 \begin{flalign*}
     \E_{\rho_N}\Big[ \1_{ B_{2\sqrt{s_N}}^c} e^{\tilde{t}^TX } \Big] &\le \E_{\rho_N}\Big[ e^{2\tilde{t}^TX } \Big] \P\Big( \|X\| >2\sqrt{s_N}\Big) =\mathcal O (\exp\{-c_1 s_N\}). 
 \end{flalign*}
Now, apply \eqref{eq:Lip1} for any other $q_N=d_N\to\infty$, then again by Cauchy-Schwarz,
\begin{align*}
\E_{\rho_N}\Big[ \1_{ D_N^c} e^{2\tilde{t}^TX } \Big] \le c_{4}\exp\left\{ 4c_3\|\tilde{t}\|^2 \right\} \exp\left\{ -c_1 2d_N \right\} \longrightarrow 0\qedhere
\end{align*}
\end{proof}

\subsection{Approximating the perturbation}\label{subsec:approx}
By \eqref{eq:MGFrho} and Lemma \ref{lemma:D_N-restriction}, we have shown so far that 
\begin{flalign}\label{eq:zwitscherndeszwischenergebnis}
&\E\left[\exp\left\{ t^T (Z + \widetilde{m}) \right\}\right] \\
=&\frac{ \int_{D_N}
\exp\Big\{-\frac{1}{2} \|x\|^2 + \frac{N}{s_N}\sum_{k=1}^{s_N}
G\Big(\sqrt{\frac{s_N}{N}}(C x)_k\Big)
+ t^T(A-A^2)^{-\frac{1}{2}}x
\Big\}\d^{s_N}x }{ \int_{D_N}
\exp\Big\{
-\frac{1}{2} \|x\|^2
+ \frac{N}{s_N}\sum_{k=1}^{s_N}
G\Big(\sqrt{\frac{s_N}{N}}(C x)_k\Big)
\Big\}\d^{s_N}x }\big(1+o(1)\big).\nonumber
\end{flalign}
 Now, perform another change of variables in the numerator by shifting $x\mapsto x-(A-A^2)^{-\frac{1}{2}}t\in \tilde D_N$ for $\tilde D_N=D_N-(A-A^2)^{-\frac{1}{2}}t$ and the above becomes
$$
 \exp\left\{-\frac{1}{2}t^T(A-A^2)^{-1}t \right\}\frac{ \int_{\tilde D_N}
\exp\Bigg\{-\frac{1}{2} \|x\|^2 + \frac{N}{s_N}\sum_{k=1}^{s_N}
G\Big(\sqrt{\frac{s_N}{N}}(C x)_k + \sqrt{\frac{s_N}{N}}\Big((I-A)^{-1}t \Big)_k\Big)
\Bigg\}\d^{s_N}x }{ \int_{\tilde D_N}
\exp\Bigg\{
-\frac{1}{2} \|x\|^2
+ \frac{N}{s_N}\sum_{k=1}^{s_N}
G\Big(\sqrt{\frac{s_N}{N}}(C x)_k\Big)
\Bigg\}\d^{s_N}x },
$$
since $C(A-A^2)^{-\frac{1}{2}} = A(A-A^2)^{-1} = (I-A)^{-1}$. Therefore, we have reduced the proof of Theorem \ref{CLT}\,(i) to showing that the above quotient tends to 1 as $N\to \infty$. 
Note that the numerator and denominator differ only by the summand $\sqrt{\frac{s_N}{N}}\Big((I-A)^{-1}t \Big)_k$ inside the function $G$. Hence, we have to show that this can be neglected in the limit. Since $t$ has only finitely many non-zero coordinates, clearly the summand goes to zero as $N\to \infty$, so we can do the Taylor expansion
\begin{align}\label{eq:Taylorrr}
\sum_{k=1}^{s_N}G\Big(\sqrt{\tfrac{s_N}{N}}(C x)_k + \sqrt{\tfrac{s_N}{N}}\big((I-A)^{-1}t\big)_k\Big) = \sum_{k=1}^{s_N}G\Big(\sqrt{\tfrac{s_N}{N}}(C x)_k\Big) + G'\big( \xi_k \big)\sqrt{\frac{s_N}{N}}\Big( (I-A)^{-1} t\Big)_k,
\end{align}
where $\xi_k$ are intermediate values given by $\xi_k=\sqrt{\tfrac{s_N}{N}}(C x)_k + \lambda_k\sqrt{\tfrac{s_N}{N}}\big((I-A)^{-1}t\big)_k $  for some $0\le\lambda_k\le 1$ and we have to show that the second term goes to zero. 
\begin{lemma}\label{lem:G'=o}
If $s_N = o\left({N}{(\log N)^{-5/2}} \right)$ there exists $c_6>0$ such that for $d_N=c_6\log N$ it holds
$$\frac N {s_N}\sum_{k=1}^{s_N}G'\Big( \xi_k \Big)\sqrt{\frac{s_N}{N}}\Big( (I-A)^{-1} t\Big)_k=o(1)$$
uniformly in $x\in B_{\max\{4\sqrt{s_N}, 4 \sqrt{2d_N} \}}\cap B^{(2d_N)}_{4\sqrt {2d_N}}$.
\end{lemma}
Note that $D_N,\tilde D_N\subseteq  B_{\max\{4\sqrt{s_N}, 4 \sqrt{2d_N} \}}\cap B^{(d_N)}_{4\sqrt {d_N}}$. 

A key argument for the proof of this lemma is the following property of the matrices $C$ and $(I-A)^{-1}$. Namely both $C$ and $(I-A)^{-1}$ are symmetric circulant matrices with exponentially decaying entries away from the diagonal (and the corner).

\begin{lemma}\label{lem:matrix_decay}
The entries of $(I-A)^{-1}$ satisfy
\begin{equation}\label{eq:IAinv}
\big((I-A)^{-1}\big)_{j,k}
=\frac{ \kappa_1^{\abs{j-k}} + \kappa_1^{s_N-\abs{j-k}} }{(1-\kappa_1^{s_N}) \sqrt{(1-\beta)^2-4\alpha^2} },
\qquad 
\kappa_1=\frac{(1-\beta)-\sqrt{(1-\beta)^2-4\alpha^2}}{2\alpha}.
\end{equation}
It is easy to see that $0<\kappa_1<1$.  Likewise, for $C$ there exists $0<\kappa_2< 1$, such that
\begin{equation}\label{eq:Cexp}
C_{j,k}
= c_{5}\Big(
\kappa_2^{|j-k|} +  \kappa_2 ^{(s_N-|j-k|)}\Big).
\end{equation}
\end{lemma} 

\begin{proof}
    
The proof of this lemma is a straightforward computation, using the fact that all circulant matrices can be diagonalized by the discrete Fourier transform matrix. More precisely, the $s_N\times s_N$-matrix $I-A$ can be written as $I-A =F_{s_N}^{-1} \diag(F_{s_N}(I-A)e_1) F_{s_N} $, where
$$
(F_{s_N})_{j,k} = e^{-\frac{2\pi i}{s_N}jk}.
$$
This implies that for $j>k$, 
$$(I-A)^{-1} = F_{s_N}^{-1} \diag(F_{s_N}(I-A)e_1)^{-1} F_{s_N} = F_{s_N}^{-1} \diag((F_{s_N}(I-A)e_1)^{-1}) F_{s_N}$$
and again a straightforward computation gives that
$$
(I-A)^{-1}_{j,k}= \frac{1}{s_N} \sum_{l=0}^{s_N-1} \frac{e^{\frac{2\pi il }{s_N}(j-k)}}{1-\beta - 2\alpha \cos\left( \frac{2\pi l}{s_N}\right)} =  \frac{1}{s_N} \frac{\kappa_1}{\alpha}\sum_{l=0}^{s_N-1} \frac{e^{\frac{2\pi il }{s_N}(j-k)}}{1 - 2\kappa_1 \cos\left( \frac{2\pi l}{s_N}\right) + \kappa_1^2}
$$
where we used that $\frac{\alpha}{\kappa_1}(1+\kappa_1^2) = 1-\beta$. Note that the following identity holds (that we will prove below):
\begin{equation} \label{eq:DiscretePoissonId}
\frac{1}{s_N}\sum_{l=0}^{s_N-1} \frac{e^{\frac{2\pi il }{s_N}(j-k)}}{1 - 2\kappa_1 \cos\left( \frac{2\pi l}{s_N}\right) + \kappa_1^2} = \frac{1}{1-\kappa_1^2}\frac{\kappa_1^{j-k} + \kappa_1^{s_N -(j-k)}}{1-\kappa_1^{s_N}}
\end{equation}
By \eqref{eq:DiscretePoissonId}, we obtain that $(I-A)^{-1}_{jk}= \frac{\kappa_1}{\alpha}\frac{1}{1-\kappa_1^2}\frac{\kappa_1^{j-k} + \kappa_1^{s_N -(j-k)}}{1-\kappa_1^{s_N}}$ and since 
$$
\frac{\alpha}{\kappa_1}(1-\kappa_1^2) = \frac{\alpha}{\kappa_1}(1+\kappa_1^2) - 2\alpha \kappa_1 = 1-\beta -2\alpha \kappa_1 = \sqrt{(1-\beta)^2-4\alpha^2},
$$
where the last step follows by the definition of $\kappa_1$ in \eqref{eq:IAinv}, this proves the representation of $(I-A)^{-1}$ in \eqref{eq:IAinv}.

\medskip

Let us now prove the identity in \eqref{eq:DiscretePoissonId}: Using the formula for geometric series, it is not hard to see that for any $0<r<1$,
\begin{flalign*}
\sum_{k\in \Z}r^{|k|}e^{\frac{2\pi i k l}{s_N}} &=   -1 +  \sum_{k=0}^{\infty}r^{k}\left(e^{\frac{2\pi i k l}{s_N}} +e^{-\frac{2\pi i k l}{s_N}} \right) = -1 + \frac{1}{1-re^{\frac{2 \pi i l}{s_N}}} + \frac{1}{1-re^{-\frac{2 \pi i l}{s_N}}} \\
& = -1 + \frac{2 - 2r\cos\left( \frac{2\pi l }{s_N}\right)}{1-2r\cos \left(\frac{2 \pi l}{s_N} \right) + r^2} = \frac{1-r^2}{1-2r\cos \left(\frac{2 \pi l}{s_N} \right) + r^2},
\end{flalign*}
hence 
\begin{equation}\label{eq:GeometricSum}
    \frac{1}{1-r^2}\sum_{k\in \Z}r^{|k|}e^{\frac{2\pi i k l}{s_N}} = \frac{1}{1-2r\cos \left(\frac{2 \pi l}{s_N} \right) + r^2}
\end{equation}
This implies that for any $m \in \N$,
\begin{flalign*}
    \frac{1}{s_N}\sum_{l=0}^{s_N-1} \frac{e^{\frac{2\pi il }{s_N}m}}{1 - 2r \cos\left( \frac{2\pi l}{s_N}\right) + r^2} &= \frac{1}{s_N}\sum_{l=0}^{s_N-1} e^{\frac{2\pi il }{s_N}m} \frac{1}{1-r^2} \sum_{k\in \Z} r ^{|k|}e^{\frac{2\pi i l k}{s_N}} \\
    &= \frac{1}{1-r^2} \sum_{k \in \Z} r^{|k|}\frac{1}{s_N}\sum_{l=0}^{s_N-1}e^{\frac{2\pi i l}{s_N}(m+k)} \\
    &= \frac{1}{1-r^2} \sum_{q \in \Z} r^{|qs_N-m|},
\end{flalign*}
where the last step follows since $\frac{1}{s_N}\sum_{l=0}^{s_N-1}e^{\frac{2\pi i l}{s_N}(m+k)}$ is equal to one if $m+k = 0(\mod s_N)$ and otherwise if it is equal to zero. Also note that now without loss of generality, we can assume that $m < s_N$ since otherwise we can simply go over to $m (\mod s_N)$. It remains to show that $\sum_{q \in \Z} r^{|qs_N-m|} = \frac{r^{m} + r^{s_N -m}}{1-r^{s_N}}$. Partitioning the sum into the cases $q=0, q\ge 1$ and $q\le -1$,  we obtain the first claim
\begin{align*}
    \sum_{q \in \Z} r^{|qs_N-m|} = r^m + \sum_{q \ge 1}r^{qs_N - m } + \sum_{p\ge1}r^{ps_N + m} = r^m + \frac{r^{s_N -m}}{1-r^{s_N}} + \frac{r^{s_N + m}}{1-r^{s_N}}  = \frac{r^{m} + r^{s_N-m}}{1-r^{s_N}}.
\end{align*} 

For $C=(A^{-1} -I )^{-\frac{1}{2}}$ the same computation is more involved since the eigenvalues of $A^{-1} -I$ are more complicated. For this reason we will not compute the exact constants and use the result from Theorem 3.4 in \cite{BeRa07} proving that there exists positive constants $K$ and $0 < \kappa_2 <1$ such that 
$$
(A^{-1} -I )^{-\frac{1}{2}}_{j,k} < cond \cdot K \cdot (\kappa_2^{|j-k|} + \kappa_2 ^{s_N-|j-k|}),
$$
where $cond$ denotes the spectral condition number of the matrix of eigenvectors of $A^{-1} - I$ and we define $c_5 \coloneqq cond \cdot  K$.
\end{proof}

\begin{proof}[Proof of Lemma \ref{lem:G'=o}]
 First let us assume that $s_N \gg \log N$, the other case $s_N = \mathcal{O}(\log N)$ is easier and we discuss it in the end. 
Let $l$ be the largest non-zero component of $t$ and define $\widetilde{d}_N \coloneqq \frac{5}{-\log \kappa_1}\log N$ such that $\kappa_1^{\widetilde{d}_N}=N^{-5}$.
Let $\mathcal{I}_{\widetilde{d}_N} \coloneqq \{1, \ldots, \widetilde{d}_N\}\cup\{ s_N-\widetilde{d}_N+1, \dots, s_N \}$ be the indices near the endpoints $1,s_N$ and let us consider $k \notin \mathcal{I}_{\widetilde{d}_N}$ in order to show that their contribution is negligible: for $k=\widetilde{d}_N+1, \ldots, s_N-\widetilde{d}_N$, we have by \eqref{eq:IAinv}
\begin{flalign}
    \Big\lvert \Big( (I-A)^{-1} t\Big)_k \Big\rvert=&\sum_{j=1}^{l} \abs{(I-A)^{-1}_{j,k}t_j} \\ 
    =& \frac{1}{(1-\kappa_1^{s_N})\sqrt{(1-\beta)^2 -4\alpha^2}} \sum_{j=1}^l (\kappa_1^{|j-k|}+\kappa_1^{s_N-|j-k|})\abs{t_j} \nonumber \\
 \le&    \frac{2\kappa_1^{\widetilde{d}_N}}{(1-\kappa_1^{s_N})\sqrt{(1-\beta)^2 -4\alpha^2}}   \sum_{j=1}^l\abs{t_j}=\mathcal O (N^{-5})\label{eq:jFar}
\end{flalign}
uniformly in $k\notin \mathcal{I}_{\widetilde{d}_N}$, where the last step follows from the definition of $\widetilde{d}_N$.

Next, we need to find a suitable bound for $G'(\xi_k)$. As $\xi_k=\sqrt{\tfrac{s_N}{N}}(C x)_k + \lambda_k\sqrt{\tfrac{s_N}{N}}\big((I-A)^{-1}t\big)_k $  for some $0\le\lambda_k\le 1$ are the intermediate values from Taylor's theorem, we know that 
\begin{equation}\label{def:Xi}
  \xi_k = \sqrt{\frac{s_N}{N}}(Cx)_k + \mathcal O (N^{-5})
\end{equation}
Recall that we are restricted to the ball $B_{4\sqrt{s_N}}$, hence
$$|\xi_k|\le 4\|C\|  \frac{s_N}{\sqrt{N}}+\mathcal O (N^{-5}).$$ Note that $|G'(x)|\le |x|$ by \eqref{eq:logcosh}
and therefore 
\begin{equation}\label{eq:Gprime}
    |G'(\xi_k)| \le 4 \|C\|  \frac{s_N}{\sqrt{N}}+\mathcal O (N^{-5}).
\end{equation}
Putting together the estimates (\ref{eq:jFar}) and (\ref{eq:Gprime}), we arrive at some rough bound
\begin{flalign}
\Big\lvert\frac{N}{s_N}\sum_{k=\widetilde{d}_N}^{s_N-\widetilde{d}_N}G'(\xi_k)\sqrt{\frac{s_N}{N}} \Big((I-A)^{-1}t \Big)_k\Big\rvert 
&=\mathcal O(N s_N N^{-5})\label{eq:IrN}.
\end{flalign}
It remains to consider
$$
\frac{N}{s_N}\sum_{k\in \mathcal{I}_{\widetilde{d}_N}}G'(\xi_k)\sqrt{\frac{s_N}{N}} \Big((I-A)^{-1}t \Big)_k,
$$ 
where now the contribution of $((I-A)^{(-1)}t )_k$ is not negligible.
We will again make use of the exponential decay away from the diagonal, this time for the matrix $C$. Hence, we will just use the uniform upper bound
\begin{equation}\label{eq:boundCt}
\abs{((I-A)^{-1}t)_k}\le \|(I-A)^{-1}\|\|t\|\le (1-\beta -2\alpha)^{-1} \|t\|\eqqcolon C_{\beta,\alpha,t}
\end{equation}
and show that 
$$
\frac{N}{s_N}\sum_{k\in \mathcal{I}_{\widetilde{d}_N}}G'(\xi_k)\sqrt{\frac{s_N}{N}} C_{\beta,\alpha, t}
$$
is small by bounding $G'(\xi_k)$ for $\xi_k=\sqrt{\tfrac{s_N}{N}}(C x)_k + \lambda_k\sqrt{\tfrac{s_N}{N}}\big((I-A)^{-1}t\big)_k $. To that end, split $(Cx)_k$ into two sums. Recall the definition $\mathcal{I}_{d_N} \coloneqq \{ 1, \cdots, d_N, s_N-d_N+1, \cdots, s_N\}$ for 
$$
d_N=\frac{5}{- \log \kappa_2 }\log N + \widetilde{d}_N = \left(\frac{5}{- \log \kappa_2 }+\frac{5}{- \log \kappa_1 } \right)\log N =: c_6 \log N
$$
and write
$$
(C x)_k=\sum_{j\notin \mathcal{I}_{d_N}}C_{kj}x_j+\sum_{j\in \mathcal{I}_{ d_N}}C_{kj}x_j
$$
then using the exponential decay \eqref{eq:Cexp} we obtain
\begin{align}\label{eq:kclose_jfar}
\sum_{j\notin \mathcal{I}_{d_N}}\abs{C_{k,j}x_j} 
&\le c_{5} \sum_{j\notin \mathcal{I}_{d_N}}( \kappa_2^{|j-k|} +  \kappa_2^{(s_N-|j-k|)})\abs{x_j}\nonumber\\
&\le c_5   \kappa_2^{d_N-\widetilde{d}_N}\sum_{j\notin \mathcal{I}_{d_N}}\abs{x_j} 
=\mathcal O(\kappa_2^{ d_N-\widetilde{d}_N}\, s_N ^{3/2}) =\mathcal O (N^{-5}s_N^{3/2}).
\end{align}
For the near part $j\in \mathcal{I}_{d_N}$ we need to use $x\in B^{(2d_N)}_{4\sqrt {2d_N}}$. That means now we are integrating over a set where the first and last $d_N$ coordinates lie in the $2d_N$-dimensional ball of radius $4\sqrt{2d_N}$, i.e.~$\|P_{\mathcal{I}_{d_N}}x\|\le4\sqrt{2d_N}$, hence
$$
|\sum_{j\in \mathcal{I}_{d_N}} C_{k,j}x_j|= |(CP_{\mathcal{I}_{d_N}}x)_k|\le \|CP_{\mathcal{I}_{d_N}}x\|\le \|C\| \cdot\|P_{\mathcal{I}_{d_N}}x\|\le  4\|C\|\cdot \sqrt {2d_N},
$$
where the spectral norm $\|C\|$ is independent of $N$.
By \eqref{eq:kclose_jfar},
\begin{flalign*}
|\xi_k |&= \sqrt{\frac{s_N}{N}}\abs{(C x)_k} + \mathcal{O}\Big( \sqrt{\frac{s_N}{N}}\Big) \\
&\le \sqrt{\frac{s_N}{N}}\Big(4\|C\| \sqrt{2d_N} + \mathcal{O}(N^{-5}s_N^{3/2})\Big) + \mathcal{O}\Big( \sqrt{\frac{s_N}{N}}\Big)  = \mathcal O \Big( \sqrt{\frac{s_N\log N} {N}}\Big)  \to 0.
\end{flalign*} 
Recall the definition $G(x) = \log \cosh (x) -\frac{x^2}{2}$ and therefore $G'(x) = \tanh(x) - x$. A Taylor expansion around zero yields that $G'(x) = -\frac{x^3}{3} + \mathcal{O}(x^5)$, hence for $N$ large enough
$$
G'\Big(\xi_k\Big)
= \mathcal O\Big(\frac{s_N\log N}{N} \Big)^{3/2} 
$$
uniformly in $k\in\mathcal I_{\widetilde{d}_N}$.
It follows 
\begin{flalign*}\frac{N}{s_N}\sum_{k\in \mathcal{I}_{\widetilde{d}_N}}G'(\xi_k)\sqrt{\frac{s_N}{N}} C_{\beta,\alpha, t} =  \mathcal O\Big(\frac{N}{s_N} \widetilde{d}_N \sqrt{\frac{s_N}{N}}\big(\frac{s_N\log N}{N} \big)^{3/2}\Big) =\mathcal{O}\Big( \frac{s_N}{N}(\log N)^{5/2} \Big) ,
\end{flalign*}
converging to zero for $s_N = o\left(\frac{N}{(\log N)^{5/2}} \right)$.
The claim follows from combining this with \eqref{eq:IrN}.

In the case where $s_N =\mathcal O(\log N)$, the statement of Lemma \ref{lem:G'=o} is easily proved, since essentially $\mathcal{I}_{\widetilde{d}_N}^c=\emptyset$. Indeed, for $x \in B_{4\sqrt{2d_N}}$, for every $k=1,\ldots, s_N$ we have that 
\begin{flalign}\label{s_Nsmall1}
   \left\vert\sqrt{\frac{s_N}{N}}(Cx)_k\right\vert \le \sqrt{\frac{s_N}{N}}\|C\|4\sqrt{2d_N} \le 8c_6\|C\|\frac{\log N}{\sqrt{N}} 
\end{flalign}
and by \eqref{eq:boundCt},
\begin{flalign}\label{s_Nsmall2}
    \left|\lambda_k \sqrt{\frac{s_N}{N}}\left( (I-A)^{-1} t\right)_k \right| \le C_{\beta,\alpha,t} \sqrt{\frac{s_N}{N}} \le C_{\beta,\alpha,t} \sqrt{2}c_6\sqrt{\frac{\log N}{N}}.
\end{flalign}
In particular, \eqref{s_Nsmall1} and \eqref{s_Nsmall2} imply that $|\xi_k| = \mathcal{O}\left( \frac{\log N}{\sqrt{N}}\right) \longrightarrow 0$ as $N\to \infty$. Using again Taylor expansion of $G'$ around zero and plugging in all the $\xi_k$ yields
\begin{flalign*}
   \left\vert \frac N {s_N}\sum_{k=1}^{s_N}G'\Big( \xi_k \Big)\sqrt{\frac{s_N}{N}}\Big( (I-A)^{-1} t\Big)_k \right\vert 
   \le& \sqrt{\frac{N}{s_N}} \|(I-A)^{-1}\|\|t\| \sum_{k=1}^{s_N} G'(\xi_k) \\
   =& \mathcal{O}\left( \sqrt{\frac{N}{s_N}} \left(\frac{\sqrt{s_N d_N}}{\sqrt{N}} \right)^3\right) \\
   =&\mathcal{O}\left( \frac{(\log N)^2}{N}\right) = o(1).
\end{flalign*}
This finishes the proof of Lemma \ref{lem:G'=o}.
\end{proof}

\subsection{Removing the Gaussian addition}\label{subsec:remove}
Finally, we are in the position to combine our ingredients for the
\begin{proof}[Proof of Theorem \ref{CLT}\,(i)]
Recall, we had shown in \eqref{eq:zwitscherndeszwischenergebnis} that
\begin{flalign*}
&\E\left[\exp\left\{ t^T (Z + \widetilde{m}) \right\}\right]\exp\left\{\frac{1}{2}t^T(A-A^2)^{-1}t \right\} \\
=&\frac{ \int_{\tilde D_N}
\exp\Big\{-\frac{1}{2} \|x\|^2 + \frac{N}{s_N}\sum_{k=1}^{s_N}
G\Big(\sqrt{\frac{s_N}{N}}(C x)_k + \sqrt{\frac{s_N}{N}}\Big((I-A)^{-1}t \Big)_k\Big)
\Big\}\d^{s_N}x }{ \int_{\tilde D_N}
\exp\Big\{
-\frac{1}{2} \|x\|^2
+ \frac{N}{s_N}\sum_{k=1}^{s_N}
G\Big(\sqrt{\frac{s_N}{N}}(C x)_k\Big)
\Big\}\d^{s_N}x }\big(1+o(1)\big).\nonumber
\end{flalign*}
which we claim to converge to $1$ as $N\to\infty$. By the Taylor approximation \eqref{eq:Taylorrr} and Lemma \ref{lem:G'=o}, we obtain
\begin{align*}
     &\int_{\tilde D_N}
\exp\Big\{-\frac{1}{2} \|x\|^2 + \frac{N}{s_N}\sum_{k=1}^{s_N}
G\Big(\sqrt{\frac{s_N}{N}}(C x)_k + \sqrt{\frac{s_N}{N}}\Big((I-A)^{-1}t \Big)_k\Big)
\Big\}\d^{s_N}x \\
=& \int_{\tilde D_N}
\exp\Big\{-\frac{1}{2} \|x\|^2 + \frac{N}{s_N}
\sum_{k=1}^{s_N}G\Big(\sqrt{\tfrac{s_N}{N}}(C x)_k\Big) + G'\Big( \xi_k \Big)\sqrt{\frac{s_N}{N}}\Big( (I-A)^{-1} t\Big)_k\Big\}\d^{s_N}x \\
=& \int_{\tilde D_N}
\exp\Big\{-\frac{1}{2} \|x\|^2 + \frac{N}{s_N}
\sum_{k=1}^{s_N}G\Big(\sqrt{\tfrac{s_N}{N}}(C x)_k\Big) + o(1)\Big\}\d^{s_N}x 
\end{align*}
Therefore,
$$
 \E\left[\exp\left\{ t^T (Z + \widetilde{m}) \right\}\right]\longrightarrow \exp \left\{ \frac{1}{2}t^T(A-A^2)^{-1}t \right\},
$$
as claimed.
where $A$ and $A^2$ are infinite matrices, if we view $t\in\R^\N$.
That means we have shown that all finite-dimensional marginals of $\mu_{N,\alpha, \beta}\circ \widetilde{m}^{-1} \star \mathcal{N}\left(0, A^{-1}\right)$ converge weakly to $\mathcal{N}(0, (A-A^2)^{-1})$ (with the respective dimension). It remains to reverse the contribution of the independent Gaussian variable $Z$. To that end, note that if we multiply the above by the inverse of the moment-generating function of the $\mathcal{N}(0, A^{-1})$-distribution, we obtain
\begin{align*}
\exp \left\{ \frac{1}{2}t^T(A-A^2)^{-1}t \right\}\exp \left\{  -\frac{1}{2}t^TA^{-1}t\right\} &= \exp \left\{\frac{1}{2}t^T( (A-A^2)^{-1}-A^{-1})t \right\} \\
&= \exp \left\{ \frac{1}{2}t^T(I-A)^{-1}t \right\},
\end{align*}
since 
\begin{align*}
(A-A^2)^{-1}-A^{-1} &= (A-A^2)^{-1}(I - (A-A^2)A^{-1}) \\
&= (A-A^2)^{-1}A=  (I-A)^{-1}.
\end{align*}
\end{proof}

\section{Proof of Theorem \ref{WLLN}\,(ii)}
The strategy for the law of large numbers in the low temperature regime is similar to the high temperature case. We will show that, for any fixed $\delta>0$, with high probability not all magnetizations are more than $\delta$ away from $m^*$, or from $-m^*$ respectively. Afterwards we will use this to show that the magnetization of the first block converges to $m^*$ or $-m^*$ in probability and deduce that this is true for all block magnetizations uniformly. The last step is to show that typically all block magnetizations have the same sign. The proofs are all very similar in spirit (to the high temperature case). The partition function will be bounded from below by choosing a suitable sub-collection of its summands (depending on the respective objective). In the numerator of the Gibbs measure we will use the Cauchy-Schwarz bound (\ref{CauchySchwarz}) on the relevant blocks and see that the ones that are far away cancel with the partition function. Throughout the whole chapter assume that $\beta + 2\alpha >1$ and denote by $m^* = m^*(\beta + 2 \alpha)$ the largest solution to the Curie-Weiss equation 
    $$
    \tanh((\beta + 2\alpha)m) =m.
    $$
\begin{lemma}\label{LT1}

Assume that $s_N = o(N)$, then for any $\eps>0$, there exists a positive constant $c_{\beta, \alpha}(\eps)$ such that
$$
\mu_{N, \beta, \alpha}\left( \inf_{k=1}^{s_N} \min \left\{|m_k -m^*|, |m_k + m^*| \right\} > \eps \right) \le \exp\left\{- N  c_{\beta,\alpha}(\eps)\right\} \left(\frac{N}{s_N} \right)^{s_N}
$$
and the right hand side converges to 0.
In other words, with high probability, not all blocks have a magnetization bounded away from $+m^*$ and from $-m^*$.  
\end{lemma}

\begin{proof}
As in the high temperature case, let us first find a lower bound for the partition function. To that end, recall that the magnetization of each block takes values in the set
$$
\mathcal{A}_N = \left\{ -1 + 2k\frac{s_N}{N} \, : \, k = 0, \ldots, \frac{N}{s_N} \right\}.
$$
Then
$$
Z_{N,\beta,\alpha} \ge \max_{\widetilde{m} \in \mathcal{A}_N} \exp\left\{ N \left(\frac{\beta + 2\alpha}{2}\widetilde{m}^2 - \log2 + s(\widetilde{m})\right) - V_N \right\},
$$
where we simply bound the sum in $Z_{N,\beta,\alpha}$ by the largest summand that satisfies that all entries in the magnetization vector coincide, and $V_N= s_N \left(\log \left( \sqrt{\frac{N(1-\widetilde{m}^2) \pi}{s_N }}\right)+ o(1) \right) $ denotes the lower order terms from Stirlings' formula. 

Recall the function $F_{\beta+2 \alpha}(m)$ from (\ref{F}) and further recall that from the usual Curie-Weiss model we know that in the low temperature regime, the maximizers of $F_{\beta+2 \alpha}$ on the interval $[-1,1]$ are given by $\pm m^*$
(see \cite{EllisEntropyLargeDeviationsAndStatisticalMechanics}). Since $F_{\beta+2 \alpha}$ is continuous and $\mathcal{A}_N$ is getting finer as $N$ becomes large, we obtain that for every $\delta>0$ there exists $\widetilde{N}\in \N$ such that for all $N \ge \widetilde{N}$, there exists $\widetilde{m} \in \mathcal{A}_N$ with
\begin{equation}\label{maximizer}
     F_{\beta+2 \alpha}(\widetilde{m}) \ge F_{\beta+2 \alpha}(m^*) - \delta.
\end{equation}
This implies that for all such $N$,
$$
Z_{N,\beta,\alpha} \ge \exp \left\{ N ( F_{\beta+2 \alpha}(m^*) - \delta) - V_N  \right\}.
$$
For fixed $\eps>0$, denote by $A_\eps$ the union of the balls  $B_\eps(m^*)$ and $ B_\eps(-m^*)$ of radius $\eps$ centered in $m^*$ and $-m^*$, resp., i.e.\ $A_\eps = B_\eps(m^*) \cap B_\eps (-m^*)$. By the above observation and  by using the Cauchy-Schwarz bound on all $k$ we obtain that
\begin{flalign*}
    &\mu_{N,\beta,\alpha}\left( \inf_{k=1}^{s_N} \min \left\{|m_k -m^*|, |m_k + m^*| \right\} > \eps \right) \\
    =& \mu_{N,\beta, \alpha}\left( m_1, \ldots, m_{s_N}\in A_\eps^c \right)\\
    \le & \exp \left\{ -N F_{\beta+2 \alpha}(m^*) + N \delta + V_N\right\} \sum_{m_1, \ldots, m_{s_N}} \1_{\left\{ m_1, \ldots, m_{s_N} \in A_\eps ^c \right\}} \exp\left\{ \frac{1}{2} \frac{N}{s_N} m^TAm \right\} \frac{1}{2^N} \prod_{k=1}^{s_N} { \frac{N}{s_N} \choose \frac{N}{s_N} \frac{1+m_k}{2}} \\
    \le & \exp \left\{ -N F_{\beta+2 \alpha}(m^*) + N \delta +V_N\right\} \sum_{m_1, \ldots, m_{s_N}} \exp\left\{ \frac{N}{s_N} \sum_{k=1}^{s_N}\left(\frac{\beta + 2\alpha}{2}m_k^2 -\log2 +s(m_k)\right) \right\}\1_{\left\{m_1, \ldots, m_{s_N} \in A_\eps^c  \right\}}  \\
    =& \exp \left\{ -N F_{\beta+2\alpha}(m^*) + N \delta +V_N \right\} \sum_{m_1, \ldots, m_{s_N}} \exp\left\{ \frac{N}{s_N} \sum_{k=1}^{s_N} F_{\beta+2\alpha}(m_k)\right\}\1_{\left\{ m_1, \ldots, m_{s_N} \in A_\eps^c \right\}} \\
    =& \exp\left\{ N\left(\delta+ \frac{V_N}{N} \right)\right\} \prod_{k=1}^{s_N}\left( \sum_{m_k} \exp\left\{ \frac{N}{s_N} (F_{\beta+2\alpha}(m_k)-F_{\beta+2\alpha}(m^*))\right\} \1_{\{m_k \in A_\eps^c\}}\right)
\end{flalign*}
Note that again this is the $s_N$-fold product of independent Curie-Weiss models, this time in the low temperature regime. It is well-known that $F_{\beta+2\alpha}$ is maximal at $\pm m^*$ and $F_{\beta+2\alpha}(m)-F_{\beta+2\alpha}(m^*)$ is negative everywhere else and decreasing in the distance of $m$ and $m^*$ if $m$ is non-negative and in the distance of $m$ and $-m^*$ if $m$ is non-positive. This means, that on the event where $m$ is bounded away from both $m^*$ and $-m^*$ there exists a positive constant $c_{\beta, \alpha}(\eps)$ (depending on $\alpha, \beta$ and $\eps$) such that if we choose $\delta$ small enough,
$$
F_{\beta+2\alpha}(m)-F_{\beta+2\alpha}(m^*) \le - \left(c_{\beta, \alpha}(\eps) + \delta + \frac{V_N}{N} \right).
$$
It follows that 
$$
\mu_{N,\beta,\alpha}\left( \inf_{k=1}^{s_N} \min \left\{|m_k -m^*|, |m_k + m^*| \right\} > \eps \right) \le \exp\left\{ -N  c_{\beta,\alpha}(\eps)\right\} \left(\frac{N}{s_N} \right)^{s_N} 
$$
implying the statement.
\end{proof}

\begin{lemma} Assume that $s_N = o\left( \frac{N}{\log N}\right)$. Let $\eps >0$, then for $c_{\beta, \alpha}(\eps)$ from Lemma \ref{LT1}, it is true that
    $$
    \mu_{N,\beta,\alpha}( \min \left\{|m_1 - m^*|, |m_1 + m^*|\right\}> \eps) \le  2\exp\left\{ -\frac{N}{s_N}\frac{c_{\beta, \alpha}(\eps)}{4}\right\}
    $$
\end{lemma}

\begin{proof}

For a fixed $\widetilde{\delta}>0$ (small enough and to be chosen later), we will define random variables $\widetilde{R}_N$ and $\widetilde{L}_N$ by
\begin{flalign*}
    &\widetilde{R}_N = \inf \left\{ i=1, \ldots, s_N \, :\, \min \left\{|m_i -m^*|, |m_i + m^*| \right\} \le \widetilde{\delta} \right\} \\
    & \widetilde{L}_N = \inf \left\{ i=1, \ldots, s_N \, :\, \min \left\{|m_{s_N-i} -m^*|, |m_{s_N-i} + m^*| \right\} \le \widetilde{\delta} \right\}.
\end{flalign*}
By Lemma \ref{LT1} we know that the event $2 \le \widetilde{R}_N < s_N-\widetilde{L}_N \le s_N$ occurs with high probability and hence we will estimate
\begin{flalign} \label{fixRandL}
&\mu_{N,\beta, \alpha} \left(\min \left\{|m_1 - m^*|, |m_1 + m^*|\right\}> \eps, 2 \le \widetilde{R}_N < s_N-\widetilde{L}_N \le s_N\right) \nonumber \\
\le & \sum_{2\le r <s_N-l\le s_N}\mu_{N, \beta, \alpha} \left(\min \left\{|m_1 - m^*|, |m_1 + m^*|\right\}> \eps, \, \widetilde{R}_N=r, \widetilde{L}_N = l \right).
\end{flalign}
For fixed $r,l$, we define the set $\mathcal{A}_{r,l}\coloneqq \left\{ \widetilde{R}_N =r, \widetilde{L}_N = l \right\}$. More precisely, we will further divide these sets into the sets
$$
\mathcal{A}_{r,l}^{+,+} \coloneqq \left\{ \widetilde{R}_N =r, \widetilde{L}_N = l , m_r>0, m_{s_N-l}>0\right\},
$$
and $\mathcal{A}_{r,l}^{-,-}, \mathcal{A}_{r,l}^{+,-}, \mathcal{A}_{r,l}^{-,+}$ defined similarly. On the event that $\mathcal{A}_{r,l}^{+,+}$ occurs, by definition it is true that
$$
 0\le | m_r-m^*|, |m_{s_N-l}-m^*| \le \widetilde{\delta} \quad \text{and} \quad |m_k- (\pm m^*)| > \widetilde{\delta } \quad \forall k=2,\ldots, r-1,s_N -l +1, \ldots, s_N.
$$
In order to bound the partition function, we will again partially bound the $s_N$-fold sum by carefully chosen summands. More precisely, let us define $\mathcal{I}_{l,r} \coloneqq \{1,\ldots, r, s_N-l,\ldots, s_N\}$, then for $k\in \mathcal{I}_{l,r}$, 
we bound the sum over $m_k$ by the summand $m_k=\widetilde{m}$, where $\widetilde{m}\in \mathcal{A}_N$ maximizes $F_{\beta,\alpha}$ on $\mathcal{A}_N$. Hence we obtain
\begin{flalign*}
    Z_{N,\beta, \alpha} &= \sum_{m_1, \ldots, m_{s_N}} \exp\left\{ \frac{1}{2}\frac{N}{s_N} \sum_{k=1}^{s_N}\beta m_k^2 + \alpha m_k(m_{k-1}+m_{k+1}) \right\} \frac{1}{2^N} \prod_{k=1}^{s_N} { \frac{N}{s_N} \choose \frac{N}{s_N}\frac{1+m_k}{2} } \\ 
    & \ge \exp\left\{\frac{1}{2}\frac{N}{s_N}\sum_{ k\in \mathcal{I}_{l-1,r-1}} (\beta + 2\alpha)\widetilde{m}^2 \right\}\exp\left\{\frac{1}{2}\frac{N}{s_N}(\beta+ \alpha) 2 \widetilde{m}^2 \right\} \prod_{\mathcal{I}_{l,r}}\frac{1}{2^{\frac{N}{s_N}}}{ \frac{N}{s_N} \choose \frac{N}{s_N}\frac{1+\widetilde{m}}{2}} \\
    & \quad \times\sum_{m_{r+1}, \ldots, m_{s_N-l-1}}\exp\left\{ \frac{1}{2}\frac{N}{s_N}2\alpha \widetilde{m}(m_{r+1} + m_{s_N-l-1})\right\} \exp\left\{ \frac{1}{2}\frac{N}{s_N}\sum_{k=r+1}^{s_N-l-1}\beta m_k^2 \right\} \\
    & \quad \times \exp\left\{ \frac{1}{2}\frac{N}{s_N}\sum_{k=r+2}^{s_N-l-2}\alpha m_k (m_{k-1}+m_{k+1}) \right\}  \prod_{k=r+1}^{s_N-l-1}\frac{1}{2^{\frac{N}{s_N}}} {\frac{N}{s_N}\choose \frac{N}{s_N}\frac{1+m_k}{2}}\\
    &\ge \exp\left\{ \frac{N}{s_N} (r+l-1) F_{\beta+2\alpha}(\widetilde{m}) + 2\frac{N}{s_N} F_{\beta + \alpha}(\widetilde{m}) - V_{N,r,l} \right\} \exp\left\{ \frac{N}{s_N} \alpha2 (\widetilde{m}- m^*)\right\}\widetilde{Z}_N^{r,l},
\end{flalign*}
where $V_{N,r,l} = (r+l+1) \left(\log\left( \sqrt{\frac{N (1-\widetilde{m}^2) \pi}{s_N}}\right)+ o(1) \right)$ denotes again the lower order terms from Stirlings' formula, and
\begin{flalign*}
\widetilde{Z}_N^{r,l} =& \sum_{m_r,\ldots, m_{s_N-l}} \exp\left\{\frac{1}{2}\frac{N}{s_N}\left(2\alpha m^*(m_{r+1} + m_{s_N-l-1}) +\sum_{k=r+1}^{s_N-l-1}\beta m_k^2 + \sum_{k=r+2}^{s_n-l-2}\alpha m_k(m_{k-1} + m_{k+1})  \right) \right\} \\
& \quad \times \prod_{k=r}^{s_N-l}\frac{1}{2^{\frac{N}{s_N}}} {\frac{N}{s_N}\choose \frac{N}{s_N}\frac{1+m_k}{2}} \\
&\eqqcolon \sum_{m_r,\ldots, m_{s_N-l}} \exp\left\{\widetilde{H}_N^{r,l} \right\}\prod_{k=r}^{s_N-l}\frac{1}{2^{\frac{N}{s_N}}} {\frac{N}{s_N}\choose \frac{N}{s_N}\frac{1+m_k}{2}},
\end{flalign*}
i.e.\ $\widetilde{Z}_N^{r,l}$ is the partition function of a similar model as the original one but with only the blocks $m_{r+1}, \ldots, m_{s_N-l-1}$ and with $2m^*$-boundary conditions instead of periodic boundary conditions. Recall from (\ref{maximizer}) that $F_{\beta + 2\alpha}(\widetilde{m}) \le F_{\beta + 2\alpha}(m^*) + \delta$ implying that
\begin{equation}\label{Z_rl}
     Z_{N,\beta, \alpha} \ge  \exp\left\{ \frac{N}{s_N} (r+l-1) (F_{\beta+2\alpha}(m^*) - \delta) \right\} \exp\left\{ 2\frac{N}{s_N} F_{\beta + \alpha}(\widetilde{m}) + V_{N,r,l} \right\} \exp\left\{- \frac{N}{s_N} \alpha2 |\widetilde{m}- m^*| \right\}\widetilde{Z}_N^{r,l}
\end{equation}
We will plug this into the Gibbs measure later but first we will find the right upper bound for the Hamiltonian. To that end, we will apply the Cauchy-Schwarz bound $m_km_{k+1} \le \frac{m_k^2}{2} + \frac{m_{k+1}^2}{2}$ for $k=1,\ldots,r-1, \ldots, s_N-l,\ldots, s_N$ and obtain
\begin{flalign*}
   \sum_{k=1}^{s_N} \beta m_k^2 + \alpha m_k(m_{k+1} + m_{k-1}) &\le \sum_{k\in \mathcal{I}_{l-1,r-1}} (\beta + 2\alpha)m_k^2  + \sum_{k=r+1}^{s_N-l-1}\beta m_k^2 + \sum_{k=r+2}^{s_N-l-2} \alpha m_k(m_{k-1} + m_{k+1})  \\
    & \quad + (\beta + \alpha)(m_r^2 + m_{s_N-l}^2) + 2\alpha m_rm_{r+1} + 2\alpha m_{s_N-l-1}m_{s_N-l} 
\end{flalign*}
Using the fact that on the set where $|m_r- m^*|,|m_{s_N-l}- m^*| <\widetilde{\delta}$ it holds that 
$$
2\alpha ( m_r m_{r+1} + m_{s_N-l-1}m_{s_N-l}) \le 2\alpha m^*(m_{r+1} + m_{s_N-l-1}) + 4\alpha\widetilde{\delta} 
$$
and therefore
\begin{flalign}\label{H_rl}
    &\frac{1}{2}\frac{N}{s_N}\sum_{k=1}^{s_N} \beta m_k^2 + \alpha m_k(m_{k+1} + m_{k-1}) \nonumber \\
    \le& \frac{1}{2}\frac{N}{s_N}\sum_{k\in \mathcal{I}_{l-1,r-1}} (\beta + 2\alpha)m_k^2  + \widetilde{H}_N^{r,l} + \frac{N}{s_N}2\alpha \widetilde{\delta} + \frac{1}{2}\frac{N}{s_N}(\beta + \alpha)(m_r^2 + m_{s_N-l}^2)
\end{flalign}
If we now plug (\ref{Z_rl}) and (\ref{H_rl}) into the Gibbs measure, we see that the measure factorizes for all coordinates between $s_N-l$ and $r$ while for the remaining (still interacting) coordinates, the numerator and denominator coincide and therefore cancel. More precisely, by rearranging the terms from (\ref{Z_rl}) and (\ref{H_rl}), one gets
\begin{flalign*}
    &\mu_{N, \beta, \alpha}\left( \min\left\{|m_1 - m^*|,|m_1 + m^*|\right\}> \eps, \,  \mathcal{A}_{r,l}^{++}\right) \\
    \le& \exp\left\{\frac{N}{s_N}(r+l-1)\delta  + V_{N,r,l} \right\} \exp\left\{ \frac{N}{s_N}2\alpha|m^* - \widetilde{m}| \right\}\frac{1}{\widetilde{Z}_N^{r,l}}\\
    &\quad \times \left(\sum_{m_r} \1_{|m_r-m^*| \le \widetilde{\delta}} \exp\left\{ \frac{N}{s_N}(F_{\beta + \alpha}(m_r) - F_{\beta + \alpha}(\widetilde{m}))\right\} \right)^2\\
    & \quad \times\left(\sum_{m_1} \1_{\{   \min\left\{|m_1 - m^*|,|m_1 + m^*|\right\}> \eps\}} \exp \left\{ \frac{N}{s_N} \left(F_{\beta+2 \alpha}(m_1)-F_{\beta +2 \alpha}(m^*)\right) \right\} \right) \\
    & \quad \times \left( \sum_{m_k} \1_{\{ \min \left\{ |m_k - m^*|, |m_k + m^*|\right\}>\widetilde{\delta}\}} \exp\left\{ \frac{N}{s_N} \left(F_{\beta + 2\alpha}(m_k)-F_{\beta + 2\alpha}(m^*)\right)\right\}\right)^{r+l-2} \widetilde{Z}_N^{r,l} \\
    \le & \exp\left\{\frac{N}{s_N}(r+l-1)\delta + V_{N,r,l} \right\} \exp\left\{ \frac{N}{s_N}2\alpha|m^* - \widetilde{m}|\right\} \left(\sum_{m_r} \1_{|m_r-m^*| \le \widetilde{\delta}} \exp\left\{ \frac{N}{s_N}(F_{\beta + \alpha}(m_r) - F_{\beta + \alpha}(\widetilde{m}))\right\} \right)^2 \\
     & \quad \times  \exp\left\{- \frac{N}{s_N} c_{\beta,\alpha}(\eps)\right\} \exp\left\{-\frac{N}{s_N}(r+l-2) c_{\beta, \alpha}(\widetilde{\delta}) \right\} \left( \frac{N}{s_N}\right)^{r+l-1},
\end{flalign*}
where the last step follows (as in the proof of Lemma \ref{LT1}) since we have reduced the model to $r-l-1$ many independent Curie-Weiss models and $c_{\beta,\alpha}(\eps), c_{\beta, \alpha }(\widetilde{\delta})$ are the positive constants from Lemma \ref{LT1}. Further note that the distance between $m^*$ and $\widetilde{m}$ is actually quite small. Indeed, since $\widetilde{m}$ maximizes the continuous function $F_{\beta + 2\alpha}$ on the set $\mathcal{A}_N$, which is an approximation of the interval $[-1,1]$ with spacing of size $2\frac{s_N}{N}$, in particular this implies that $|m^* - \widetilde{m}|<\frac{s_N}{N}$. It remains to find a bound for the contribution of the $r$'th and $(s_N-l)$'th coordinates. Since $m^* = m^*(\beta + 2\alpha)$ maximizes $F_{\beta + 2\alpha}$ but not $F_{\beta + \alpha}$, we cannot apply the same arguments. Instead note that still 
$$
|\widetilde{m}-m_r|\le|\widetilde{m}- m^*| + |m^*- m_r| = \mathcal{O}\left(\frac{s_N}{N}\right) + \widetilde{\delta}
$$
and the function $F_{\beta + \alpha}$ is continuous, there exists a small constant $c(\widetilde{\delta})$ that goes to zero as $\widetilde{\delta}\to 0$, such that $|F_{\beta + \alpha}(m_r) - F_{\beta + \alpha}(\widetilde{m})| \le c(\widetilde{\delta})$. If we then bound the sum by the number of summands it follows that
$$
\left(\sum_{m_r} \1_{|m_r-m^*| \le \widetilde{\delta}} \exp\left\{ \frac{N}{s_N}(F_{\beta + \alpha}(m_r) - F_{\beta + \alpha}(\widetilde{m}))\right\} \right) \le \exp\left\{ \frac{N}{s_N} c(\widetilde{\delta})\right\} \widetilde{\delta}\frac{N}{s_N}
$$
Hence, we finally obtain the following upper bound
\begin{flalign*}
    &\mu_{N, \beta, \alpha}\left( \min\left\{|m_1 - m^*|,|m_1 + m^*|\right\}> \eps, \,  \mathcal{A}_{r,l}^{++}\right) \\
    \le & \widetilde{\delta}^2\exp\left\{\frac{N}{s_N}(r+l-1)\delta + V_{N,r,l} \right\}\exp\left\{ \frac{N}{s_N} c(\widetilde{\delta})\right\}\exp\left\{- \frac{N}{s_N} c_{\beta,\alpha}(\eps)\right\} \\
    & \qquad \qquad \times \exp\left\{-\frac{N}{s_N}(r+l-2) c_{\beta, \alpha} (\widetilde{\delta}) \right\}\left( \frac{N}{s_N}\right)^{r+l+1} \\
    & \le \exp\left\{ -\frac{N}{s_N}\left( c_{\beta, \alpha}(\eps) - c(\widetilde{\delta}) - 2\frac{s_N}{N}\log\frac{N}{s_N} \right) \right\} \\
    & \qquad \qquad \times \exp\left\{ -(r+l-1)\frac{N}{s_N}\left( \frac{r+l-2}{r+l-1}c_{\beta, \alpha}(\widetilde{\delta}) - \delta - \frac{s_N}{N}\left(\log\frac{N}{s_N} + V_{N,r,l}\right) \right) \right\}
\end{flalign*}
Now we want to choose $\widetilde{\delta}$ and $\delta$ small enough such that the above converges to 0. This is indeed possible: Choose first $\widetilde{\delta}$ small enough such that $c(\widetilde{\delta}) < \frac{c_{\beta, \alpha}(\eps)}{2}$ and note that for $N$ large enough it is true that $2\frac{s_N}{N}\left(\log\frac{N}{s_N}+ V_{N,r,l} \right) < \frac{c_{\beta,\alpha}(\eps)}{4}$. Then choose $\delta < \frac{c_{\beta, \alpha}(\widetilde{\delta})}{4}$ and again for $N$ large enough we have that $\frac{s_N}{N}\log \frac{N}{s_N} < \frac{c_{\beta , \alpha}(\widetilde{\delta})}{8}$. It follw that
\begin{flalign} \label{eq:A++}
&\mu_{N, \beta, \alpha}\left( \min\left\{|m_1 - m^*|,|m_1 + m^*|\right\}> \eps, \,  \mathcal{A}_{r,l}^{++}\right) \nonumber \\
\le& \exp\left\{ -\frac{N}{s_N} \frac{c_{\beta, \alpha}(\eps)}{4} \right\} \exp\left\{ -(r+l-1)\frac{N}{s_N} \frac{c_{\beta, \alpha}(\widetilde{\delta})}{8} \right\}
\end{flalign}

Further note that we get the same bound for 
$$
\mu_{N, \beta, \alpha}\left( \min\left\{|m_1 - m^*|,|m_1 + m^*|\right\}> \eps, \,  \mathcal{A}_{r,l}^{--}\right)
$$
by replacing the sums in $Z_{N,\beta,\alpha}$ by the summand $m_k=-\widetilde{m}$. On the events $\mathcal{A}_{r,l}^{+-}$ and $\mathcal{A}_{r,l}^{-+}$, we know that there are (at least) two sign changes in the chain of blocks. If we flip the entire negative component that includes $m_{s_N-l}$ or $m_r$ respectively, then the energy on the boundary of this component increases (because the spins are now aligned) while the energy in the interior remains the same. Note that flipping all the coordinates brings us back to the cases that were already treated. More precisely, on the set $\mathcal{A}_{r,l}^{+-}\cap \left\{ m_1>0 \right\}$, there must exist $l_1, l_2$ with $0\le l_1 \le l \le l_2 < s_N-r$ such that $m_{s_N-l_1}, m_{s_N-l_2}\le 0$. Define the random variables
\begin{flalign*}
    &L_N^1 \coloneqq \inf\{\widetilde{l}\ge 0\, :\, m_{s_N-\widetilde{l} }\le 0\} \\
    &L_N^2 \coloneqq \inf \{ \widetilde{l}\ge l \, : \, m_{s_N - \widetilde{l}}\le 0\},
\end{flalign*}
i.e.\ $S_{s_N-L_N^1}$ is the closest block to $S_1$ with a non-positive magnetization and $S_{s_N-L_N^2}$ is the last block after $S_{s_N-l}$ with a non-positive magnetization. For $N\in \N$ and fixed $l_1,l_2$, denote by $f_{l_1,l_2}: [-1,1]^{s_N} \to [-1,1]^{s_N}$ the function that takes a vector and flips all coordinates between $s_N-l_2$ and $s_N - l_1$ while leaving all other coordinates unchanged. Then
\begin{flalign*}
    &\mu_{N,\beta, \alpha}\left( \min\{ |m_1-m^*|,|m_1+m^*|\}>\eps, \, m_1\ge 0, \, \mathcal{A}_{r,l}^{+-}, \, L_N^1=l_1, \, L_N^2 =l_2 \right) \\
    =& \frac{1}{Z_{N,\beta, \alpha}}\sum_{m_1, \ldots, m_{s_N}} \exp \left\{\frac{1}{2}\frac{N}{s_N}m^TAm \right\} \1_{\{\min\{ |m_1-m^*|,|m_1+m^*|\}>\eps, \, m_1\ge 0\}} \1_{\mathcal{A}_{r,l}^{+-}} \\
    & \quad \times \1_{\{L_N^1=l_l^1 , \, L_N^2 =l_2\}}\frac{1}{2^N}\prod_{k=1}^{s_N}{ \frac{N}{s_N} \choose \frac{N}{s_N}\frac{1+m_k}{2} }\\
    =& \frac{1}{Z_{N,\beta, \alpha}}\sum_{m_1, \ldots, m_{s_N}} \exp \left\{\frac{1}{2}\frac{N}{s_N}f_{l_1, l_2}(m)^TAf_{l_1,l_2}(m) \right\} \1_{\{\min\{ |m_1-m^*|,|m_1+m^*|\}>\eps, \, m_1\ge 0\}} \1_{\mathcal{A}_{r,l}^{+-}} \\
    & \quad \times \1_{\{L_N^1=l_l^1 , \, L_N^2 =l_2\}}\frac{1}{2^N}\prod_{k=1}^{s_N}{ \frac{N}{s_N} \choose \frac{N}{s_N}\frac{1+m_k}{2} }.
\end{flalign*}
Note that 
$$
 \1_{\{\min\{ |m_1-m^*|,|m_1+m^*|\}>\eps, \, m_1\ge 0\}} =  \1_{\{\min\{ |f_{l_1,l_2}(m)_1-m^*|,|f_{l_1,l_2}(m)_1+m^*|\}>\eps, \, f_{l_1,l_2}(m)_1\ge 0\}},
$$
since $1 \notin  \{s_N-l_2,\ldots,  s_N-l_1\}$ and therefore $m_1 = f_{l_1,l_2}(m)_1$. The same is true for all coordinates $k \in \mathcal{I}_{l_1,r}= \{1,\ldots,r,s_N-l_1+1, \ldots, s_N\}$. On all other coordinates it changes the sign hence it follows that on the set $\{L_N^1=l_1\, , \, L_N^2=l_2\}$,
\begin{flalign*}
\1_{\mathcal{A}_{r,l}^{+-}} &\overset{\textbf{def}}{=} \1_{\{ |m_r-m^*|< \widetilde{\delta}\}} \1_{\{ |m_{s_N-l}+m^*|< \widetilde{\delta}\}} \prod_{k \in \mathcal{I}_{l-1,r-1}} \1_{\{ \min\{|m_k - m^*|, |m_k+ m^*|\}\ge\widetilde{\delta}\}} \\
     &  =  \1_{\{ |f_{l_1,l_2}(m)_r-m^*|< \widetilde{\delta}\}} \1_{\{ |f_{l_1,l_2}(m)_{s_N-l}-m^*|< \widetilde{\delta}\}} \prod_{{k \in \mathcal{I}_{l-1,r-1}}} \1_{\{ \min\{|f_{l_1,l_2}(m)_k - m^*|, |f_{l_1,l_2}(m)_k+ m^*|\}\ge\widetilde{\delta}\}} \\
     &  = \1_{\mathcal{A}_{r,l}^{++}}.
\end{flalign*}
Plugging this into the above equation yields
\begin{flalign*}
     &\mu_{N,\beta, \alpha}\left( \min\{ |m_1-m^*|,|m_1+m^*|\}>\eps, \, m_1\ge 0, \, \mathcal{A}_{r,l}^{+-}, \, L_N^1=l_1, \, L_N^2 =l_2 \right) \\
     =& \frac{1}{Z_{N,\beta, \alpha}}\sum_{m_1, \ldots, m_{s_N}} \exp \left\{\frac{1}{2}\frac{N}{s_N}f_{l_1, l_2}(m)^TAf_{l_1,l_2}(m) \right\} \1_{\{\min\{ |f_{l_1,l_2}(m)_1-m^*|,|f_{l_1,l_2}(m)_1+m^*|\}>\eps, \, f_{l_1,l_2}(m)_1\ge 0\}} \\
    & \quad \times  \1_{\mathcal{A}_{r,l}^{++}}\1_{\{L_N^1=l_l^1 , \, L_N^2 =l_2\}}\frac{1}{2^N}\prod_{k=1}^{s_N}{ \frac{N}{s_N} \choose \frac{N}{s_N}\frac{1+f_{l_1,l_2}(m)_k}{2} } \\
    \le& \frac{1}{Z_{N,\beta, \alpha}}\sum_{m_1, \ldots, m_{s_N}} \exp \left\{\frac{1}{2}\frac{N}{s_N}m^TAm \right\} \1_{\{\min\{ |m_1-m^*|,|m_1+m^*|\}>\eps, \, m_1\ge 0\}} \\
    &\quad \times  \1_{\mathcal{A}_{r,l}^{++}}\frac{1}{2^N}\prod_{k=1}^{s_N}{ \frac{N}{s_N} \choose \frac{N}{s_N}\frac{1+m_k}{2} } \\
    &= \mu_{N,\alpha, \beta}\left( \min\{ |m_1-m^*|,|m_1+m^*|\}>\eps, \, m_1\ge 0, \, \mathcal{A}_{r,l}^{++}\right).
\end{flalign*}
Therefore
\begin{flalign*}
    &\mu_{N,\beta, \alpha}\left( \min\{ |m_1-m^*|,|m_1+m^*|\}>\eps, \, m_1\ge 0, \, \mathcal{A}_{r,l}^{+-},  \right) \\
    =&\sum_{0\le l_1 \le l_2 \le s_N}\mu_{N,\beta, \alpha}\left( \min\{ |m_1-m^*|,|m_1+m^*|\}>\eps, \, m_1\ge 0, \, \mathcal{A}_{r,l}^{+-}, \, L_N^1=l_1, \, L_N^2 =l_2 \right) \\
    \le& s_N^2 \mu_{N,\alpha, \beta}\left( \min\{ |m_1-m^*|,|m_1+m^*|\}>\eps, \, m_1\ge 0, \, \mathcal{A}_{r,l}^{++}\right)
\end{flalign*}
and for the case where $m_1\le 0$, one gets the same result by using $\mathcal{A}_{r,l}^{--}$. Putting all these cases together, we arrive at the bound
\begin{flalign*}
    &\mu_{N,\beta,\alpha}\left(  \min\{ |m_1-m^*|,|m_1+m^*|\}>\eps, \, \mathcal{A}_{r,l}\right) \\
    \le &2(s_N^2 + 1)\mu_{N,\beta,\alpha}\left(  \min\{ |m_1-m^*|,|m_1+m^*|\}>\eps, \, \mathcal{A}_{r,l}^{++}\right) \\
    \le & 2(s_N^2+1) \exp\left\{ -\frac{N}{s_N} \frac{c_{\beta, \alpha}(\eps)}{4} \right\} \exp\left\{ -(r+l-1)\frac{N}{s_N} \frac{c_{\beta, \alpha}(\widetilde{\delta})}{8} \right\} \\
    \le & 4 \exp\left\{ -\frac{N}{s_N} \frac{c_{\beta, \alpha}(\eps)}{4} \right\} \exp\left\{ -\frac{N}{s_N}\left((r+l-1) \frac{c_{\beta, \alpha}(\widetilde{\delta})}{8}- 2 \frac{s_N}{N}\log s_N \right) \right\}
\end{flalign*}
Since the above upper bound is decreasing in $r+l$, we can simply use the uniform bound $r+l-1 \ge 1$. Hence, if we plug this into (\ref{fixRandL}), we obtain
\begin{flalign*}
    &\mu_{N,\beta, \alpha} \left(\min \left\{|m_1 - m^*|, |m_1 + m^*|\right\}> \eps, 1 <\widetilde{R}_N < s_N-\widetilde{L}_N \le s_N\right)  \\
    \le & s_N^2 4 \exp\left\{ -\frac{N}{s_N} \frac{c_{\beta, \alpha}(\eps)}{4} \right\} \exp\left\{ -\frac{N}{s_N}\left( \frac{c_{\beta, \alpha}(\widetilde{\delta})}{8}- 2 \frac{s_N}{N}\log s_N \right) \right\} \\
    \le & 4 \exp\left\{ -\frac{N}{s_N} \frac{c_{\beta, \alpha}(\eps)}{4} \right\} \exp\left\{ -\frac{N}{s_N}\left( \frac{c_{\beta, \alpha}(\widetilde{\delta})}{8}- 4 \frac{s_N}{N}\log s_N \right) \right\},
\end{flalign*}
where we bounded the sum running over $2\le r \le s_N-1$ and $1\le l\le s_N-2$ each by a factor $s_N$. Note that since we assumed that $s_N = o\left(\frac{N}{\log N} \right)$, again for $N$ large enough $4\frac{s_N}{N} \log s_N < \frac{c_{\beta,\alpha}(\widetilde{\delta})}{16}$ and finally this yields
\begin{flalign*}
&\mu_{N,\beta, \alpha} \left(\min \left\{|m_1 - m^*|, |m_1 + m^*|\right\}> \eps \right) \\
\le&  \exp\left\{ -\frac{N}{s_N}\left(\frac{c_{\beta, \alpha}(\eps)}{4} + \mathcal{O}(1) \right)\right\} + \exp \left\{-N \left(c_{\beta, \alpha}(\widetilde{\delta}) + o(1) \right)\right\}
\end{flalign*}

\end{proof}

By a union bound it follows immediately that $\sup_{i=1}^{s_N} \min\left\{ |m_i-m^*|, |m_i+m^*| \right\}$ converges to 0 in probability, since
$$
\mu_{N\beta, \alpha} \left(\sup_{i=1}^{s_N} \min\left\{ |m_i-m^*|, |m_i+m^*| \right\}> \eps \right) \le s_N \mu_{N\beta, \alpha} \left(\min\left\{ |m_1-m^*|, |m_1+m^*| \right\}> \eps\right).
$$
It is easy to deduce that in this case, all $m_i's$ have to have the same sign with high probability. 
\begin{lemma}
    $$
    \mu_{N,\beta,\alpha}\left( \sup_{k=1}^{s_N} \min\left\{ |m_k-m^*|, |m_k+m^*| \right\}\le \eps, \exists i : m_im_{i+1}<0 \right) \le s_N^2\exp\left\{ -\frac{N}{s_N} 2 \alpha (m^*-\eps)^2 \right\}
    $$
\end{lemma}

\begin{proof}
First note that if there exists an index $i$ such that $m_i$ and $m_{i+1}$ don't have the same sign, then there must exists another index $j\neq i$ such that the same is true for $m_jm_{j+1}$. Hence we obtain
\begin{flalign*}
    &\mu_{N,\beta,\alpha}\left( \sup_{k=1}^{s_N} \min\left\{ |m_k-m^*|, |m_k+m^*| \right\}\le \eps, \exists i : m_im_{i+1}<0 \right) \\
    \le & \sum_{i=1}^{s_N} \mu_{N,\beta,\alpha}\left( \sup_{k=1}^{s_N} \min\left\{ |m_k-m^*|, |m_k+m^*| \right\}\le \eps, m_im_{i+1}<0 \right)\\
    = & \sum_{i=1}^{s_N} \mu_{N,\beta,\alpha}\left( \sup_{k=1}^{s_N} \min\left\{ |m_k-m^*|, |m_k+m^*| \right\}\le \eps, m_im_{i+1}<0, \exists j\neq i : m_jm_{j+1}<0 \right) \\
    \le &\sum_{i=1}^{s_N}\sum_{j\neq i} \mu_{N,\beta,\alpha}\left( \sup_{k=1}^{s_N} \min\left\{ |m_k-m^*|, |m_k+m^*| \right\}\le \eps, m_im_{i+1}<0,  m_jm_{j+1}<0 \right) \\
    \le & s_N^2 \max_{i\neq j}  \mu_{N,\beta,\alpha}\left( \sup_{k=1}^{s_N} \min\left\{ |m_k-m^*|, |m_k+m^*| \right\}\le \eps, m_im_{i+1}<0,  m_jm_{j+1}<0 \right)
\end{flalign*}
Fix $i$ and $j$ and let us give again a lower bound for the partition function:
\begin{flalign*}
    Z_{N,\beta,\alpha} &= \sum_{m_1, \ldots, m_{s_N}} \exp\left\{ \frac{1}{2}\frac{N}{s_N}\sum_{k=1}^{s_N} \beta m_k^2 + 2\alpha m_km_{k+1} \right\}\frac{1}{2^N}\prod_{k=1}^{s_N}{ \frac{N}{s_N} \choose \frac{N}{s_N}\frac{1+m_k}{2}} \\
    &\ge \sum_{m_1, \ldots, m_{s_N}} \1_{m_im_{i+1}<0}\1_{m_jm_{j+1}<0} \exp\left\{ \frac{1}{2}\frac{N}{s_N}\sum_{\substack{k=1 \\k\neq i,j}}^{s_N} \beta m_k^2 + 2\alpha m_km_{k+1} \right\}\frac{1}{2^N}\prod_{k=1}^{s_N}{ \frac{N}{s_N} \choose \frac{N}{s_N}\frac{1+m_k}{2}} \\
    &\quad \times \exp\left\{\frac{1}{2}\frac{N}{s_N}\beta (m_i^2 + m_j^2) \right\} \\
    &\eqqcolon \widetilde{Z}_{N,\beta,\alpha}.
\end{flalign*}
Next we will use the fact that on the event $ \sup_{k=1}^{s_N} \min\left\{ |m_k-m^*|, |m_k+m^*| \right\}\le \eps$, the event $m_im_{i+1}, m_jm_{j+1}<0$ already implies $m_im_{i+1}, m_jm_{j+1} \le -(m^*-\eps)^2$. It follows that
\begin{flalign*}
   & \mu_{N,\beta, \alpha}\left( \sup_{k=1}^{s_N} \min\left\{ |m_k-m^*|, |m_k+m^*| \right\}\le \eps, m_im_{i+1}<0,  m_jm_{j+1}<0 \right) \\
    \le& \frac{1}{\widetilde{Z}_{N,\beta,\alpha}}\exp\left\{ -\frac{N}{s_N} 2 \alpha (m^*-\eps)^2 \right\} \sum_{m_1, \ldots, m_{s_N}} \1_{m_im_{i+1}<0}\1_{m_jm_{j+1}<0}\1_{\left\{\sup_{k=1}^{s_N} \min\left\{ |m_k-m^*|, |m_k+m^*| \right\}\le \eps \right\}} \\
    & \times\exp\left\{ \frac{1}{2}\frac{N}{s_N}\sum_{\substack{k=1 \\k\neq i,j}}^{s_N} \beta m_k^2 + 2\alpha m_km_{k+1} \right\} \frac{1}{2^N}\prod_{k=1}^{s_N}{ \frac{N}{s_N} \choose \frac{N}{s_N}\frac{1+m_k}{2}} \exp\left\{\frac{1}{2}\frac{N}{s_N}\beta (m_i^2 + m_j^2) \right\} \\
    &\le \exp\left\{ -\frac{N}{s_N} 2 \alpha (m^*-\eps)^2 \right\},
\end{flalign*}
where the last step is true since $\1_{\left\{\sup_{k=1}^{s_N} \min\left\{ |m_k-m^*|, |m_k+m^*| \right\}\le \eps \right\}} \le 1$ and if we drop this indicator function, the sum simply becomes $\widetilde{Z}_{N,\beta,\alpha}$
\end{proof}

\section{Proof of Theorem \ref{CLT}\,(ii)}
The proof for the CLT in the low temperature regime is similar in spirit as for high temperature but new challenges arise due to the non-uniqueness for which we need technical adjustments. In the whole section we will assume that $\beta + 2\alpha <1$. Recall that by Lemma \ref{LowTemp}, in the low temperature case the function $\varphi_N$ has two minimizers that we denote by $\overrightarrow m^*$ and $-\overrightarrow m^* $, where $\overrightarrow m^*$ is the $s_N$-dimensional vector that carries the value $m^*$ in every coordinate. We want to prove that under the Gibbs measure, conditioned on the magnetization being in a box of of side length $2\delta$ around one of the minimizers, $\sqrt{\frac{N}{s_N}}(m\mp \overrightarrow m^*) \Rightarrow \mathcal{N}(0, \Sigma^*)$, where $\delta < 2m^*$.
We will prove this by showing convergence of the moment generating functions.

\subsection{Decomposing the integral} \label{subsec:dec}
For the remainder of this section, let us fix $A$, and hence $\alpha, \beta,m^*$. Let $Z\sim \mathcal{N}(0, A^{-1})$ be an $s_N$-dimensional Gaussian random variable independent of the magnetization and $t \in \R^{s_N}$ such that $t$ has only finitely many non-zero entries, then similarly to the proof of Lemma \ref{phi},
\begin{flalign*}
    &\E\left[ \exp \left\{ t^T\left(Z+ \sqrt{\frac{N}{s_N}}\left(m- \overrightarrow m^*\right)\right) \right\} \, \middle|  \, m\in [m^* - \delta , m^* + \delta]^{s_N}\right] \\
    =& \exp \left\{- \sqrt{\frac{N}{s_N}}t^T \overrightarrow m^* \right\} \E\left[ \exp \left\{ t^T\left(Z+ \sqrt{\frac{N}{s_N}}m\right) \right\} \, \middle|  \, m\in [m^* - \delta , m^* + \delta]^{s_N} \right] \\
    =& \exp \left\{- \sqrt{\frac{N}{s_N}}t^T \overrightarrow m^* \right\} c_{N,\delta} \sum_{\substack{\sigma \in \{-1,1\}^N \\ m\in[m^* - \delta , m^* + \delta]^{s_N}}} \frac{1}{2^N} \int_{\R^{s_N}} \exp\left\{ -\frac{1}{2} y^TAy + t^Ty\right\}\exp\left\{\sqrt{\frac{N}{s_N}} y^TAm \right\} \d^{s_N}y \\
    =& \exp \left\{- \sqrt{\frac{N}{s_N}}t^T \overrightarrow m^* \right\}  c_{N,\delta}  \left(\frac{N}{s_N} \right)^\frac{s_N}{2} \int_{\R^{s_N}} \exp\left\{ - \frac{1}{2}\frac{N}{s_N}x^TAx + \sqrt{\frac{N}{s_N}}t^Tx \right\} I_N(x,\overrightarrow m^*,\delta) \d^{s_N}x
\end{flalign*}
 where $c_{N,\delta} = \left(  \left(\frac{N}{s_N} \right)^\frac{s_N}{2} \int_{\R^{s_N}} \exp\left\{ - \frac{N}{2s_N}xTAx + \sqrt{\frac{N}{s_N}}t^Tx \right\} I_N(x,\overrightarrow m^*,\delta) \, \d^{s_N}x\right)^{-1}$ and the last step follows from the change of variables $y= \sqrt{\frac{N}{s_N}}x$ and defining
 \begin{align}\label{eq:I}
 I_N(x,\overrightarrow m^*,\delta) = \sum_{\substack{\sigma \in \{-1,1\}^N \\ m\in [m^* - \delta , m^* + \delta]^{s_N}}}  \frac{1}{2^N} \exp\left\{\frac{N}{s_N} x^TAm \right\}.
 \end{align}
 Further we define
 \begin{align}\label{eq:Ic}
I_N^c(x, \overrightarrow m^*, \delta) = \sum_{\substack{\sigma \in \{-1,1\}^N \\ m\notin [m^* - \delta , m^* + \delta]^{s_N}}}  \frac{1}{2^N} \exp\left\{\frac{N}{s_N} x^TAm \right\}
\end{align}
and as in the proof of Lemma \ref{phi}, we have 
 \begin{align}\label{eq:I+Ic}
I_N(x,\overrightarrow m^*,\delta) + I_N^c(x,\overrightarrow m^*,\delta) = \exp\left\{ \frac{N}{s_N}\sum_{k=1}^{s_N} \log\cosh (x^TAe_k)\right\}.
\end{align}
Recall the definition of $\varphi_N$ from \eqref{def_phi}. For $r>0$ small enough (to be chosen later), we will compare the moment generating function to
\begin{flalign}
    &J_N(t,r,\overrightarrow m^*): =  \left( \frac{N}{s_N}\right)^\frac{s_N}{2}  \exp \left\{ \frac{N}{s_N}\varphi_N(\overrightarrow m^*) - \sqrt{\frac{N}{s_N}}t^T \overrightarrow m^* \right\} \label{eq:J}
    \\
    &\quad \quad \quad \quad \times \int_{[m^*-r,m^*+r]^{s_N}}\exp\left\{ -\frac{N}{s_N} \varphi_N(x) + \sqrt{\frac{N}{s_N}} t^Tx\right\} \d^{s_N}x ,\nonumber
\end{flalign}
which deviates from the moment generating function by the terms
\begin{flalign}
    & K_N(t,r, \overrightarrow m^*, \delta) := \left( \frac{N}{s_N}\right)^\frac{s_N}{2}  \exp \left\{ \frac{N}{s_N}\varphi_N(\overrightarrow m^*) - \sqrt{\frac{N}{s_N}}t^T \overrightarrow m^* \right\} \label{eq:K}\\
    &\quad \quad \quad \quad \times \int_{[m^*-r,m^*+r]^{s_N}}\exp\left\{ -\frac{1}{2}\frac{N}{s_N}x^TAx + \sqrt{\frac{N}{s_N}} t^Tx\right\} I_N^c(x,\overrightarrow m^*,\delta) \d^{s_N}x ,\nonumber\\
    & L_N(t,r,\overrightarrow m^*, \delta): = \left( \frac{N}{s_N}\right)^\frac{s_N}{2}  \exp \left\{ \frac{N}{s_N}\varphi_N(\overrightarrow m^*) - \sqrt{\frac{N}{s_N}}t^T \overrightarrow m^* \right\} \label{eq:L}\\
    &\quad \quad \quad \quad \times \int_{\left([m^*-r, m^*+r]^{s_N} \right)^c}\exp\left\{ -\frac{1}{2}\frac{N}{s_N}x^TAx + \sqrt{\frac{N}{s_N}} t^Tx\right\} I_N(x,\overrightarrow m^*,\delta) \d^{s_N}x.\nonumber
\end{flalign}
We will shortly write $J_N(t), K_N(t), L_N(t)$ and drop the other arguments whenever it is clear from the context. Merging the terms, we retrieve the moment generating function
\begin{align}\label{eq:MGF_JKL}
 \E\left[ \exp \left\{ t^T\left(Z+ \sqrt{\frac{N}{s_N}}\left(m- \overrightarrow m^*\right)\right) \right\} \, \middle|  \, m\in [m^*-\delta,m^*+\delta]^{s_N} \right]=\frac{J_N(t) - K_N(t) + L_N(t)}{J_N(0) - K_N(0) + L_N(0)} .
\end{align}
This representation is inspired by \cite[(4.6)]{EllisWang}. 
It will turn out that only the $J_N$-terms will contribute in the limit and we will start by showing that $\frac{J_N(t)}{J_N(0)}$ converges to the moment generating function of an (infinite dimensional) normal distribution. Afterwards we will prove that the contributions of the $K_N$- and $L_N$-terms are negligible in the limit.

\subsection{The dominating term J}\label{subsec:J}
By a change of variables we obtain
\begin{flalign}\label{transformation1}
    \frac{J_N(t)}{J_N(0)} &=  \frac{ \int_{[m^*-r,m^*+r]^{s_N}} \exp\left\{ -\frac{N}{s_N}\left(\varphi_N(x) - \varphi_N(\overrightarrow m^*)\right)  + \sqrt{\frac{N}{s_N}} t^T(x-\overrightarrow m^*)\right\} \d^{s_N}x}{ \int_{[m^*-r,m^*+r]^{s_N}} \exp\left\{ -\frac{N}{s_N}\left(\varphi_N(x) - \varphi_N(\overrightarrow m^*)\right) \right\} \d^{s_N}x } \nonumber \\
    &= \frac{\int_{\left[-\sqrt{\frac{N}{s_N}}r, \sqrt{\frac{N}{s_N}}r \right]^{s_N}} \exp\left\{ -\frac{N}{s_N}\left(\varphi_N\left(\overrightarrow m^* + \sqrt{\frac{s_N}{N}}y\right) - \varphi_N(\overrightarrow m^*)\right)  + t^Ty\right\} \d^{s_N}y }{ \int_{\left[-\sqrt{\frac{N}{s_N}}r, \sqrt{\frac{N}{s_N}}r \right]^{s_N}} \exp\left\{ -\frac{N}{s_N}\left(\varphi_N\left(\overrightarrow m^* + \sqrt{\frac{s_N}{N}}y\right) - \varphi_N(\overrightarrow m^*)\right) \right\} \d^{s_N}y }
    \end{flalign}

\begin{lemma}\label{lem:log-concaveLT}
    For all $r>0$ sufficiently small (depending on $A$),
    $$
    \frac{J_N(t,r,\overrightarrow m^*)}{J_N(0, r, \overrightarrow m^*)} = \E_{\rho_N}\left[ \exp\left\{ t^T X \right\} \right],
    $$
    where $X$ denotes an $s_N$-dimensional random variable whose distribution is strongly log-concave with density
    $$
    \rho_N(x) \propto \exp\left\{-\frac{N}{s_N} \left( \varphi_N\left(\overrightarrow m^* + \sqrt{\frac{s_N}{N}}x \right) - \varphi_N( \overrightarrow m^*)\right) \right\} \1_{\left[-\sqrt{\frac{N}{s_N}}r, \sqrt{\frac{N}{s_N}}r \right]^{s_N}}(x)
    $$
\end{lemma}
\begin{proof} Similar to the proof of Lemma \ref{lem:log-concave}, set $$
U(x)=\begin{cases}
    \frac{N}{s_N} \left( \varphi_N\left(\overrightarrow m^* + \sqrt{\frac{s_N}{N}}x \right) - \varphi_N( \overrightarrow m^*)\right),\quad & \text{if } x\in \left[-\sqrt{\frac{N}{s_N}}r, \sqrt{\frac{N}{s_N}}r \right]^{s_N}\\
    +\infty, &\text{else.}
\end{cases}$$ We will show that $\Hess U(x)$ is strictly positive definite on the convex set $ \left[-\sqrt{\frac{N}{s_N}}r, \sqrt{\frac{N}{s_N}}r \right]^{s_N}$, which implies the claim. The Hessian of $\varphi_N$ in \eqref{Hess} implies
    \begin{flalign}
        \Hess U(x)
        = A-  A\sum_{k=1}^{s_N} \sech^2\left(  \left( \overrightarrow m^* + \sqrt{\frac{s_N}{N}}x\right)^TAe_k\right)e_ke_k^TA. \label{eq:HessU}
    \end{flalign}
Recall that for $x=0$ we would get
    \begin{align}\label{eq:sechtanh}
    \sech^2\left((\overrightarrow m^*)^TAe_k \right) = \sech^2((\beta +2\alpha) m^*) =1-\tanh^2((\beta +2\alpha) m^*) = 1-(m^*)^2.
    \end{align}
    Since $A-A^2(1-(m^*)^2)$ is positive definite, there exists a small constant $r(\beta,\alpha)>0$, depending on $\beta$ and $\alpha$ such that $A-A^2(1-(m^*)^2+ r(\beta, \alpha))$ is still positive definite. For $x \in\left[-\sqrt{\frac{N}{s_N}}r, \sqrt{\frac{N}{s_N}}r \right]^{s_N}$, we have that $\sqrt{\frac{s_N}{N}}x = y$ for some $y \in [-r,r]^{s_N}$ and therefore by Taylor's theorem for an intermediate value $\xi \in [-r,r]^{s_N}$, we obtain
    \begin{flalign*}
        &\sech^2\left( \left( \overrightarrow m^* + \sqrt{\frac{s_N}{N}}x\right)^TAe_k\right) \\        =&
        \sech^2\left( (\overrightarrow m^* + y)^T Ae_k \right)\\
        =& \sech^2((\overrightarrow m^*)^TAe_k) -2 (\beta y_k + \alpha(y_{k-1} + y_{k+1}))\sech^2(\xi^TAe_k)\tanh(\xi^TAe_k).
    \end{flalign*}
    It follows that for any such $x$, we have $\sech^2\left( \left( \overrightarrow m^* + \sqrt{\frac{s_N}{N}}x\right)^TAe_k\right) \le 1-(m^*)^2+2r(\beta + 2\alpha) $ implying that $\Hess U(x)$ is component-wise larger than $A-A^2(1-(m^*)^2+r)$ which is positive definite for $2(\beta+  2\alpha)r < r(\beta, \alpha)$.
\end{proof}
Let us point out that the denominator never vanishes with the following bound, which we shall need later.
 \begin{cor}\label{cor:J(0)not0}
There exists $\varepsilon>0$ such that $J_N(0)>\varepsilon^{s_N}>0$.
 \end{cor}
 \begin{proof}
Note that \eqref{eq:HessU} also gives a crude uniform upper bound $$\Big\lVert\Hess \frac{N}{s_N} \varphi_N\left(\overrightarrow m^* + \sqrt{\frac{s_N}{N}}x \right)\Big\rVert\le \lambda_{\max}(A)+\lambda_{\max}(A)^2=(\beta+2\alpha)(1+\beta+2\alpha)=:\lambda_0 .$$ 
A Taylor expansion of $\frac{N}{s_N}\varphi_N\left( \overrightarrow m^*+\sqrt{\frac{s_N}{N}}x\right)$ for $\left\vert\sqrt{\frac{s_N}{N}}x \right\vert \le r$ yields
\begin{flalign} \label{eq:upperBoundHess}
    \frac{N}{s_N}\left(\varphi_N\left( \overrightarrow m^*+\sqrt{\frac{s_N}{N}}x\right) -\varphi_N(\overrightarrow m^*) \right) = \frac{1}{2}x^T \Hess \varphi_N\left(\overrightarrow m^* + \sqrt{\frac{s_N}{N}\xi_x } \right) x 
    \le \frac  {\lambda_0} 2  \|x\|^2
\end{flalign}
for an intermediate value $\xi_x \in \left[ m^*,  m^* +\sqrt{\frac{N}{s_N}}x \right]^{s_N}$.
Recall the definition \eqref{eq:J} of $J_N(0)$, then applying a change of variables as in \eqref{transformation1} together with \eqref{eq:upperBoundHess} yields
\begin{flalign*}
    J_N(0) &= \int_{\left[-\sqrt{\frac{N}{s_N}}r, \sqrt{\frac{N}{s_N}}r \right]^{s_N}} \exp \left\{- \frac{N}{s_N}\left(\varphi_N\left( \overrightarrow m^*+\sqrt{\frac{s_N}{N}}x\right) -\varphi_N(\overrightarrow m^*) \right) \right\} \d^{s_N}x \\
   &\ge \int_{\left[-\sqrt{\frac{N}{s_N}}r, \sqrt{\frac{N}{s_N}}r \right]^{s_N}} \exp \left\{- \frac{\lambda_0}{2} \|x\|^2 \right\} \d^{s_N}x\\
   &=\Big(\sqrt{2\pi\lambda_0^{-1}}\mathcal N (0,\lambda_0^{-1})\big([-\sqrt{\tfrac N {s_N}},\sqrt{\tfrac N {s_N}}]\big)\Big)^{s_N}\ge (\pi/\lambda_0) ^{s_N}
\end{flalign*}
as claimed.
 \end{proof}

To shorten notation, define the $s_N \times s_N$-dimensional matrix 
\begin{align}\label{eq:Cstern}
(\Cstern)^{-1} = \Hess \varphi_N (\overrightarrow m^*) = A-A^2(1- (m^*)^2)
\end{align}
and let us introduce the function 
$$
G(x) := \frac{x^T((\Cstern )^{-1}- A)x}{2} - (\overrightarrow m^*)^TAx + \sum_{k=1}^{s_N} \log\cosh\left( (\overrightarrow m^* + x)^T Ae_k \right) - \log\cosh\left(  (\overrightarrow m^*)^TAe_k\right),
$$
which in contrast to \eqref{eq:G} is defined as a function of all coordinates $G:\R^{s_N}\to \R$. Then, we obtain 
$$
\frac{N}{s_N}G\left(\sqrt{\frac{s_N}{N}}y \right) = \frac{y^T(\Cstern )^{-1}y}{2} - \frac{N}{s_N}\left( \varphi_N \left(\overrightarrow m^* + \sqrt{\frac{s_N}{N}}y \right) - \varphi_N\left(\overrightarrow m^* \right)\right)
$$
and therefore \eqref{transformation1} implies
\begin{flalign*}
\frac{J_N(t)}{J_N(0)} =&\frac{\int_{\left[-\sqrt{\frac{N}{s_N}}r, \sqrt{\frac{N}{s_N}}r \right]^{s_N}} \exp\left\{- \frac{1}{2}y^T(\Cstern)^{-1}y + \frac{N}{s_N}G\left(\sqrt{\frac{s_N}{N}}y \right) + t^Ty \right\} \d^{s_N}y }{\int_{\left[-\sqrt{\frac{N}{s_N}}r, \sqrt{\frac{N}{s_N}}r \right]^{s_N}} \exp\left\{- \frac{1}{2}y^T(\Cstern)^{-1}y + \frac{N}{s_N}G\left(\sqrt{\frac{s_N}{N}}y \right) \right\} \d^{s_N}y} \\
=&
\frac{\int_{(\Cstern)^{-\frac{1}{2}} \left[-\sqrt{\frac{N}{s_N}}r, \sqrt{\frac{N}{s_N}}r \right]^{s_N}} \exp\left\{- \frac{1}{2}\| x\|^2+ \frac{N}{s_N}G\left(\sqrt{\frac{s_N}{N}}(\Cstern)^{\frac{1}{2}}x \right) + t^T(\Cstern)^{\frac{1}{2}}x \right\} \d^{s_N}x }{\int_{ (\Cstern)^{-\frac{1}{2}} \left[-\sqrt{\frac{N}{s_N}}r, \sqrt{\frac{N}{s_N}}r \right]^{s_N}} \exp\left\{- \frac{1}{2}\|x\|^2+ \frac{N}{s_N}G\left(\sqrt{\frac{s_N}{N}}(\Cstern)^{\frac{1}{2}}x \right) \right\} \d^{s_N}x},
\end{flalign*}
where in the second step, we did another change of variables $x=(\Cstern)^{-\frac{1}{2}}y$. Note that the multiplication with the positive definite matrix $(\Cstern)^{-\frac{1}{2}}$ preserves the statement of Lemma \ref{lem:log-concaveLT} for the transformed density
$$
\widetilde{\rho}_N(x) \propto \exp\left\{ - \frac{1}{2}\| x\|^2+ \frac{N}{s_N}G\left(\sqrt{\frac{s_N}{N}}(\Cstern)^{\frac{1}{2}}x \right)\right\} \1_{ (\Cstern)^{-\frac{1}{2}}\left[-\sqrt{\frac{N}{s_N}}r, \sqrt{\frac{N}{s_N}}r \right]^{s_N}}(x).
$$
As in the high temperature case, the next step is to restrict the area of integration in the first few coordinates, as in Lemma \ref{lemma:D_N-restriction}.
\begin{lemma}\label{lemma:D_N-restriction_LT}
    Let $d_N = \lceil c_9 \log N \rceil$ for some constant $c_9 >0$ (to be defined later) and  define 
    $$
    D_N \coloneqq
        (\Cstern )^{-\frac{1}{2}}\left[ - \sqrt{\frac{N}{s_N}}r, \sqrt{\frac{N}{s_N}}r \right]^{s_N}\cap B_{2\sqrt{2d_N}}^{(2d_N)} 
    $$ Then there exists a universal constant $c_{10} >0$, such that as $N\to\infty$
    $$
    \E_{\widetilde{\rho}_N}\left[ \1_{D_N^c}(X)\exp\left\{ t^T(\Cstern )^{\frac{1}{2}}X\right\} \right] = \mathcal{O}\left( \exp\left\{ - c_{10} d_N \right\}\right) \to 0.
    $$ 
\end{lemma}
Having strong log-concavity of Lemma \ref{lem:log-concaveLT}, the statement of the lemma follows exactly as in Lemma \ref{lemma:D_N-restriction} (again the case $2d_N >s_N$ being the simpler), hence we omit the proof.
Finally, we can do another change of variables $x \mapsto x+ (\Cstern )^{\frac{1}{2}} t \in \widetilde{D}_N$ for $\widetilde{D}_N = D_N+ (\Cstern )^{\frac{1}{2}} t  $, then the integral becomes
\begin{align}
    &\lim_{N\to \infty} \frac{J_N(t)}{J_N(0)}\nonumber \\
    =& \lim_{N\to \infty}\exp\left\{ \frac{1}{2}t^T\Cstern t\right\}\frac{\int_{\widetilde{D}_N} \exp\left\{ -\frac{1}{2} \|x\|^2 + \frac{N}{s_N} G\left(  \sqrt{\frac{s_N}{N}}(\Cstern)^{\frac{1}{2}}x+ \sqrt{\frac{s_N}{N}}\Cstern t\right) \right\} \d^{s_N}x }{\int_{\widetilde{D}_N} \exp\left\{ -\frac{1}{2} \|x\|^2 + \frac{N}{s_N} G\left(  \sqrt{\frac{s_N}{N}}(\Cstern )^{\frac{1}{2}}x \right) \right\} \d^{s_N}x}.\label{eq:J(t)/J(0)}
\end{align}
Again we will show that the quotient on the right hand side converges to 1. As in the high temperature case, the numerator and denominator differ only by the term $\sqrt{\frac{s_N}{N}}\Cstern t$ inside the function $G$ and we will show that it can be neglected in the limit. Since $t$ has only finitely many non-zero coordinates, clearly the term goes to zero as $N \to \infty$ and we do the Taylor expansion
$$
\frac{N}{s_N}G\left(\sqrt{\frac{s_N}{N}}(\Cstern )^{\frac{1}{2}}x + \sqrt{\frac{s_N}{N}}\Cstern t \right) = \frac{N}{s_N}G\left( \sqrt{\frac{s_N}{N}}(\Cstern )^{\frac{1}{2}}x\right) + \frac{N}{s_N} \nabla G(\xi)^T \sqrt{\frac{S_N}{N}}\Cstern t ,
$$
where $\xi= \sqrt{\frac{s_N}{N}}(\Cstern )^{\frac{1}{2}}x + \lambda \sqrt{\frac{s_N}{N}}\Cstern t $, for some $\lambda \in [0,1]$, is an intermediate value. We need to show that
$$
 \sqrt{\frac{N}{s_N}} \nabla G(\xi)^T\Cstern t=\sqrt{\frac{N}{s_N}} \sum_{k=1}^{s_N}(\nabla G(\xi))_k (\Cstern t)_k
$$
is small.
\begin{lemma}\label{lem:G'=o_LT}
    If $s_N = o\left( \frac{N}{(\log N)^4}\right)$ and $d_N = c_{9}\log N$ for some constant $c_{9}>0$, then
    $$
    \sqrt{\frac{N}{s_N}} \nabla G(\xi)^T\Cstern t
    = o(1)
    $$
    uniformly in $x \in (\Cstern )^{-\frac{1}{2}}\left[ -\sqrt{\frac{N}{s_N}}2r, \sqrt{\frac{N}{s_N}}2r\right]^{s_N} \cap B_{4\sqrt{d_N}}^{(2d_N)}$.
\end{lemma}
Note that $D_N, \widetilde{D}_N \subset \left(\Cstern \right)^{-\frac{1}{2}} \left[ -\sqrt{\frac{N}{s_N}}2r, \sqrt{\frac{N}{s_N}}2r\right]^{s_N} \cap B_{4\sqrt{d_N}}^{(2d_N)}$ for $N$ large enough. As in the high temperature regime, a key argument in proving the above statement are the decay properties of the entries of the matrices $(\Cstern )^{\frac{1}{2}}$ and $\Cstern$, away from the diagonal and the corners.
\begin{lemma}\label{lem:exp-decay-C*}
    Let $\Cstern = \left( A - A^2(1- (m^*)^2) \right)^{-1}$, then there exist constants $c_{7}, c_{8}>0$ and $0< \kappa_3, \kappa_4 <1$ such that
    \begin{equation}
        \Cstern _{j,k} = c_7\left( \kappa_3^{|j-k|} + \kappa_3^{s_N-|j-k|}\right)
    \end{equation}
    and
    \begin{equation}
        (\Cstern )^{\frac{1}{2}}_{j,k} = c_{8}\left( \kappa_4^{|j-k|} + \kappa_4^{s_N-|j-k|}\right).
    \end{equation} 
\end{lemma}

We omit the proof as it follows the same steps of Lemma \ref{lem:matrix_decay} for the inverse, and square-root respectively, of the circulant banded matrix $ A - A^2(1- (m^*)^2)$.

\begin{proof}[Proof of Lemma \ref{lem:G'=o_LT}]
    As in the proof of Lemma \ref{lem:G'=o}, we will first consider the case $s_N \gg \log N$ and afterwards deduce the case $s_N = \mathcal{O}(\log N)$ from the existing computations.
    Recall that $t$ has only finitely many non-zero components and let $l$ be the largest index of a non-zero component. Define $\widetilde{d}_N = -\frac{5}{\log \kappa_3}\log N$, such that $\kappa_3^{\widetilde{d}_N} = N^{-5}$. Again, define $\mathcal{I}_{\widetilde{d}_N} = \{1, \ldots, \widetilde{d}_N\}\cup \{s_N-\widetilde{d}_N+1, \ldots, s_N\}$ be the indices near the endpoints $1,s_N$ and let us first consider $k \notin \mathcal{I}_{\widetilde{d}_N}$ in order to show that their contribution is negligible. For $k=\widetilde{d}_N+1, \ldots, s_N-\widetilde{d}_N$, we have by Lemma \ref{lem:exp-decay-C*}
    \begin{flalign*}
        |(\Cstern t)_k| &\le \sum_{j=1}^{s_N} |\Cstern_{k,j}t_j|= \sum_{j=1}^l |\Cstern_{k,j}t_j| \\
        &= c_7 \sum_{j=1}^l \left( \kappa_3^{|j-k|} + \kappa_3^{s_N - |j-k|} \right)|t_j| \le c_7 2\kappa_3^{\widetilde{d}_N}\sum_{j=1}^l |t_j| = \mathcal{O}(N^{-5})
    \end{flalign*}
 uniformly in $k\notin \mathcal{I}_{\widetilde{d}_N}$, where the last step follows from the definition of $\widetilde{d}_N$. Next, we need to find a suitable bound for $(\nabla G(\xi))_k$. As $\xi =  \sqrt{\frac{s_N}{N}}(\Cstern)^{\frac{1}{2}}x + \lambda \sqrt{\frac{s_N}{N}}\Cstern t$ for some $\lambda \in [0,1]$ is the intermediate value from Taylor's theorem, we know that
 $$
 \xi_k = \sqrt{\frac{s_N}{N}}\left((\Cstern)^{\frac{1}{2}}x\right)_k + \lambda \sqrt{\frac{s_N}{N}}(\Cstern t)_k = \sqrt{\frac{s_N}{N}}\left((\Cstern)^{\frac{1}{2}}x\right)_k + \mathcal{O}(N^{-5}),
 $$
for $k\notin \mathcal{I}_{\widetilde{d}_N}$. Recall that $x$ is restricted to the set $(\Cstern)^{-\frac{1}{2}} \left[ - \sqrt{\frac{N}{s_N}}2r, \sqrt{\frac{N}{s_N}}2r \right]^{s_N}$, hence
$$
|\xi_k| \le 2r + \mathcal{O}(N^{-5}) = \mathcal{O}(1).
$$
Note that 
\begin{flalign*}
   |\left(\nabla G(\xi)\right)_k | &= \Bigg|\left( \left((\Cstern )^{-1} - A\right)\xi\right)_k - m^*(\beta + 2\alpha) + \left(A\sum_{j=1}^{s_N}\tanh\left( (\overrightarrow m^* + \xi)^T Ae_j \right) e_j \right)_k \Bigg| \\
   &\le (1- (m^*)^2)|(A^2 \xi)_k| + (1-m^*)(\beta + 2\alpha),
\end{flalign*}
where the inequality follows since $(\Cstern )^{-1} = A - A^2(1- (m^*)^2)$ by definition and since $-1 \le \tanh \le 1$. Putting all the observations together, we obtain 
\begin{flalign*}
    &\sqrt{\frac{N}{s_N}}\Bigg| \sum_{k \notin \mathcal{I}_{\widetilde{d}_N}} (\nabla G(\xi))_k (\Cstern t)_k \Bigg| \\
    \le &\sqrt{\frac{N}{s_N}} \sum_{k \notin \mathcal{I}_{\widetilde{d}_N}} \left( (1-(m^*)^2) |(A^2\xi)_k| + \mathcal{O}(1) \right) \mathcal{O}(N^{-5}) \\
    \le& \mathcal{O}\left(\sqrt{\frac{N}{s_N}} N^{-5} \right) \sum_{k \notin \mathcal{I}_{\widetilde{d}_N}} \big((1- (m^*)^2)\|A^2\|2r + \mathcal{O}(1) \big)=o(1).
\end{flalign*}
It remains to show 
$$
\sqrt{\frac{N}{s_N}} \sum_{k\in \mathcal{I}_{\widetilde{d}_N}} (\nabla G (\xi))_k (\Cstern t)_k \to 0,
$$
but for $k\in \mathcal I_{\widetilde{d}_N}$, the contribution of $(\Cstern t)_k$ is not negligible. We will again make use of the exponential decay away from the diagonal, this time for $(\Cstern )^{\frac{1}{2}}$. Hence, we may just use the uniform upper bound
\begin{equation}\label{eq:boundCstern2}
    |(\Cstern  t)_k| \le \|\Cstern \|\|t\| \le C_{\beta, \alpha,t}
\end{equation}
by bounding all eigenvalues and we aim to show that
\begin{align}\label{eq:aimm}
    \sqrt{\frac{N}{s_N}} \sum_{k\in \mathcal{I}_{\widetilde{d}_N}} (\nabla G (\xi))_kC_{\beta, \alpha,t}\to 0
\end{align}
by bounding $(\nabla G(\xi))_k$ for $\xi = \sqrt{\frac{s_N}{N}}(\Cstern )^{\frac{1}{2}}x + \lambda \sqrt{\frac{s_N}{N}}\Cstern t$, with $\lambda \in [0,1]$. 
Recall the definition $\mathcal{I}_{d_N} = \{1, \ldots, d_N, s_N-d_N+1, \ldots, s_N\}$ for 
$$
d_N = -\frac{5}{\log \kappa_4}\log N + \widetilde{d}_N = \left(-\frac{5}{\log \kappa_4}-\frac{5}{\log \kappa_3} \right)\log N =: c_{9}\log N
$$
and split
$$
\left( (\Cstern )^{\frac{1}{2}}x \right)_k = \sum_{j\notin \mathcal{I}_{d_N}} (\Cstern )^{\frac{1}{2}}_{kj}x_j + \sum_{j\in \mathcal{I}_{d_N}} (\Cstern)^{\frac{1}{2}}_{kj}x_j.
$$
It follows from Lemma \ref{lem:exp-decay-C*} that
\begin{align*}
    \sum_{j\notin \mathcal{I}_{d_N}} \Big| (\Cstern)^{\frac{1}{2}}_{kj}x_j\Big| \le c_{8}\sum_{j\notin \mathcal{I}_{d_N}} (\kappa_4^{|j-k|} + \kappa_4^{s_N-|j-k|} )|x_j|\le 2 c_{8}  \kappa_4^{d_N} \sum_{j\notin \mathcal{I}_{d_N}} |x_j|= \mathcal{O}\left(N^{-5} d_N\sqrt{\frac{N}{s_N}}\right).
\end{align*}
For the near part $j\in \mathcal{I}_{d_N}$, we need to use $x \in B_{4\sqrt{2d_N}}^{(2d_N)}$ having small projected norm, i.e.
\begin{flalign}\label{eq:boundCstern}
    \Bigg| \sum_{j\in \mathcal{I}_{d_N}} (\Cstern )^{\frac{1}{2}}_{k,j}x_j  \Bigg| =  \Bigg| \left((\Cstern )^{\frac{1}{2}}P_{\mathcal{I}_{d_N}}x \right)_k \Bigg| \le  \|(\Cstern )^{\frac{1}{2}} P_{\mathcal{I}_{d_N}}x \| \le \|(\Cstern )^{\frac{1}{2}}\|\| P_{\mathcal{I}_{d_N}}x\| \le 4\|(\Cstern )^{\frac{1}{2}}\| \sqrt{2d_N},
\end{flalign}
where the spectral norm $\|(\Cstern )^{\frac{1}{2}}\|$ is bounded uniformly in $N$. In conclusion,
\begin{flalign*}
    |\xi_k| &= \sqrt{\frac{s_N}{N}}\Big|\left((\Cstern )^{\frac{1}{2}}x\right)_k \Big| + \mathcal{O}\left(\sqrt{\frac{s_N}{N}} \right) \\
    &\le \sqrt{\frac{s_N}{N}}\left( 4 \|(\Cstern )^{\frac{1}{2}}\|\sqrt{2d_N} + \mathcal{O}\left( N^{-5} d_N\sqrt{\frac{N}{s_N}}\right) \right) +  \mathcal{O}\left(\sqrt{\frac{s_N}{N}} \right) \\
    &= \mathcal{O}\left(\sqrt{\frac{s_N}{N}\log N} \right) \longrightarrow 0
\end{flalign*}
In other words, $\xi$ is component-wise small and we can do a Taylor expansion of $\tanh$, yielding 
\begin{flalign}\label{eq:Taylor_nablaG}
    \nabla G(\xi)& = \left( (\Cstern )^{-1} - A \right)\xi - (\overrightarrow m^*)^TA  + A\sum_{k=1}^{s_N} \tanh\left( (\overrightarrow m^* + \xi)^TAe_k \right)e_k \nonumber \\
    &= -(1-(m^*)^2) A^2 \xi - (\overrightarrow m^*)^TA + A\sum_{k=1}^{s_N} \Big(\tanh\left( (\overrightarrow m^*)^TAe_k \right) + \sech^2\big(  (\overrightarrow m^*)^TAe_k\big) (Ae_k)^T\xi  \nonumber\\
    &\quad - \sech^2\big(  \widetilde{\xi}^TAe_k\big)\tanh\big(  \widetilde{\xi}^TAe_k\big)\xi^T (Ae_k)(Ae_k)^T \xi \Big)e_k,
\end{flalign}
where $\widetilde{\xi}$ is another intermediate value between $\overrightarrow m^*$ and $\overrightarrow m^*+\xi$. Recall that $\tanh\left((\overrightarrow m^*)^TAe_k \right) = \tanh\left( (\beta + 2\alpha)m^* \right) = m^*$ for every $k=1, \ldots, s_N$, it follows
$$
A\sum_{k=1}^{s_N} \tanh\left( (\overrightarrow m^*)^TAe_k \right) e_k = (\overrightarrow m^*)^TA,
$$
and by \eqref{eq:sechtanh},
$$
\sum_{k=1}^{s_N} \sech^2\left(  (\overrightarrow m^*)^TAe_k\right) (Ae_k)^T\xi e_k  = (1 - (m^*)^2)\sum_{k=1}^{s_N} (Ae_k)^T\xi e_k  = (1- (m^*)^2) A^2\xi.
$$
Therefore, by writing $\tilde\lambda_k = \sech^2\left(  \widetilde{\xi}^TAe_k\right)\tanh\left(  \widetilde{\xi}^TAe_k\right)$ which is a value in $[-1,1]$ depending on $\widetilde{\xi}$ and $k$, the equation \eqref{eq:Taylor_nablaG} reduces to
$$
\nabla G(\xi) = A\sum_{k=1}^{s_N} \tilde\lambda_k(\xi^TAe_k)^2e_k
$$
This means that 
$$
\nabla G (\xi) =  
A \sum_{k=1}^{s_N}\tilde\lambda_k(\beta \xi_k + \alpha(\xi_{k-1} + \xi_{k+1}))^2e_k
$$
implying uniformly for $k\in \mathcal{I}_{\widetilde{d}_N}$
\begin{align}\label{eq:Taylor_nablaG2}
|(\nabla G (\xi) )_k| &= \Big| \beta\tilde\lambda_k(\beta \xi_k + \alpha(\xi_{k-1} + \xi_{k+1}))^2 + \alpha\tilde\lambda_{k-1}(\beta \xi_{k-1} + \alpha(\xi_{k-2} + \xi_{k}))^2 \nonumber \\
&\qquad + \alpha\tilde\lambda_{k+1}(\beta \xi_{k+1} + \alpha(\xi_{k} + \xi_{k+2}))^2\Big| \nonumber  \\
&\le (\beta +2 \alpha)^3 \max_{j=k-2}^{k+2}\xi_j^2=  \mathcal{O}\left(\frac{s_N}{N}\log N \right) .
\end{align}
Ultimately, we obtain our claim \eqref{eq:aimm},
$$
\sqrt{\frac{N}{s_N}}\sum_{k\in \mathcal{I}_{\widetilde{d}_N}} \left( \nabla G(\xi) \right)_k C_{\beta,\alpha, t} = \mathcal{O}\left(\sqrt{\frac{N}{s_N}} \widetilde{d}_N \frac{s_N}{N}\log N \right) = \mathcal{O}\left( \sqrt{\frac{s_N}{N}}(\log N)^2 \right),
$$
converging to 0 for $s_N = o\left(\frac{N}{(\log N)^4} \right)$.

In the case where $s_N < 2d_N = 2 c_9 \log N$, the statement of Lemma \ref{lem:G'=o_LT} is easily proved. Indeed, for every $x \in B_{4\sqrt{2d_N}}$, recall the intermediate value $\xi = \sqrt{\frac{s_N}{N}} (C^*)^{\frac{1}{2}}x + \lambda \sqrt{\frac{s_N}{N}}C^*t$. By the computations in \eqref{eq:boundCstern2} and \eqref{eq:boundCstern} we obtain that for every $k=1,\ldots, s_N$,
\begin{flalign*}
    |\xi_k| \le \sqrt{\frac{s_N}{N}} 4 \| (C^*)^{\frac{1}{2}}\|\sqrt{2d_N} + \sqrt{\frac{s_N}{N}}C_{\beta,\alpha,t} = \mathcal{O}\left( \sqrt{\frac{s_N d_N}{N}}\right) \longrightarrow 0,
\end{flalign*}
as $N \to \infty$. Therefore we can apply the Taylor expansion from \eqref{eq:Taylor_nablaG} yielding, similarly to \eqref{eq:Taylor_nablaG2}, that $|(\nabla G(\xi))_k| = \mathcal{O}\left(\frac{s_Nd_N}{N} \right)$ and ultimately
$$
\sqrt{\frac{N}{s_N}} \sum_{k=1}^{s_N} (\nabla G(\xi))_kC_{\beta,\alpha, t} = \mathcal{O}\left( \sqrt{\frac{s_N}{N}}s_N\log N\right) = \mathcal{O}\left( \frac{(\log N)^{\frac{5}{2}}}{\sqrt{N}}\right) =o(1).
$$
This finishes the proof.
\end{proof}

Lemma \ref{lemma:D_N-restriction_LT} and Lemma \ref{lem:G'=o_LT} imply for \eqref{eq:J(t)/J(0)} the asymptotic equivalence
\begin{align}\label{eq:J(t)/J(0)2}
\frac{J_N(t)}{J_N(0)} \sim 
\exp\left\{ \frac{1}{2}t^T\Cstern t\right\}=\E[e^{tY}],
\end{align}
where $Y \sim \mathcal{N}(0,\Cstern )$.

Let us assume for the moment that $K_N$ and $L_N$ do not contribute in the limit, which we will verify in the upcoming subsections. Then, \eqref{eq:J(t)/J(0)2}, \eqref{eq:MGF_JKL} and removing the Gaussian addition imply that the moment generating function converges
$$
\E\left[\exp\left\{\sqrt{\frac{N}{s_N}}t^T(m-\overrightarrow m^*) \right\} \, \middle\vert m \in \left[m^* - \delta, m^* + \delta \right]^{s_N} \right] \longrightarrow \E\left[e^{t^T\widetilde{Y}}\right],
$$
where $\widetilde{Y}\sim \mathcal{N}(0, \Sigma^* )$ with $\Sigma^* $ being the finite projections limit of $
\Cstern- A^{-1}$. Recalling the definition 
$$
\Cstern = \left( \Hess \varphi_N\left( \overrightarrow m^*\right)\right)^{-1} = \left(A-A^2(1-(m^*)^2) \right)^{-1}
$$
we can compute that
$$
\Cstern - A^{-1} = \left(A-A^2(1-(m^*)^2)\right)^{-1} - A^{-1} = (1- (m^*)^2)\left( I-A(1-(m^*)^2) \right)^{-1}.
$$
In Lemma \ref{lem:matrix_decay} we computed the explicit matrix coefficients of $(I-A)^{-1}$ in the high temperature regime. Since the matrix $\left( I-A(1-(m^*)^2) \right)^{-1}$ has the same structure, it follows immediately that 
\begin{flalign*}
    &(1- (m^*)^2)\left(\left( I-A(1-(m^*)^2) \right)^{-1}\right)_{j,k} \\
    =&  (1- (m^*)^2) \frac{\kappa_5^{|j-k|}+ \kappa_5^{s_N-|j-k|}  }{ (1-\kappa_5^{s_N}) \sqrt{ (1-\beta(1-(m^*)^2))^2 - 4\alpha^2(1-(m^*)^2)^2  } },
\end{flalign*}
where
$$
\kappa_5 = \frac{ (1-\beta(1-(m^*)^2))- \sqrt{ (1-\beta(1-(m^*)^2) ) - 4\alpha^2(1-(m^*)^2)^2 }  }{2\alpha(1-(m^*)^2)}.
$$
In particular, in the case $s_N \to \infty$,
\begin{align}\label{eq:Sigma*}
\Sigma^*_{j,k} =  \frac{(1-(m^*)^2)\kappa_5^{|j-k|}}{\sqrt{ (1-\beta(1-(m^*)^2))^2 - 4\alpha^2(1-(m^*)^2)^2  }}.
\end{align}

This would finish the proof of Theorem \ref{CLT}\,(ii) and it remains to prove that the $K_N$- and $L_N$-terms do not contribute in the limit. 

\subsection{The negligible term K}\label{subsec:K}
To control the term $K_N$ defined in \eqref{eq:K}, let us start by finding an upper bound for $I_N^c$. 

\begin{lemma}\label{lem:K}
For all $0<\delta<2m^*$ exist some $\varepsilon(\delta),r(\delta)>0$ such that
\begin{align*}
I_N^c(x, \overrightarrow m^*, \delta) \le s_N \exp\left\{ -\frac{N}{s_N} \eps\right\}\exp\left\{ \frac{N}{s_N}\sum_{k=1}^{s_N} \log\cosh(x^TAe_k) \right\}
\end{align*}
for all $x\in[m^*-r, m^*+r]^{s_N} $ and consequently, for all $t \in \R^{s_N}$,
\begin{align*}
K_N(t,r, \overrightarrow m^*, \delta) \le s_N \exp\left\{ -\frac{N}{s_N} \eps\right\}J_N(t,r,\overrightarrow m^*).
\end{align*}
\end{lemma}

\begin{proof}
Recall that $I_N^c(t)$ is a sum over $ m\in \left( [m^*-\delta,m^*+\delta]^{s_N} \right)^c $ and clearly
$$
\left( [m^*-\delta,m^*+\delta]^{s_N} \right)^c = \bigcup_{k=1}^{s_N}  \{ |m_k - m^*| > \delta\}.
$$
For any fixed $k$, we use \eqref{eq:I} and \eqref{eq:Ic} to write
$$
    \sum_{\substack{\sigma\in \{-1,1\}^N\\ m_{k}> m^* + \delta }} \frac{1}{2^N}\exp\left\{ \frac{N}{s_N} x^TAm \right\} =(I_N(x, \overrightarrow m^*, \delta )+I_N^c(x, \overrightarrow m^* , \delta)) \widetilde{\P}(m_{k} > m^* + \delta),
$$
where $\widetilde{\P}$ is the probability measure associated to the weights in the sum on the left hand side. From a Chernoff bound, we obtain for all $\gamma >0$
\begin{flalign}\label{DeltaAway}
    &\sum_{\substack{\sigma\in \{-1,1\}^N\\ m_{k}> m^* + \delta }} \frac{1}{2^N}\exp\left\{ \frac{N}{s_N} x^TAm \right\} \nonumber \\
    \le& (I_N(x)+I_N^c(x)) \exp\left\{ -\frac{N}{s_N}\gamma (\beta + 2\alpha)(m^* + \delta) \right\} \widetilde{\E}\left[\exp\left\{ \gamma(\beta + 2\alpha)\sum_{i\in S_{k}}\sigma_i \right\} \right] \nonumber\\
    =& \exp\left\{ -\frac{N}{s_N}\gamma (\beta + 2\alpha)(m^* + \delta) \right\} \sum_{\sigma_1, \ldots, \sigma_N \in \{-1,1\}} \frac{1}{2^N} \prod_{i \in S_{k}} \exp\left\{ (\gamma(\beta+ 2\alpha) + x^TAe_{k} )\sigma_i\right\} \nonumber\\
    & \quad \quad \quad \times \prod_{\substack{l=1 \\ l\neq k}}^{s_N} \prod_{i \in S_l} \exp\left\{ x^TAe_l \sigma_i \right\} \nonumber \\
    =& \exp\left\{ -\frac{N}{s_N}\gamma (\beta + 2\alpha)(m^* + \delta) \right\}  \exp \left\{ \frac{N}{s_N}\log\cosh\left(\gamma(\beta + 2\alpha) + x^TAe_{k}\right)\right\} \\
    & \quad \quad \quad  \times \exp\left\{ \frac{N}{s_N}\sum_{\substack{l=1 \\ l\neq k}}^{s_N} \log\cosh\left(x^TAe_l\right)\right\}. \nonumber
\end{flalign}
Next, we use a Taylor expansion of 
$$
\log\cosh \left( \gamma(\beta + 2\alpha) + x^TAe_{k} \right) = \log\cosh\left( (\overrightarrow \gamma + x)^TAe_K \right)
$$ 
around $x$, and we obtain
\begin{flalign*}
    &\log\cosh\left(\gamma(\beta + 2\alpha) + x^TAe_{k}\right) \\
    =& \log\cosh\left( x^TAe_{k}\right) + \gamma \frac{\partial}{\partial y_{k}}\log\cosh\left( y^TAe_{k}\right)\Big\vert_{y=\xi} +  \gamma \frac{\partial}{\partial y_{k-1}}\log\cosh\left( y^TAe_{k}\right)\Big\vert_{y=\xi} \\
    &\quad \quad \quad+  \gamma \frac{\partial}{\partial y_{k+1}}\log\cosh\left( y^TAe_{k}\right)\Big\vert_{y=\xi} \\
    =& \log\cosh\left( x^TAe_{k}\right) + \gamma(\beta + 2\alpha) \tanh\left(\xi^TAe_{k} \right),
\end{flalign*}
where $\xi$ is an intermediate value between $x$ and $x+\overrightarrow \gamma$.
Then we apply another Taylor expansion of $\tanh\left( \xi^TAe_k \right) = \tanh(\beta \xi_k + \alpha(\xi_{k-1} + \xi_{k+1}))$ around $(m^*,m^*,m^*)$ for $(\xi_{k-1},\xi_{k},\xi_{k+1})$ (recall that $|\xi_k - m^*|\le \gamma + |x_k - m^*| \le \gamma + r $) yielding 
$$
\tanh\left(x^TAe_{k}\right) = \tanh((\beta + 2\alpha)m^*)+ (\beta + 2\alpha)r\sech^2(\xi^TAe_{k})
$$
for an intermediate value $\xi$ but since $\sech \le 1$, we can choose $r$ small enough such that $(\beta + 2\alpha) r\sech^2\left(\xi^T A e_{2 k-1}\right) \le \frac{\delta}{2}$ and using the definition of $m^*$, we obtain
$$
\log\cosh\left(\gamma(\beta + 2\alpha) + x^TAe_{k}\right) \le \log\cosh\left( x^TAe_{k}\right) + \gamma(\beta + 2\alpha)m^* + \gamma\frac{\delta}{2}(\beta + 2\alpha).
$$
Plugging this into (\ref{DeltaAway}) finally yields
$$\sum_{\substack{\sigma\in \{-1,1\}^N\\ m_{k}> m^* + \delta }} \frac{1}{2^N}\exp\left\{ \frac{N}{s_N} x^TAm \right\} \le \exp \left\{ -\frac{N}{s_N}\gamma(\beta + 2\alpha)\frac{\delta}{2} \right\}\exp \left\{ \frac{N}{s_N}\sum_{l=1}^{s_N} \log\cosh\left(x^TAe_l\right) \right\}.
$$
Similarly, we obtain the same upper bound for $\{m_{k} < m^* - \delta\}$. Hence, by setting $\eps = \gamma(\beta + 2\alpha)\frac{\delta}{2}>0$, we arrive at
$$
I_N^c(x, \overrightarrow m^*, \delta) \le s_N \exp\left\{ -\frac{N}{s_N} \eps\right\}\exp\left\{ \frac{N}{s_N}\sum_{k=1}^{s_N} \log\cosh(x^TAe_k) \right\}
$$
and plugging this into $K_N(t)$ and applying a change of variables yields
\begin{align*}
    K_N(t) \le& s_N \exp\left\{ -\frac{N}{s_N} \eps\right\} \left( \frac{N}{s_N}\right)^{\frac{s_N}{2}}\\
    &\times\int_{[m^*-r,m^*+r]^{s_N}} \exp \left\{ -\frac{N}{s_N} \left(\varphi_N(x) - \varphi_N(\overrightarrow m^*)\right) + \sqrt{\frac{N}{s_N}}t^T(x-\overrightarrow m^*) \right\} \d^{s_N}x \\
    =&s_N \exp\left\{ -\frac{N}{s_N} \eps\right\}J_N(t).\qedhere
\end{align*}
\end{proof}
It follows from \eqref{eq:J(t)/J(0)2} that $\frac{J_N(t)}{J_N(0)}$ is bounded as $N\to \infty$, hence $\frac{K_N(t)}{J_N(0)} = \mathcal{O}\big(s_N \exp\{-\eps \frac{N}{s_N}\}\big)$ is negligible. 

\subsection{The negligible term L}\label{subsec:L}
At last, let us take care of the $L_N(t)$ term as defined in \eqref{eq:L}. We want to show that it is negligible in the limit in the sense that $\frac{L_N(t)}{J_N(0)} \to 0$. To that end, let us make the following observation.
\begin{lemma}\label{lem:L1}
For $0<\delta <2m^*$ and $ 0< r < m^*$ it holds
   \begin{align}
    &L_N(t,r,\overrightarrow m^*, \delta )\nonumber \\
    \le &K_N(t,r,-\overrightarrow m^*, (2m^* - \delta )) 
    + \left( \frac{N}{s_N}\right)^{\frac{s_N}{2}} \exp \left\{\frac{N}{s_N}\varphi_N(\overrightarrow m^*) - \sqrt{\frac{N}{s_N}}t^T\overrightarrow m^* \right\} \nonumber\\
    &\quad \times \int_{\R^{s_N}\setminus \left( \left[m^* - r,  m^* +r\right]^{s_N} \cup \left[ -m^* - r,  -m^* +r\right]^{s_N}\right) }\exp \left\{-\frac{N}{s_N} \varphi_N(x) + \sqrt{\frac{N}{s_N}}t^Tx \right\} \d^{s_N}x. \label{eq:LN_int}
\end{align}
\end{lemma}
\begin{proof}
Recall that we defined $L_N$ in \eqref{eq:L} to be an integration over
\begin{flalign}\label{L_Nbound1}
   &\left(\left[ m^* - r,  m^* +r\right]^{s_N}\right)^c \nonumber \\
   = & \R^{s_N}\setminus \left( \left[m^* - r,  m^* +r\right]^{s_N} \cup \left[ -m^* - r,  -m^* +r\right]^{s_N}\right) \cup \left[ -m^* - r,  -m^* +r\right]^{s_N},
\end{flalign}
where we separated the area $ [ -m^* - r,  -m^* +r ]^{s_N}$ eventually leading to $K_N(t,r,-\overrightarrow m^*, (2m^* - \delta )) $.
The integrant of $L_N$ included $I_N$, which was defined in \eqref{eq:I} as a sum over spins within
$$
\left[ m^* - \delta,  m^* +\delta\right]^{s_N} \subset \left(\left[ -m^* - (2m^* - \delta), -m^*+ (2m^* -\delta)\right]^{s_N}\right)^c,
$$ 
where we used $\delta  < 2m^* $. By the definition \eqref{eq:Ic} of $I_N^c$, it follows 
\begin{equation}\label{L_Nbound2}
    I_N(x, \overrightarrow m^*, \delta ) \le I_N^c(x, - \overrightarrow m^*, 2m^* - \delta),
\end{equation}
which leads to $K_N$ as in \eqref{eq:K}. We bound the remaining integration by
\begin{equation}\label{L_Nbound3}
    I_N \le I_N + I_{N}^c = \exp \left\{\frac{N}{s_N}\sum_{k=1}^{s_N}\log \cosh\left( x^TAe_k\right) \right\},
\end{equation}
where we used \eqref{eq:I+Ic}. This $\log\cosh$-term completes the function $\varphi_N$, and the claim follows from the observations (\ref{L_Nbound1}), (\ref{L_Nbound2}) and (\ref{L_Nbound3}).

\end{proof}
The first term of Lemma \ref{lem:L1} is $K_N$, which we controlled in Lemma~\ref{lem:K}.
The integration area of the second term in \eqref{eq:LN_int} is bounded away from both of the global minimizers, hence we expect its contribution to be negligible in the limit. 
In the high temperature regime we neglected all integration areas that are "too far" from the minimum at the origin, using the subgaussian behavior of the density function. In the low temperature regime, clearly, the integration in \eqref{eq:LN_int} fails to have this behavior "close" to the origin. Far from the origin we can again make use of the fact that our distribution has subgaussian tails. 
\begin{lemma}\label{lem:subgaussianLT}
  Let $\lambda_{\min}, \lambda_{\max}>0$ be the smallest and largest eigenvalue of the matrix $A$. Let $R>2\frac{\lambda_{\max}}{\lambda_{\min}}$ such that $c_{11}= \frac{\lambda_{\min}}{2} - \frac{\lambda_{\max}}{R}>0$, then for any $\|x\|>{R\sqrt{N}}$ it holds
   $$
   \frac{N}{s_N}\varphi_N\left( \sqrt{\frac{s_N}{N}}x \right) \ge c_{11}\|x\|_2^2.
   $$
\end{lemma}
\begin{proof}
A simple calculation using that $\log \cosh(z) \le |z|$ and $\|x\|_1 \le \sqrt{s_N}\|x\|_2$ shows
\begin{align*}
    \frac{N}{s_N}\varphi_N\left( \sqrt{\frac{s_N}{N}}x \right) &= \frac{1}{2}x^TAx - \frac{N}{s_N}\sum_{k=1}^{s_N} \log \cosh\left( \sqrt{\frac{s_N}{N}}x^TAe_k \right) \ge \frac{1}{2}x^TAx - \frac{N}{s_N}\sum_{k=1}^{s_N} \sqrt{\frac{s_N}{N}} | x^TAe_k | \\
    &= \frac{1}{2}x^TAx - \sqrt{\frac{N}{s_N}} \| Ax \|_1 \ge  \frac{1}{2}x^TAx - \sqrt{N} \| Ax \|_2 \ge  \frac{1}{2}x^TAx - \frac{1}{R}\|x\|_2\| Ax \|_2 \\
    &\ge  \left(\frac{\lambda_{\min}}{2} - \frac{\lambda_{\max}}{R} \right)\|x\|_2^2.\qedhere
    \end{align*}
\end{proof}
 
Under our assumption $s_N = o\left( \frac{\sqrt{N}}{\log N}\right)$, the following crude bound will be sufficient.
\begin{lemma}\label{lem:LJ}
If $s_N = o\left( \frac{\sqrt{N}}{\log N}\right)$ and for $0<\delta <2m^*$ and $ 0< r < m^*$ it holds
\begin{align*}
\frac{L_N(t,r,\overrightarrow m^*, \delta )}{J_N(0,r,\overrightarrow m^*)}\le e^{\mathcal O(\sqrt N)} \int_{\Delta_N^r}  \exp \left\{-\frac{N}{s_N} \left(\varphi_N\left( \sqrt{\frac{s_N}{N}}x\right) - \varphi_N(\overrightarrow m^*) \right) \right\} \d^{s_N}x  +o(1) ,
\end{align*}
where 
   $$
   \Delta_N^r = B_{R\sqrt{N}} \setminus \left( \left[\sqrt{\frac{N}{s_N}}( m^*-r),\sqrt{\frac{N}{s_N}}( m^*+r)  \right]^{s_N} \cup \left[\sqrt{\frac{N}{s_N}}(- m^*-r),\sqrt{\frac{N}{s_N}}(- m^*+r)  \right]^{s_N}  \right)
   $$
\end{lemma}
\begin{proof}
First, note that Corollary \ref{cor:J(0)not0} gives $J_N(0)^{-1}=e^{\mathcal O(s_N)}=e^{\mathcal O(\sqrt N)}$.

Our goal is to bound the terms of Lemma \ref{lem:L1}, where the first term $K_N$ has been bounded in Lemma~\ref{lem:K}. By a change of variables for the integral \eqref{eq:LN_int}, it remains to bound
\begin{align*} 
\int_{ \Delta_N^r\cup B_{R\sqrt{N}} ^c} \exp\left\{ - \frac{N}{s_N}\left( \varphi_N\left(\sqrt{\frac{s_N}{N}}x \right) - \varphi_N(\overrightarrow m^*) \right) + t^T\bigg(x -\sqrt{\frac{N}{s_N}} \overrightarrow m^* \bigg)  \right\} \d^{s_N}x.
\end{align*}
The integration far away from the origin is negligible due to Lemma \ref{lem:subgaussianLT}, Cauchy-Schwarz and a Gaussian tail bound
   \begin{flalign*}
       &\int_{B_{R\sqrt{N}}^c} \exp\left\{ - \frac{N}{s_N}\left( \varphi_N\left(\sqrt{\frac{s_N}{N}}x \right) - \varphi_N(\overrightarrow m^*) \right) + t^T\left(x -\sqrt{\frac{N}{s_N}} \overrightarrow m^* \right)  \right\} \d x \\
        \le&  \exp\left\{ \frac{N}{s_N}\varphi_N(\overrightarrow m^*) + \sqrt{ {N} }\|t\|m^* \right\} \int_{B_{R\sqrt{N}}^c} \exp\left\{ - c_{11}\|x\|^2 + \|t\| \|x\| \right\}\d x \\
        \le& \exp\left\{ N\left(\frac{\beta + 2\alpha}{2}(m^*)^2 - \log\cosh((\beta + 2 \alpha)m^*) + o(1)\right)\right\} \exp\left\{2c_{11} R^2 N \right\} (2\pi 2c_{11})^{\frac{s_N}{2}}   =o(1),
   \end{flalign*} 
      since $\frac{\beta + 2\alpha}{2}(m^*)^2 - \log\cosh((\beta + 2 \alpha)m^*)<0$. On the other hand, on $\Delta_N^r$, we use a similar bound $t^T (x -\sqrt{\frac{N}{s_N}} \overrightarrow m^*)=\mathcal O(\sqrt N)$, yielding the claim.
\end{proof}

As a last step, we need the following uniform energy gap.

\begin{lemma}\label{lem:minimal}
    If $s_N = o\left( \frac{\sqrt{N}}{\log N}\right)$ and $0<r < m^*$, then there exists $\rho>0$ (that depends only on $r$ and not on $N$) such that  
    \begin{flalign*}
         \inf_{x\in \Delta_N^r}  \varphi_N\left( \sqrt{\frac{s_N}{N}}x\right) \ge \varphi_N(\overrightarrow m^*) +\rho .
    \end{flalign*}
\end{lemma}

With this lemma at hand, the proof of $L_N(t)/J_N(0)=o(1)$ is complete. Indeed, Lemma \ref{lem:LJ} and Lemma \ref{lem:minimal} imply
\begin{align}
    \frac{L_N(t,r,\overrightarrow m^*, \delta )}{J_N(0,r,\overrightarrow m^*)}\le e^{-\rho N/s_N+\mathcal O(\sqrt N)} +o(1)\to 0.
\end{align}
  
     Note that this is the only step in the proof of Theorem \ref{CLT}\,(ii) using the assumption $s_N = o\left( \frac{\sqrt{N}}{\log N}\right)$. We believe that this assumption can be weakened with a much more tightened bound, but for convenience of the reader of page 42, we believe that a less technical argument under this stronger assumption still conveys the same message of the result.

\begin{proof}[Proof of Lemma \ref{lem:minimal}]
In order to find the minimum of $\varphi_N $ on $\Delta_N^r$, let us recall the following properties that we proved in Section 4:
\begin{itemize}
    \item The only critical points of $\varphi_N$ having either all coordinates strictly positive or all coordinates strictly negative, are the vectors $\pm \overrightarrow m^*$, see the proof of Lemma \ref{LowTemp}.
    \item The only critical point that has a zero coordinate, is the zero vector and it is not a minimum according to Lemma \ref{definiteness}.
    \item For any point $x$ that has only positive or only negative coordinates and any subset $J \subset \{1, \ldots, s_N\}$ with $|J| \le \frac{s_N}{2}$, if we define $y$ having flipped coordinates
    \begin{flalign}\label{eq:flip}
        &y_k = x_k \quad \text{for all } k \notin J \nonumber \\
       & y_k = - x_k \quad \text{for all } k \in J ,
    \end{flalign}
    then the proof of Lemma \ref{SameSign} showed the strict energy gap 
\begin{align}\label{eq:energygapflip}
    \varphi_N(x) - \varphi_N(y) \le - 2\beta \sum_{j\in J}|x_j|.
\end{align}
\end{itemize} 
These observations together imply that $\inf_{x\in \Delta_N^r}  \varphi_N\left( \sqrt{\frac{s_N}{N}}x\right)$  must be attained on the boundary  $\partial\Delta_N^r$ or in the interior of flipped copies of boxes
\begin{align}\label{eq:flippedbox}
    \prod_{k=1}^{s_N} \left[\sqrt{\frac{N}{s_N}}( \eps_k m^*-r),\sqrt{\frac{N}{s_N}}( \eps_k m^*+r)  \right],
\end{align}
    for all choices of $\eps_1, \ldots, \eps_{s_N} \in \{-1,+1\}$.

Clearly, the minimum cannot be attained at $\partial B_{R\sqrt N}^c$ by Lemma \ref{lem:subgaussianLT}. On the boundary $\partial\left[\sqrt{\frac{N}{s_N}}( m^*-r),\sqrt{\frac{N}{s_N}}( m^*+r)  \right]^{s_N}$, we use the strict convexity in its interior. Indeed, we have shown in Lemma \ref{lem:log-concaveLT} that $\Hess\frac{N}{s_N} \left( \varphi_N\left(\overrightarrow m^* + \sqrt{\frac{s_N}{N}}x \right) - \varphi_N( \overrightarrow m^*)\right)$ is uniformly strictly positive definite for $x\in \left[-\sqrt{\frac{N}{s_N}}r, \sqrt{\frac{N}{s_N}}r \right]^{s_N}$. For the smallest singular value $\lambda_{\min}>0$  of the Hessian, it follows
\begin{align}
         \inf_{x\in \partial\left[\sqrt{\frac{N}{s_N}}( m^*-r),\sqrt{\frac{N}{s_N}}( m^*+r)  \right]^{s_N}}  \varphi_N\left( \sqrt{\frac{s_N}{N}}x\right) -\varphi_N(\overrightarrow m^*) \ge \lambda_{\min}r^2 .
\end{align}
    
It remains to control possible critical points within flipped boxes \eqref{eq:flippedbox}. 
For any $y$ as in \eqref{eq:flip} in such a flipped box, there is $x \in \left[ \sqrt{\frac{N}{s_N}}(m^* - r), \sqrt{\frac{N}{s_N}} (m^*+r)\right]^{s_N}$ and $J \subset \{1,\ldots, s_N\}$ such that 
\eqref{eq:energygapflip} gives
    $$
    \varphi_N\left( \sqrt{\frac{s_N}{N}}x\right) - \varphi_N\left( \sqrt{\frac{s_N}{N}}y\right) \le - \sum_{j\in J}2\beta \sqrt{\frac{s_N}{N}} x_j \le -  2\beta|J|(m^*-r)\le -  2\beta (m^*-r)<0.
    $$
Consequently,
    \begin{align*}
    \varphi_N\left( \sqrt{\frac{s_N}{N}}y\right) - \varphi_N(\overrightarrow m^*)   \ge  2 \beta (m^* - r) + \varphi_N\left( \sqrt{\frac{s_N}{N}}x\right) - \varphi_N(\overrightarrow m^*)  \ge   2 \beta (m^* - r)   .
    \end{align*}
The claim follows for $\rho= \min\{\lambda_{\min}r^2, 2\beta(m^* - r)\}$.
\end{proof}

\appendix
\section{Proof of Proposition \ref{Ising}}
The statement of Proposition \ref{Ising} follows from Theorem 8.39 in \cite{Georgii_book}. 
Fix $0<c<1$ such that $\frac{1}{c}\in \N$ and assume that $s_N=cN \in \N$ for all $N$. Then each block contains $\frac{1}{c}$ many spins and $m_k$ can take the values in $\mathcal{A}_c = \left\{-1 + 2cl \, : \, l = 0, \cdots, \frac{1}{c}\right\}$, where for every $l$,
$$
\P(m_k = 1- 2lc) = \P\left(m_k = -(1-2lc)\right) = \P\left(\{\sigma_1, \ldots, \sigma_{\frac{1}{c}}\}\text{ contains $l$ many -1-spins}\right) =\frac{{\frac{1}{c} +l-1 \choose l}}{2^\frac{1}{c}}.
$$
Then, $\mu_{N, \beta,\alpha}(m \in \cdot)$ behaves like the distribution of the magnetization under the Gibbs measure of a 
one-dimensional Ising model with nearest-neighbor interacting spins $m_1, \ldots, m_{cN}$ and a-priori distribution given by the above weight. In the language of \cite{Georgii_book} Gibbs measures can be described in terms of potentials, which are the building blocks of 
the energy function (see \cite[Section 2.1]{Georgii_book}). For our system, the potential is given by $\Phi = (\Phi_E)_{E\subset \N}$, where 
$$\Phi_E (\sigma) = \begin{cases}
  \frac{\alpha}{c}m_im_{i+1} + \frac{\beta}{2c}m_i^2 & \text{if } E=\{i, i+1\}\\    
  0 & \text{otherwise}.
\end{cases}
$$
According to Theorem 8.39 in \cite{Georgii_book}, the spin model does not exhibit a phase transition if
\begin{equation}\label{GeorgiiCondition}
    \sup_{i \in \{1, \ldots, s_N\}} \left(\sup_{x,y} |\Phi_{\{i-1,i\}}(x) - \Phi_{\{i-1,i\}}(y)| + \sup_{x,y} |\Phi_{\{i,i+1\}}(x) - \Phi_{\{i,i+1\}}(y)|\right) < \infty.
\end{equation}
Note that for any $i\in \{1, \ldots, s_N\}$ and any $x,y \in \mathcal{A}_c$
$$
|\Phi_{\{i-1,i\}}(x) - \Phi_{\{i-1,i\}}(y)| \le 2 \frac{\alpha}{c} + \frac{\beta}{2c},
$$
hence, the left hand side in (\ref{GeorgiiCondition}) is at most $2 \frac{\alpha}{c} + \frac{\beta}{c} < \infty$. Therefore, the potential $\Phi$ admits a unique infinite volume Gibbs measure $\mu_m$. The weak convergence in product topology is defined as the convergence of convergence of all finite-dimensional marginals. This follows from the DLR equations for $\mu_m$.  This is our first assertion.

For the last statement note that Georgii (see Chapter 7 in \cite{Georgii_book}, in particular Section 7.3) shows that in the setting of Ising type interactions on compact spaces the set of all Gibbs measures $\mathcal(\Phi)$ is a (Choquet) simplex with the ergodic measures as its extreme points. Hence, uniqueness implies ergodicity of $\mu_m$. By the ergodic theorem, the total magnetization
$\frac 1N \sum_{i=1}^N \sigma_i = \frac{1}{s_N}\sum_{k=1}^{s_N}m_k$ converges almost surely to a constant under $\mu_m$. Since our system is symmetric under spin-flip, i.e.\ 
$\Phi_E (\sigma)=\Phi_E (-\sigma)$ for all $\sigma$ and $E$, this limit can only be 0.
\qed
\section*{Acknowledgements} The authors thank Thomas Lebl\'e for valuable suggestions and stimulating discussions. IL and ML were supported by the German Research Foundation under Germany’s Excellence Strategy EXC 2044 - 390685587, Mathematics M\"unster: Dynamics - Geometry - Structure. JJ was supported by the DFG priority program SPP 2265 Random Geometric Systems.


\begin{thebibliography}{KLSS20}
	
	\bibitem[BD01]{BrockDurlauf}
	William~A. Brock and Steven~N. Durlauf.
	\newblock Discrete choice with social interactions.
	\newblock {\em Rev. Econom. Stud.}, 68(2):235--260, 2001.
	
	\bibitem[BL00]{bobkov:ledoux:2000}
	S.G. Bobkov and M.~Ledoux.
	\newblock From brunn-minkowski to brascamp-lieb and to logarithmic sobolev
	inequalities.
	\newblock {\em Geometric and Functional Analysis}, 10:1028--1052, 2000.
	
	\bibitem[BM17]{basak_muk}
	Anirban Basak and Sumit Mukherjee.
	\newblock Universality of the mean-field for the potts model.
	\newblock {\em Probability Theory and Related Fields}, 168(3):557--600, 2017.
	
	\bibitem[BR07]{BeRa07}
	Michele Benzi and NADER RAZOUK.
	\newblock Decay bounds and o(n) algorithms for approximating functions of
	sparse matrices.
	\newblock {\em ETNA. Electronic Transactions on Numerical Analysis [electronic
		only]}, 28, 01 2007.
	
	\bibitem[BRS19]{BRS19}
	Quentin Berthet, Philippe Rigollet, and Piyush Srivastava.
	\newblock Exact recovery in the {I}sing blockmodel.
	\newblock {\em Ann. Statist.}, 47(4):1805--1834, 2019.
	
	\bibitem[CL10]{CoLo10}
	Rama Cont and Matthias L\"{o}we.
	\newblock Social distance, heterogeneity and social interactions.
	\newblock {\em J. Math. Econom.}, 46(4):572--590, 2010.
	
	\bibitem[Col14]{Col14}
	Francesca Collet.
	\newblock Macroscopic limit of a bipartite {C}urie-{W}eiss model: a dynamical
	approach.
	\newblock {\em J. Stat. Phys.}, 157(6):1301--1319, 2014.
	
	\bibitem[CS11]{Chatterjee_Shao}
	Sourav Chatterjee and Qi-Man Shao.
	\newblock Nonnormal approximation by {S}tein's method of exchangeable pairs
	with application to the {C}urie-{W}eiss model.
	\newblock {\em Ann. Appl. Probab.}, 21(2):464--483, 2011.
	
	\bibitem[DM23]{deb_mukherjee}
	Nabarun Deb and Sumit Mukherjee.
	\newblock Fluctuations in mean-field ising models.
	\newblock {\em The Annals of Applied Probability}, 33(3):1961--2003, 2023.
	
	\bibitem[EL10]{EL10}
	Peter Eichelsbacher and Matthias L{\"o}we.
	\newblock Stein's method for dependent random variables occurring in
	statistical mechanics.
	\newblock {\em Electron. J. Probab.}, 15:no. 30, 962--988, 2010.
	
	\bibitem[Ell06]{EllisEntropyLargeDeviationsAndStatisticalMechanics}
	Richard~S. Ellis.
	\newblock {\em Entropy, large deviations, and statistical mechanics}.
	\newblock Classics in Mathematics. Springer-Verlag, Berlin, 2006.
	\newblock Reprint of the 1985 original.
	
	\bibitem[EN78a]{Ellis_Newman_78b}
	Richard~S. Ellis and Charles~M. Newman.
	\newblock Limit theorems for sums of dependent random variables occurring in
	statistical mechanics.
	\newblock {\em Z. Wahrsch. Verw. Gebiete}, 44(2):117--139, 1978.
	
	\bibitem[EN78b]{Ellis_Newman_78a}
	Richard~S. Ellis and Charles~M. Newman.
	\newblock The statistics of {C}urie-{W}eiss models.
	\newblock {\em J. Statist. Phys.}, 19(2):149--161, 1978.
	
	\bibitem[ENR80]{EllisNewman_80}
	Richard~S. Ellis, Charles~M. Newman, and Jay~S. Rosen.
	\newblock Limit theorems for sums of dependent random variables occurring in
	statistical mechanics. {II}. {C}onditioning, multiple phases, and
	metastability.
	\newblock {\em Z. Wahrsch. Verw. Gebiete}, 51(2):153--169, 1980.
	
	\bibitem[EW90]{EllisWang}
	Richard~S. Ellis and Kongming Wang.
	\newblock Limit theorems for the empirical vector of the
	{C}urie-{W}eiss-{P}otts model.
	\newblock {\em Stochastic Process. Appl.}, 35(1):59--79, 1990.
	
	\bibitem[FC11]{FC11}
	Micaela Fedele and Pierluigi Contucci.
	\newblock Scaling limits for multi-species statistical mechanics mean-field
	models.
	\newblock {\em J. Stat. Phys.}, 144(6):1186--1205, 2011.
	
	\bibitem[FKT22]{Fleermann+Kirsch+Toth:2022}
	M.~Fleermann, W.~Kirsch, and G.~Toth.
	\newblock Local central limit theorem for multi-group {C}urie-{W}eiss models.
	\newblock {\em J. Theoret. Probab.}, 35(3):2009--2019, 2022.
	
	\bibitem[FU12]{FU12}
	Micaela Fedele and Francesco Unguendoli.
	\newblock Rigorous results on the bipartite mean-field model.
	\newblock {\em J. Phys. A}, 45(38):385001, 18, 2012.
	
	\bibitem[FV18]{Velenik_book}
	S.~Friedli and Y.~Velenik.
	\newblock {\em Statistical mechanics of lattice systems}.
	\newblock Cambridge University Press, Cambridge, 2018.
	\newblock A concrete mathematical introduction.
	
	\bibitem[GBC09]{GBC09}
	Ignacio Gallo, Adriano Barra, and Pierluigi Contucci.
	\newblock Parameter evaluation of a simple mean-field model of social
	interaction.
	\newblock {\em Math. Models Methods Appl. Sci.}, 19(suppl.):1427--1439, 2009.
	
	\bibitem[GC08]{GC08}
	Ignacio Gallo and Pierluigi Contucci.
	\newblock Bipartite mean field spin systems. {E}xistence and solution.
	\newblock {\em Math. Phys. Electron. J.}, 14:Paper 1, 21, 2008.
	
	\bibitem[Geo11]{Georgii_book}
	Hans-Otto Georgii.
	\newblock {\em Gibbs measures and phase transitions}, volume~9 of {\em De
		Gruyter Studies in Mathematics}.
	\newblock Walter de Gruyter \& Co., Berlin, second edition, 2011.
	
	\bibitem[Isi25]{Ising25}
	Ernst Ising.
	\newblock Beitrag zur theorie des ferromagnetismus.
	\newblock {\em Zeitschrift für Physik}, 25, 02 1925.
	
	\bibitem[JLS21]{JLS}
	Jonas Jalowy, Matthias Löwe, and Holger Sambale.
	\newblock Fluctuations of the magnetization in the block potts model, 2021.
	
	\bibitem[Kim22]{kim2potts}
	Daecheol Kim.
	\newblock Energy landscape of the two-component curie--weiss--potts model with
	three spins.
	\newblock {\em Journal of Statistical Physics}, 188(2):15, 2022.
	
	\bibitem[KL09]{KnLo07}
	Holger Kn\"{o}pfel and Matthias L\"{o}we.
	\newblock Zur {M}einungsbildung in einer heterogenen {B}ev\"{o}lkerung---ein
	neuer {Z}ugang zum {H}opfield {M}odell.
	\newblock {\em Math. Semesterber.}, 56(1):15--38, 2009.
	
	\bibitem[KLSS20]{KLSS20}
	Holger Kn\"{o}pfel, Matthias L\"{o}we, Kristina Schubert, and Arthur Sinulis.
	\newblock Fluctuation results for general block spin {I}sing models.
	\newblock {\em J. Stat. Phys.}, 178(5):1175--1200, 2020.
	
	\bibitem[KT20]{KT20a}
	Werner Kirsch and Gabor Toth.
	\newblock Two groups in a {C}urie-{W}eiss model with heterogeneous coupling.
	\newblock {\em J. Theoret. Probab.}, 33(4):2001--2026, 2020.
	
	\bibitem[KT22]{KT22}
	Werner Kirsch and Gabor Toth.
	\newblock Limit theorems for multi-group curie–weiss models via the method of
	moments.
	\newblock {\em Mathematical Physics, Analysis and Geometry}, 25, 09 2022.
	
	\bibitem[LS18]{LSblock1}
	Matthias L\"owe and Kristina Schubert.
	\newblock Fluctuations for block spin {I}sing models.
	\newblock {\em Electron. Commun. Probab.}, 23:12 pp., 2018.
	
	\bibitem[LS20]{LS20}
	Matthias L\"{o}we and Kristina Schubert.
	\newblock Exact recovery in block spin {I}sing models at the critical line.
	\newblock {\em Electron. J. Stat.}, 14(1):1796--1815, 2020.
	
	\bibitem[LSV20]{LSV20}
	Matthias Löwe, Kristina Schubert, and Franck Vermet.
	\newblock Multi-group binary choice with social interaction and a random
	communication structure—a random graph approach.
	\newblock {\em Physica A: Statistical Mechanics and its Applications},
	556:124735, 05 2020.
	
	\bibitem[MSB21]{MSB21}
	Somabha Mukherjee, Jaesung Son, and Bhaswar~B. Bhattacharya.
	\newblock Fluctuations of the magnetization in the {$p$}-spin {C}urie-{W}eiss
	model.
	\newblock {\em Comm. Math. Phys.}, 387(2):681--728, 2021.
	
	\bibitem[MSB25]{MSB25}
	Somabha Mukherjee, Jaesung Son, and Bhaswar~B. Bhattacharya.
	\newblock Phase transitions of the maximum likelihood estimators in the
	{$p$}-spin {C}urie-{W}eiss model.
	\newblock {\em Bernoulli}, 31(2):1502--1526, 2025.
	
	\bibitem[Mun26]{munpotts}
	Kyunghoo Mun.
	\newblock Dynamical {P}hase {T}ransition for the homogeneous multi-component
	{C}urie-{W}eiss-{P}otts model.
	\newblock {\em J. Stat. Phys.}, 193(2):Paper No. 16, 2026.
	
	\bibitem[Pre09]{Presutti}
	Errico Presutti.
	\newblock {\em Scaling limits in statistical mechanics and microstructures in
		continuum mechanics}.
	\newblock Springer, 2009.
	
	\bibitem[Tho79]{Thompson}
	Colin~J. Thompson.
	\newblock {\em Mathematical statistical mechanics}.
	\newblock Princeton University Press, Princeton, NJ, 1979.
	\newblock Reprinting of the 1972 original.
	
	\bibitem[Ver18]{Vershynin}
	Roman Vershynin.
	\newblock {\em High-dimensional probability}, volume~47 of {\em Cambridge
		Series in Statistical and Probabilistic Mathematics}.
	\newblock Cambridge University Press, Cambridge, 2018.
	\newblock An introduction with applications in data science, With a foreword by
	Sara van de Geer.
	
\end{thebibliography}



\end{document}